\documentclass[journal]{IEEEtran}

\ifCLASSINFOpdf

\else

\fi

\usepackage[ruled]{algorithm2e}
\usepackage{algorithmic}

\usepackage[hidelinks]{hyperref}
\usepackage{graphicx}
\graphicspath{{figures/}}

\usepackage{amsmath}
\usepackage{cases}
\usepackage{amssymb}
\usepackage{amsfonts}
\usepackage{mathtools}

\usepackage{amsthm}

\theoremstyle{definition}
\newtheorem{lemma}{Lemma}
\newtheorem{theorem}{Theorem}
\newtheorem{proposition}{Proposition}

\theoremstyle{definition}
\newtheorem{definition}{Definition}
\newtheorem*{remark}{Remark}

\newcommand\norm[2][1]{\left\|#2\right\|_{#1}}
\newcommand\abs[1]{\left|#1\right|}
\newcommand\where[2]{\left.#1\right|_{#2}}
\newcommand\difffrac[3][1]{
\ifnum #1=1
\frac{\mathrm{d} #2}{\mathrm{d} #3}
\else
\frac{{\mathrm{d}}^{#1} #2}{\mathrm{d} #3^{#1}}
\fi
}

\newcommand\R{{\mathbb{R}}}
\newcommand\N{{\mathbb{N}}}

\newcommand{\Setminus}[2]{{\left.#1\middle\backslash #2\right.}}

\newcommand\hamilton{{\mathcal{H}}}

\newcommand\st{{\mathrm{s.t.}}}

\renewcommand\vector[1]{\boldsymbol{#1}}
\newcommand\ve{{\vector{e}}}

\newcommand\vp{{\vector{p}}}

\newcommand\vx{{\vector{x}}}
\newcommand\vy{{\vector{y}}}

\newcommand\vA{{\vector{A}}}
\newcommand\vB{{\vector{B}}}
\newcommand\vF{{\vector{F}}}
\newcommand\vM{{\vector{M}}}

\newcommand\vlambda{{\vector{\lambda}}}
\newcommand\veta{{\vector{\eta}}}

\newcommand\vtau{{\vector{\tau}}}
\newcommand\vzero{{\vector{0}}}

\newcommand\tf{{t_\mathrm{f}}}
\newcommand\f{{\mathrm{f}}}
\newcommand\vxf{{\vx_\mathrm{f}}}
\newcommand\sgn{{\mathrm{sgn}}}

\newcommand\const{{\mathrm{const}}}

\usepackage{xcolor}

\newcommand\RomanNum[1]{\uppercase\expandafter{\romannumeral #1}}

\begin{document}

\title{Chattering Phenomena in Time-Optimal Control for High-Order Chain-of-Integrator Systems with Full State Constraints (Extended Version)}

\author{Yunan~Wang,~
Chuxiong~Hu,~\IEEEmembership{Senior~Member,~IEEE,}
Zeyang~Li,
Yujie~Lin,
Shize~Lin,
and~Suqin~He

\thanks{Corresponding author: Chuxiong Hu (e-mail: cxhu@tsinghua.edu.cn).
}}

\markboth{}{}

\maketitle

\begin{abstract}
Time-optimal control for high-order chain-of-integrator systems with full state constraints remains an open and challenging problem within the discipline of optimal control. The behavior of optimal control in high-order problems lacks precise characterization, and even the existence of the chattering phenomenon, i.e., the control switches for infinitely many times over a finite period, remains unknown and overlooked. This paper establishes a theoretical framework for chattering phenomena in the considered problem, providing novel findings on the uniqueness of state constraints inducing chattering, the upper bound of switching times in an unconstrained arc during chattering, and the convergence of states and costates to the chattering limit point. For the first time, this paper proves the existence of the chattering phenomenon in the considered problem. The chattering optimal control for 4th-order problems with velocity constraints is precisely solved, providing an approach to plan time-optimal snap-limited trajectories. Other cases of order $n\leq4$ are proved not to allow chattering. The conclusions rectify a longstanding misconception in the industry concerning the time-optimality of S-shaped trajectories with minimal switching times.
\end{abstract}


\IEEEpeerreviewmaketitle

\section{Introduction}\label{sec:Introduction}

\IEEEPARstart{T}{ime-optimal} control for high-order chain-of-integrator systems with full state constraints is a classical problem within the discipline of optimal control and kinematics, yet to be resolved. With time-optimal orientations and safety constraints, control for high-order chain-of-integrator systems has achieved universal application in computer numerical control machining \cite{wang2021local,wang2025consistency}, robotic motion control \cite{wang2022learning,zhao2020pareto}, semiconductor device fabrication \cite{li2018convergence}, and autonomous driving \cite{guler2016adaptive}. However, the behavior of optimal control in this issue has yet to be thoroughly investigated. Specifically, the existence of the chattering phenomenon \cite{marchal1973chattering} remains undiscovered, let alone the complete analysis on optimal control. As summarized in \cite{caponigro2018regularization}, chattering refers to fast oscillations of controls, such as an infinite numbers of switching over a finite time interval in the control theory.

Formally, the investigated problem of order $n$ is described in \eqref{eq:optimalproblem}, where $\vx=\left(x_k\right)_{k=1}^n\in\R^n$ is the state vector, $u\in\R$ is the control, and the terminal time $\tf$ is free. $\vx_0=\left(x_{0,k}\right)_{k=1}^n$ and $\vx_\f=\left(x_{\f k}\right)_{k=1}^n$ are the assigned initial state vector and terminal state vector, respectively. $\vM=\left(M_k\right)_{k=0}^n\in\R_{++}\times\overline{\R}_{++}^{n}$, where $\overline{\R}_{++}=\R_{++}\cup\left\{\infty\right\}$ is the strictly positive part of the extended real number line. The notation $\left(\bullet\right)$ means $\left[\bullet\right]^\top$. Problem \eqref{eq:optimalproblem} possesses a clear physical significance. For instance, if $n=4$, $x_4$, $x_3$, $x_2$, $x_1$, and $u$ respectively refer to the position, velocity, acceleration, jerk, and snap of a 1-axis motion system, respectively. Problem \eqref{eq:optimalproblem} requires a trajectory with minimum motion time from a given initial state vector to a terminal state vector under box state constraints.
\begin{IEEEeqnarray}{rl}\label{eq:optimalproblem}
\min\quad& J=\int_{0}^{\tf}\mathrm{d}t=\tf,\label{eq:optimalproblem_objective}\IEEEyesnumber\IEEEyessubnumber*\\
\st\quad&\dot{x}_k\left(t\right)=x_{k-1}\left(t\right),\,\forall 1<k\leq n,\,t\in\left[0,t_\f\right],\label{eq:optimalproblem_dotxk}\\
&\dot{x}_1\left(t\right)=u\left(t\right),\,\forall t\in\left[0,t_\f\right],\label{eq:optimalproblem_dotx1}\\
&\vx\left(0\right)=\vx_0,\,\vx\left(\tf\right)=\vx_\f,\\
&\abs{x_k\left(t\right)}\leq M_k,\,\forall 1\leq k\leq n,\,t\in\left[0,t_\f\right],\label{eq:optimalproblem_x_constraint}\\
&\abs{u\left(t\right)}\leq M_0,\,\forall t\in\left[0,t_\f\right],
\end{IEEEeqnarray}

Numerous studies have been conducted on problem \eqref{eq:optimalproblem} from the perspectives of optimal control and model-based classification discourse. Problem \eqref{eq:optimalproblem} without state constraints, i.e., $\forall 1\leq k\leq n$, $M_k=\infty$, can be fully solved by Pontryagin's maximum principle (PMP) \cite{hartl1995survey}, where the analytic expression of the optimal control \cite{bartolini2002time} is well-known. Once state constraints are introduced, problem \eqref{eq:optimalproblem} becomes practically significant but challenging to solve. The 1st- and 2nd-order problems are well-known with simple solutions \cite{ma2021optimal}. Haschke et al. \cite{haschke2008line} solved the 3rd-order problem where $x_{\f2}=x_{\f3}=0$. Kr{\"o}ger \cite{kroger2011opening} developed the Reflexxes library, solving 3rd-order problems where $x_{\f3}=0$. Berscheid and Kr{\"o}ger \cite{berscheid2021jerk} fully solved 3rd-order problems without position constraints, i.e., $M_3=\infty$, resulting in the Ruckig library. Our previous work \cite{wang2024time} completely solved 3rd-order problems and fully enumerated the system behaviors for higher-order problems, except for the limit point of chattering. However, few existing methods can solve optimal solutions for 4th-order or higher-order problems with full state constraints and arbitrarily given boundary states, despite the universal application of snap-limited trajectories for lithography machines with time-optimal orientations \cite{li2015data}. Specifically, even the existence of chattering in problem \eqref{eq:optimalproblem} remains unsolved, let alone a comprehensive understanding of the optimal control.

\begin{figure*}[!t]
\centering
\includegraphics[width=\textwidth]{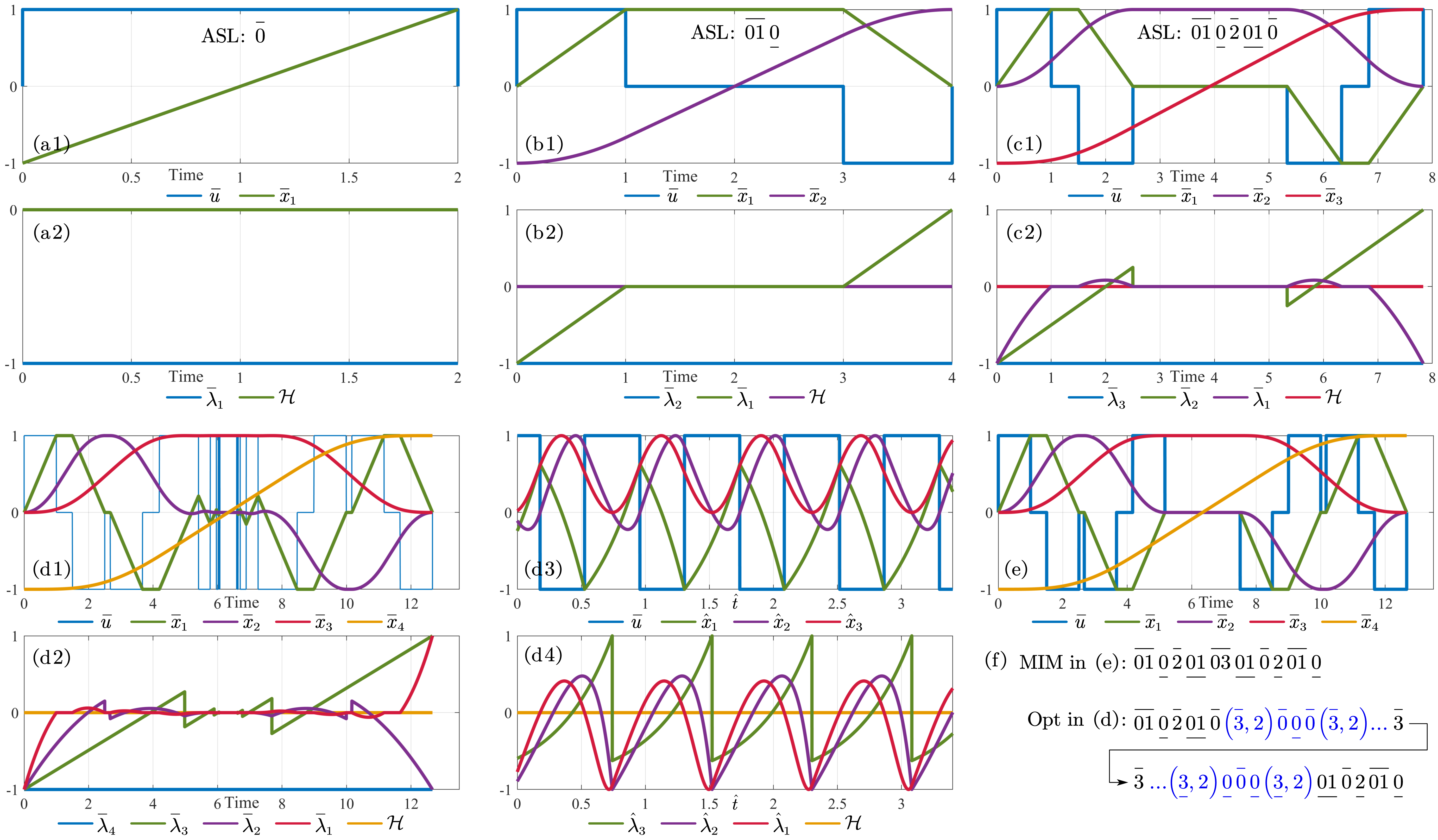}
\caption{(a-d) Strictly optimal trajectories for position-to-position problems of order $n=1,2,3,4$, respectively. (e) A suboptimal trajectory planned by the MIM method in our previous work \cite{wang2024time}. (f) Augmented switching laws (ASL, see Definition \ref{def:AugmentedSwitchingLaw}) of trajectories in (d-e). $M_0=1$, $M_1=1$, $M_2=1.5$, $M_3=4$, $M_4=15$. For an $n$th-order problem, $\vx_0=-M_n\ve_n$, $\vx_\f=M_n\ve_n$. In (a-e), $\bar{u}=\frac{u}{M_0}$. $\forall1\leq k\leq4$, $\bar{x}_k=\frac{x_k}{M_k}$, and $\bar{\lambda}_k=\frac{\lambda_k}{\norm[\infty]{\lambda_k}}$. (d3-d4) show the enlargements of (d1-d2) during the chattering period. The abscissa is in logarithmic scale with respect to time, i.e., $-\log_{10}\left(t_{\infty}-t\right)$, where $t_\infty\approx6.0732$ is the first chattering limit time. $\forall k=1,2,3$, $\hat{x}_k\left(t\right)=\frac{x_k^*\left(t\right)\left(t_\infty-t\right)^{-k}}{\norm[\infty]{x_k^*\left(t\right)\left(t_\infty-t\right)^{-k}}}$, and $\hat{\lambda}_k\left(t\right)=\frac{\lambda_k\left(t\right)\left(t_\infty-t\right)^{k-4}}{\norm[\infty]{\lambda_k\left(t\right)\left(t_\infty-t\right)^{k-4}}}$.
}
\label{fig:firstshow}
\end{figure*}

Generally, the chattering phenomenon \cite{zelikin2012theory} poses a challenge in theoretically investigating and numerically solving high-order optimal control problems with singular.
Fuller \cite{fuller1963study} found the first optimal control problem with chattering arcs, fully studying a problem for the 2nd-order chain-of-integrator system with minimum energy. Robbins \cite{robbins1980junction} constructed a 3rd-order chain-of-integrator system whose optimal control is chattering with a finite total variation. Chattering in hybrid systems is investigated as Zeno phenomenon \cite{heymann2005analysis,lamperski2012lyapunov}. Kupka \cite{kupka1988fuller} proved the ubiquity of the chattering phenomenon, i.e., optimal control problems with chattering as well as a Hamiltonian affine in the single input control constitute an open semialgebraic set. Numerous problems in the industry have been found to have optimal solutions with chattering \cite{borshchevskii1985problem,zhu2016minimum}, where the chattering phenomenon impedes the theoretical analysis and numerical computation of optimal control. In this context, little research has been conducted on the existence of the chattering phenomenon in the classical problem \eqref{eq:optimalproblem}. Neither proofs on non-existence nor counterexamples to the chattering phenomenon in problem \eqref{eq:optimalproblem} have been provided so far. In practice, there exists a longstanding oversight of the chattering phenomenon in problem \eqref{eq:optimalproblem} concerning the time-optimality of S-shaped trajectories with minimum switching times. Some works tried to minimize terminal time by reducing switching times of control \cite{he2020time,ezair2014planning}. As shown in Figs. \ref{fig:firstshow}(a-c), time-optimal trajectories of order $n\leq3$ exhibit a recursively nested S-shaped form. Hence, it is intuitively plausible to expect higher-order optimal trajectories to possess the form in Fig. \ref{fig:firstshow}(e). However, as proved in Section \ref{sec:ChatteringPhenomena4thOrder}, chattering phenomena occur in 4th-order trajectories. The optimal trajectory of order 4 is shown in Fig. \ref{fig:firstshow}(d).

Geometric control serves as a significant mathematical tool to investigate the mechanism underlying chattering \cite{agrachev2013control}. As judged in \cite{schattler2012geometric}, Zelikin and Borisov \cite{zelikin2012theory} have achieved the most comprehensive treatment of the chattering phenomenon so far. In \cite{zelikin2012theory}, the order of a singular arc is defined based on the Poisson bracket of Hamiltonian affine in control, whereas 2nd-order singular arcs with chattering have been widely investigated based on Lagrangian manifolds. However, although problem \eqref{eq:optimalproblem} has a Hamiltonian $\hamilton=\hamilton_0+\hamilton_1u$ affine in control, the chattering phenomenon in problem \eqref{eq:optimalproblem} remains challenging to investigate since $\hamilton_1$ is singular of order $\infty$.

It is meaningful to address impediments to numerical computation from the chattering nature of optimal control. Zelikin and Borisov \cite{zelikin2003optimal} reasoned that the discontinuity induced by chattering worsens the approximation in numerical integration, thus hindering the application of shooting methods in optimal control. Laurent et al. \cite{laurent2007interior} proposed an interior-point approach to solve optimal control problems, where chattering phenomena worsen the convergence. Caponigro et al. \cite{caponigro2018regularization} proposed a regularization method by adding a penalization of the total variation to suppress the chattering phenomenon, successfully obtaining quasi-optimal solutions without chattering. However, it is challenging to prove the existence of chattering through numerical computation due to the limited precision.

This paper investigates the chattering phenomenon in the open problem \eqref{eq:optimalproblem}. Section \ref{sec:PreparationWorks} formulates problem \eqref{eq:optimalproblem} by Hamiltonian and introduces some results of \cite{wang2024time} as preliminaries. Section \ref{sec:MainResults} summarized the main results of this paper. Section \ref{sec:NecessaryConditionsChattering} derives necessary conditions for the chattering phenomenon in problem \eqref{eq:optimalproblem}. Sections \ref{sec:ChatteringPhenomena4thOrder} and \ref{sec:ChatteringPhenomena3rdOrder} prove the existence and non-existence of chattering in low-order problems, respectively. The contributions of this paper are as follows.
\begin{enumerate}
\item This paper establishes a theoretical framework for the chattering phenomenon in the classical and open problem \eqref{eq:optimalproblem} within the discipline of optimal control, i.e., time-optimal control for high-order chain-of-integrator systems with full state constraints. The framework provides novel findings on the existence of chattering, the uniqueness of chattering state constraints in a chattering period, the upper bound on switching times in every unconstrained arc during chattering, and the convergence of states as well as costates to the chattering limit point. Existing works \cite{haschke2008line,kroger2011opening,berscheid2021jerk} lack precise characterization of optimal control's behavior in high-order problems. Even the existence of the chattering phenomenon remains unknown and overlooked. Due to the singular Hamiltonian $\hamilton_1$ of order $\infty$, it is difficult to directly apply predominant technologies for chattering analysis based on Lagrangian manifolds \cite{zelikin2012theory} to problem \eqref{eq:optimalproblem}, which demonstrates the necessity and significance of the established framework.
\item To the best of our knowledge, this paper proves the existence of chattering in problem \eqref{eq:optimalproblem} for the first time, rectifying a longstanding misconception in the industry concerning the time-optimality of S-shaped trajectories with minimal switching times. This paper proves that 4th-order problems with velocity constraints allow a unique chattering mode, where the decay rate in the time domain is precisely solved as $\alpha^*\approx0.1660687$. Based on the developed theory, time-optimal snap-limited trajectories with full state constraints can be planned for the first time. The chattering control is physically realizable due to the finite control frequency in practice. Note that snap-limited position-to-position trajectories are universally applied for ultra-precision control in the industry, yet the oversight of chattering impedes the approach to time-optimal profiles in previous works \cite{ezair2014planning}.
\item This paper fully enumerates existence and non-existence of chattering in problems of order $n\leq4$. Chattering does not exist in problems of order $n\leq3$ and 4th-order problems without velocity constraints. 4th-order problems with velocity constriants represent the problem allowing chattering of the lowest order. For problems of order $n\geq5$, chattering is allowed but not able to be induced by state constraints on $x_n$ and $x_1$. Furthermore, constrained arcs cannot exist in a chattering period.
\end{enumerate}

\section{Preliminaries}\label{sec:PreparationWorks}

\subsection{Problem Formulation}\label{subsec:Problem_Formulation}
Firstly, the well-known Bellman's principle of optimality (BPO) \cite{bellman1952theory} is applied to problem \eqref{eq:optimalproblem}. Consider the optimal trajectory $\vx=\vx^*\left(t\right)$ and the optimal control $u=u^*\left(t\right)$, $t\in\left[0,t_\f^*\right]$. BPO implies that $\forall 0< t_1< t_2<t_\f^*$, the trajectory $\vx=\vx^*\left(t\right)$, $t\in\left[ t_1, t_2\right]$ is optimal in the problem with the initial state vector $\vx^*\left( t_1\right)$ and the terminal state vector $\vx^*\left( t_2\right)$. The corresponding optimal control is $u=u^*\left(t\right)$, $t\in\left[ t_1, t_2\right]$.

Then, this section formulates the optimal control problem \eqref{eq:optimalproblem} from the Hamiltonian perspective. The Hamiltonian is
\begin{equation}\label{eq:hamilton}
\begin{aligned}
&\hamilton\left(\vx\left(t\right),u\left(t\right),\lambda_0,\vlambda\left(t\right),\veta\left(t\right),t\right)\\
=&\lambda_0+\lambda_1u+\sum_{k=2}^{n}\lambda_k x_{k-1}+\sum_{k=1}^{n}\eta_k\left(\abs{x_k}-M_k\right),
\end{aligned}\end{equation}
where $\lambda_0\geq0$ is a constant. $\vlambda\left(t\right)=\left(\lambda_k\left(t\right)\right)_{k=1}^n$ is the costate vector. $\lambda_0$ and $\vlambda$ satisfy $\left(\lambda_0,\vlambda\left(t\right)\right)\not=0$. The initial costates $\vlambda\left(0\right)$ and the terminal costates $\vlambda\left(\tf\right)$ are not assigned since $\vx\left(0\right)$ and $\vx\left(\tf\right)$ are given in problem \eqref{eq:optimalproblem}.

The Hamilton's equations for the costate vector is formulated as $\dot\vlambda=-\frac{\partial\hamilton}{\partial\vx}$. By \eqref{eq:hamilton}, it holds that
\begin{equation}\label{eq:derivative_costate}
\begin{dcases}
\dot\lambda_k=-\lambda_{k+1}-\eta_k\,\sgn\left(x_k\right),\,\forall 1\leq k<n,\\
\dot\lambda_n=-\eta_n\,\sgn\left(x_n\right).
\end{dcases}\end{equation}

In \eqref{eq:hamilton}, $\veta$ is the multiplier vector induced by inequality state constraints \eqref{eq:optimalproblem_x_constraint}, satisfying
\begin{equation}\label{eq:eta_constraint_zero}
\eta_k\geq0,\,\eta_k\left(\abs{x_k}-M_k\right)=0,\,\forall 1\leq k\leq n.\end{equation}
Equivalently, $\forall t\in\left[0,t_\f\right]$, $\eta_k\left(t\right)\not=0$ only if $\abs{x_k\left(t\right)}=M_k$.

In fact, $\abs{x_n}<M_n$ holds almost everywhere (a.e.), where ``a.e.'' means that a property holds except for a zero-measure set \cite{stein2009real}. Note that the set $\left\{t\in\left[0,\tf\right]:x_n\left(t\right)=M_n\right\}$ has at most one accumulation point; otherwise, applying Rolle's Theorem \cite{stein2009real} recursively, it can be proved that $\vx=M_n\ve_n$ at each accumulation point, which contradicts BPO. Hence, $x_n<M_n$ a.e. Similarly, $x_n>-M_n$ a.e. Therefore,
\begin{equation}\label{eq:derivative_costate_lambdan_zero}
\abs{x_n}<M_n,\,\eta_n=0,\,\dot\lambda_n=0\text{ a.e.}\end{equation}

PMP \cite{hartl1995survey} states that the input control $u\left(t\right)$ minimizes the Hamiltonian $\hamilton$ in the feasible set, i.e.,
\begin{equation}
u\left(t\right)\in\mathop{\arg\min}\limits_{\abs{U}\leq M_0}\hamilton\left(\vx\left(t\right),U,\lambda_0,\vlambda\left(t\right),\veta\left(t\right),t\right).\end{equation}
Hence, the bang-bang and singular controls hold, i.e.,
\begin{equation}\label{eq:bang_singular_bang_law}
u\left(t\right)=-M_0\sgn\left(\lambda_1\left(t\right)\right),\text{ if }\lambda_1\left(t\right)\not=0,\end{equation}
where $u\left(t\right)\in\left[-M_0,M_0\right]$ is undetermined during $\lambda_1\left(t\right)=0$.

Along the optimal trajectory, the Hamiltonian $\hamilton$ is constant. Since $J=\int_{0}^{\tf}\mathrm{d}t$ is in a Lagrangian form, it holds that
\begin{equation}\label{eq:hamilton_equiv_0}
\forall t\in\left[0,t_\f\right],\,\hamilton\left(\vx\left(t\right),u\left(t\right),\lambda_0,\vlambda\left(t\right),\veta\left(t\right),t\right)\equiv0.\end{equation}

A junction of costates $\vlambda$ occurs when an inequality state constraint switches between active and inactive, i.e., $\vlambda$ jumps when $\vx$ enters or leaves the constraints' boudnaries \cite{maurer1977optimal}.

\begin{proposition}[Junction condition in problem \eqref{eq:optimalproblem}]\label{prop:junction_condition}
Junction of costates in problem \eqref{eq:optimalproblem} can occur at $t_1$ if $\exists1\leq k\leq n$, s.t. (a) $\abs{x_k}$ is tangent to $M_k$, i.e., $\abs{x_k\left(t_1\right)}=M_k$ and $\abs{x_k}<M_k$ in a deleted neighborhood of $t_1$; or (b) the system enters or leaves the constrained arc $\left\{\abs{x_k}\equiv M_k\right\}$, i.e., $\abs{x_k}\equiv M_k$ at a one-sided neighborhood of $t_1$ and $\abs{x_k}<M_k$ at another one-sided neighborhood of $t_1$. Specifically, $\exists\mu\leq0$, s.t.
\begin{equation}\label{eq:junction_condition}
\vlambda\left(t_1^+\right)-\vlambda\left(t_1^-\right)=\mu\frac{\partial\left(\abs{x_k}-M_k\right)}{\partial\vx}=\mu\,\sgn\left(x_k\right)\ve_k.\end{equation}
In other words, $ \sgn\left(x_k\right)\left[\lambda_k\left(t_1^+\right)-\lambda_k\left(t_1^-\right)\right]\leq0$, while $\forall j\not=k$, $\lambda_j$ is continuous at $t_1$. Furthermore, a junction cannot occur during an unconstrained arc or a constrained arc.
\end{proposition}

\begin{remark}

The junction of $\vlambda$ significantly enriches the behavior of optimal control in problem \eqref{eq:optimalproblem}. If $\vlambda$ is continuous, then there exists an upper bound on the number of switching times. In contrast, the junction can even introduced chattering phenomena in problem \eqref{eq:optimalproblem}, i.e., $u$ switches for infinitely many times in a finite period, as reasoned in Section \ref{sec:ChatteringPhenomena4thOrder}.
\end{remark}

A 3rd-order optimal trajectory is shown in Fig. \ref{fig:PartIDemo} as an example. The bang-bang and singular controls \eqref{eq:bang_singular_bang_law} can be verified. $\lambda_3$ jumps at $t_3$ since $x_3$ is tangent to $-M_3$ at $t_3$.

For computation, the system dynamics is listed as follows.

\begin{proposition}[System dynamics of problem \eqref{eq:optimalproblem}]\label{prop:system_dynamics}
Assume that $\forall 1\leq i\leq N$, $u\equiv u_i$ on $t\in\left(t_{i-1},t_i\right)$, where $\left\{t_i\right\}_{i=0}^N$ increases strictly monotonically. Then, $\forall 1\leq k\leq n$,
\begin{equation}\label{eq:system_dynamics}
\begin{aligned}
x_{k}\left(t_N\right)=&\sum_{j=1}^{k}\frac{x_{k-j}\left(t_0\right)}{j!}T_N^{j}+\sum_{i=1}^{N}\frac{\Delta u_i}{k!}T_i^k,
\end{aligned}\end{equation}
where $\Delta u_i=u_i-u_{i-1}$, $u_0=0$, and $T_i=t_N-t_{i-1}$.
\end{proposition}

\begin{proof}
\eqref{eq:system_dynamics} holds for $ \vx\left(t_0\right)$. Assume that \eqref{eq:system_dynamics} holds for $ \vx\left(t_{N-1}\right)$. Since $u\equiv u_N$ for $t\in\left(t_{N-1},t_N\right)$, it holds that $x_{k}\left(t_N\right)=\sum_{j=0}^{k-1}\frac{T_N^j}{j!}x_{k-j}\left(t_{N-1}\right) +\frac{T_N^k}{k!}u_N$. Therefore, \eqref{eq:system_dynamics} holds for $ \vx\left(t_N\right)$. By induction, Proposition \ref{prop:system_dynamics} holds.
\end{proof}

\subsection{Main Results and Notations of \cite{wang2024time}}\label{subsec:MainResultsOfPart1}

This section introduces a theoretical framework for problem \eqref{eq:optimalproblem} developed by our previous work \cite{wang2024time}, which is helpful for investigating the chattering phenomenon in problem \eqref{eq:optimalproblem}. The following lemma proved in \cite{wang2024time} fully provides behaviors of optimal control except chattering phenomena.

\begin{lemma}[Optimal Control's Behavior of Problem \eqref{eq:optimalproblem} \cite{wang2024time}]\label{lemma:costate}
For the optimal control of problem \eqref{eq:optimalproblem}, it holds that:

\begin{enumerate}
\item\label{lemma:costate_uniqueoptimalcontrol} The optimal control is unique in an a.e. sense. In other words, if $u=u_1^*\left(t\right)$ and $u=u_2^*\left(t\right)$, $t\in\left[0,t_\f^*\right]$, are both optimal controls of \eqref{eq:optimalproblem}, then $u_1^*\left(t\right)=u_2^*\left(t\right)$ a.e.
\item\label{lemma:costate_bangbang} $u=-\sgn\left(\lambda_1\right)M_0$ a.e., where $u\left(t\right)=0$ if $\lambda_1\left(t\right)=0$.
\item\label{lemma:lambda1continue} $\lambda_1$ is continuous despite the junction condition \eqref{eq:junction_condition}.
\item\label{lemma:costate_lambdak_polynomial} $\lambda_k$ consists of $\left(n-k\right)$th degree polynomials and zero. Specifically, $\lambda_k\equiv0$ if $\exists j\geq k$, $\abs{x_j}\equiv M_j$.
\item\label{lemma:costate_sign} If $\vx$ enters $\left\{\abs{x_k}\equiv M_k\right\}$ at $t_1$ from an unconstrained arc, then $u\left(t_1^-\right)=\left(-1\right)^{k-1}\sgn\left(x_k\left(t_1\right)\right)$. If $\vx$ leaves $\left\{\abs{x_k}\equiv M_k\right\}$ at $t_1$ and enters an unconstrained arc, then $u\left(t_1^+\right)=-\sgn\left(x_k\left(t_1\right)\right)$.
\item\label{lemma:costate_junction_secondbracket} If $\exists t_1\in\left(0,\tf\right)$, s.t. $\abs{x_k}$ is tangent to $M_k$ at $t_1$. Then, one and only one of the following conclusions hold:
\begin{enumerate}
\item $\exists l<\frac{k}{2}$, s.t. $x_{k-1}=x_{k-2}=\dots=x_{k-2l+1}=0$ at $t_1$, while $x_{k-2l}\left(t_1\right)\not=0$, and $\sgn\left(x_{k-2l}\left(t_1\right)\right)=-\sgn\left(x_{k}\left(t_1\right)\right)$. The degree of $\abs{x_k\left(t_1\right)}=M_k$ is defined as $2l$.
\item $x_{k-1}=x_{k-2}=\dots=x_1=0$ at $t_1$. $u\left(t_1^+\right)=-\frac{M_0}{M_k}x_k\left(t_1\right)$, and $u\left(t_1^-\right)=\left(-1\right)^{k-1}\frac{M_0}{M_k}x_k\left(t_1\right)$. The degree of $\abs{x_k\left(t_1\right)}=M_k$ is defined as $k$.
\end{enumerate}
\end{enumerate}
\end{lemma}

The proof of Lemma \ref{lemma:costate} is provided in Appendix \ref{app:proof_lemma_costate}.

\begin{figure}[!t]
\centering
\includegraphics[width=\columnwidth]{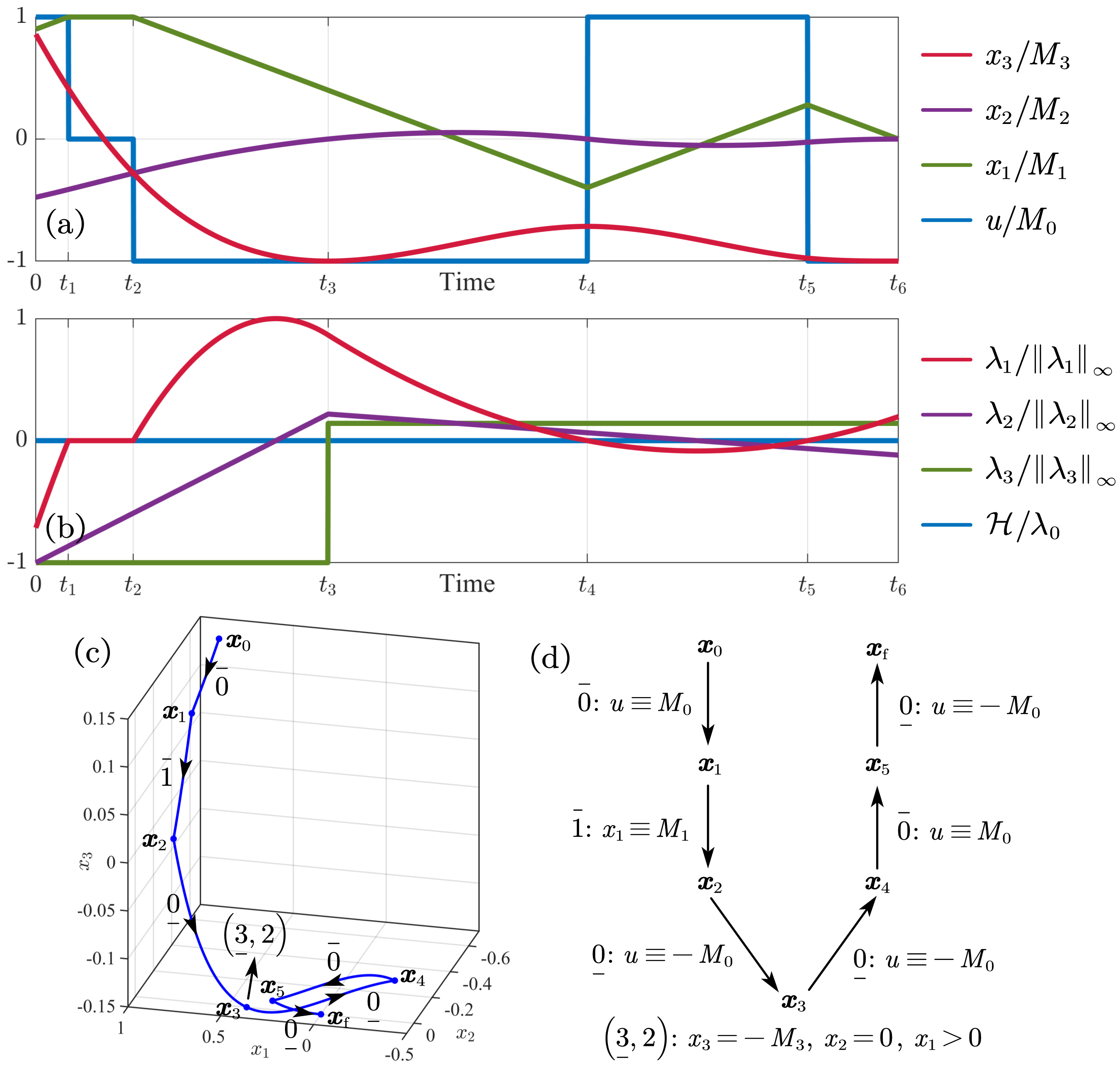}
\caption{A 3rd-order optimal trajectory planned by \cite{wang2024time}, whose ASL is $S=\overline{01}\underline{0}\left(\underline{3},2\right)\underline{0}\overline{0}\underline{0}$. In this example, $\lambda_0>0$, $\vx_0=\left(0.9,-0.715,0.1288\right)$, $\vxf=\left(0,0,-0.15\right)$, and $\vM=\left(1,1,1.5,0.15\right)$. (a) The state vector. (b) The costate vector. (c) The trajectory $\vx=\vx\left(t\right)$. (d) The flow chat of $S$.}
\label{fig:PartIDemo}
\end{figure}

Some notations and definitions are as follows. Denote the set $\mathcal{N}=\N\times\left\{\pm1\right\}$. $\forall s=\left(k,a\right)\in\mathcal{N}$, define the value of $s$ as $\abs{s}=k$, and define the sign of $s$ as $\sgn\left(s\right)=a$. $\forall k\in\N$, denote $\left(k,1\right)$ and $\left(k,-1\right)$ as $\overline{k}$ and $\underline{k}$, respectively. For $s_1,s_2\in\mathcal{N}$, denote $s_1=-s_2$ if $\abs{s_1}=\abs{s_2}$ and $\sgn\left(s_1\right)=-\sgn\left(s_2\right)$. Based on Lemma \ref{lemma:costate}, the system behavior and the tangent marker in our previous work \cite{wang2024time} are defined as follows.

\begin{definition}\label{def:SystemBehavior}
A \textbf{system behavior} of an unconstrained arc or a constrained arc in problem \eqref{eq:optimalproblem} is denoted as follows:
\begin{enumerate}
\item $\overline{0}\,\left(\underline{0}\right)$ is an unconstrained arc $\left\{u\equiv M_0\,\left(-M_0\right)\right\}$.
\item $\overline{k}\,\left(\underline{k}\right)$ is a constrained arc $\left\{x_k\equiv M_k\,\left(-M_k\right)\right\}$.
\end{enumerate}
\end{definition}

\begin{definition}\label{def:TangentMarker}
Assume that $\abs{x_k}$ is tangent to $M_k$ at $t_1$ with a degree $h$, as described in Lemma \ref{lemma:costate}.\ref{lemma:costate_junction_secondbracket}. Then, the \textbf{tangent marker} is denoted as $\left(s,h\right)$, where $s =\left(k,\sgn\left(x_k\left(t_1\right)\right)\right)\in\mathcal{N}$.
\end{definition}

\begin{definition}\label{def:AugmentedSwitchingLaw}
In problem \eqref{eq:optimalproblem}, the \textbf{augmented switching law} (ASL) of an optimal trajectory is $S=s_1s_2\dots s_N$ if the trajectory passes through $s_1,s_2,\dots,s_N$ sequentially, where $\forall 1\leq i\leq N$, $s_i$ is a system behavior or a tangent marker.
\end{definition}

An example is shown in Figs. \ref{fig:PartIDemo}(c-d), where the optimal trajectory is represented as $S=\overline{01}\underline{0}\left(\underline{3},2\right)\underline{0}\overline{0}\underline{0}$. Firstly, the system passes through $\left\{u\equiv M_0\right\}$, $\left\{x_1\equiv M_1\right\}$, and $\left\{u\equiv -M_0\right\}$. Then, $x_3$ is tangent to $-M_3$ at $t_3$. Next, the system passes through $\left\{u\equiv -M_0\right\}$, $\left\{u\equiv M_0\right\}$, and $\left\{u\equiv -M_0\right\}$. Finally, $\vx$ reaches $\vxf$ at $\tf$. It is noteworthy that the ASL does not include the motion time of each stage, which is also necessary to determine the optimal control.

Based on the formulation in Section \ref{subsec:Problem_Formulation} and the main results of \cite{wang2024time} in Section \ref{subsec:MainResultsOfPart1}, the chattering phenomenon in problem \eqref{eq:optimalproblem} can be investigated in the following sections.

\section{Main Results}\label{sec:MainResults}

This section provides main results of this paper. Section \ref{subsec:ChatteringTheory_MainResults} introduces the established theoretical framework for chattering in problem \eqref{eq:optimalproblem}, overcoming lack of mathematical tools. Section \ref{subsec:ExistenceChattering_MainResults} provides the existence of chattering in problem \eqref{eq:optimalproblem}, rectifying the long misconception on the time-optimality of S-shaped trajectories.

\subsection{Theoretical Framework for Chattering in Problem \eqref{eq:optimalproblem}}\label{subsec:ChatteringTheory_MainResults}

As pointed out in Section \ref{sec:Introduction}, neither proofs on non-existence nor counterexamples to the chattering phenomenon in the classical problem \eqref{eq:optimalproblem} have been provided so far. However, the predominant technology for chattering analysis based on Lagrangian manifolds \cite{zelikin2012theory} is difficult to directly apply to \eqref{eq:optimalproblem} for the following reasons. Note that $\hamilton=\hamilton_0+\hamilton_1u$ is affine in $u$, where $\hamilton_1=\lambda_1$. By \eqref{eq:derivative_costate}, $\forall i\in\N^*$, $\frac{\mathrm{d}^i \hamilton_1}{\mathrm{d} t^i}$ is independent of $u$; hence, the Hamiltonian is singular of order $\infty$. So the technology in \cite{zelikin2012theory} is difficult to directly apply to problem \eqref{eq:optimalproblem}.

This paper develops a specialized theoretical framework, i.e., Theorem \ref{thm:ChatteringPhenomenaBehavior}, for chattering in problem \eqref{eq:optimalproblem}. In particular, an inequality state constraint $s\in\mathcal{N}$ refers to $\sgn\left(s\right)x_{\abs{s}}\leq M_{\abs{s}}$.

\begin{theorem}\label{thm:ChatteringPhenomenaBehavior}
Assume that chattering occurs on a left-side neighborhood of $t_\infty$ in problem \eqref{eq:optimalproblem} where $t_\infty \in\left(0,t_\f\right)$ is the limit time. Then, $\exists t_0<t_\infty$ and $s \in\mathcal{N}$, s.t. $s$ is the unique state constraint allowed active for some $t\in\left[t_0,t_\infty\right]$. Furthermore, the following conclusions hold.
\begin{enumerate}
\item\label{thm:ChatteringPhenomenaBehavior_1sn} $1<\abs{s}<n$.

\item\label{thm:ChatteringPhenomenaBehavior_sh} $\exists\left\{t_i\right\}_{i=1}^\infty\subset\left(t_0,t_\infty\right)$ increasing monotonically and converging to $t_\infty$, s.t. $s$ is active on $\left\{t_i\right\}_{i=1}^\infty$ and is inactive except $\left\{t_i\right\}_{i=1}^\infty$. Furthermore, $\left\{t_i\right\}_{i=1}^\infty$ is the set of junction time.
\item\label{thm:ChatteringPhenomenaBehavior_sgnxs} $\forall t\in\left[t_0,t_\infty\right]$, $\sgn\left(x_{\abs{s}}\left(t\right)\right)\equiv\sgn\left(s\right)$.
\item\label{thm:ChatteringPhenomenaBehavior_sgn_x_lambad_sk} $\forall \abs{s}<k\leq n$, $t\in\left(t_0,t_\infty\right)$, it holds that $\sgn\left(x_k\left(t\right)\right)\equiv\mathrm{const}$ and $\sgn\left(\lambda_k\left(t\right)\right)\equiv\mathrm{const}$.
\item\label{thm:ChatteringPhenomenaBehavior_switchingtime} $\forall 1\leq k\leq\abs{s}$, $i\in\N^*$, during $\left(t_i,t_{i+1}\right)$, $\lambda_k$ has at most $\left(\abs{s}-k+1\right)$ roots and $u$ switches for at most $\abs{s}$ times.
\item\label{thm:ChatteringPhenomenaBehavior_limit} $\forall 1\leq k\leq\abs{s}$, $t\in\left(t_0,t_\infty\right)$, it holds that:
\begin{enumerate}
\item $\lambda_k$ crosses 0 for infinitely many times during $\left(t,t_\infty\right)$. $\forall i\in\N^*$, $\sgn\left(s\right)\lambda_{\abs{s}}$ increases monotonically for $\left(t_i,t_{i+1}\right)$ and jumps decreasingly at $t_i$.
\item $\sup_{\tau\in\left[t,t_\infty\right]}\abs{x_k\left(\tau\right)-x_k\left(t_\infty\right)}=\mathcal{O}(\left(t_\infty-t\right)^k)$, and $\sup_{\tau\in\left[t,t_\infty\right]}\abs{\lambda_k\left(\tau\right)}=\mathcal{O}(\left(t_\infty-t\right)^{\abs{s}-k+1})$.
\item $\forall 1\leq k<\abs{s}$, $\lim_{t\to t_\infty}x_k\left(t\right)=x_k\left(t_\infty\right)=0$. For $\abs{s}$, $\lim_{t\to t_\infty}x_{\abs{s}}\left(t\right)=x_{\abs{s}}\left(t_\infty\right)=M_{\abs{s}}\sgn\left(s\right)$. $\forall 1\leq k\leq\abs{s}$, $\lim_{t\to t_\infty}\lambda_k\left(t\right)=\lambda_k\left(t_\infty\right)=0$.
\end{enumerate}
\end{enumerate}
Similar conclusions hold for a right-side neighborhood of $t_\infty$.
\end{theorem}

\begin{remark}
$\left[t_0,t_\infty\right]$ is defined as a \textbf{chattering period}. Theorem \ref{thm:ChatteringPhenomenaBehavior} offers a visualization of chattering, as shown in Fig. \ref{fig:chattering_case}(c). Infinitely many unconstrained arcs are connected at the boundary of $s$ during $\left(t_0,t_\infty\right)$. The trajectory $\vx$ limits to the constrained arc $s$ at $t_\infty$. In other words, cases shown in Figs. \ref{fig:chattering_case}(a-b) are impossible, where more than one constraints are allowed active or constrained arcs exist in the chattering period, respectively.
\end{remark}

\begin{remark}
In a chattering period, the unique state constraint allowed active is defined as a \textbf{chattering constraint}. The uniqueness of chattering constraint is significant for investigating the chattering phenomenon. One can consider chattering constraints one by one, and get rid of all other constraints.
\end{remark}

\begin{proof}[Proof of Theorem \ref{thm:ChatteringPhenomenaBehavior}]
The uniqueness of the chattering constraint is proved in Section \ref{subsec:UniquenessOfStateConstraintChattering}. In Theorem \ref{thm:ChatteringPhenomenaBehavior}, the first four conclusions are proved in Section \ref{subsec:NecessaryConditionsStateConstraintChattering}, and the last two conclusions are proved in Section \ref{subsec:BehaviorsOfCostatesDuringChatteringPeriod}.
\end{proof}

\begin{figure}[!t]
\centering
\includegraphics[width=\columnwidth]{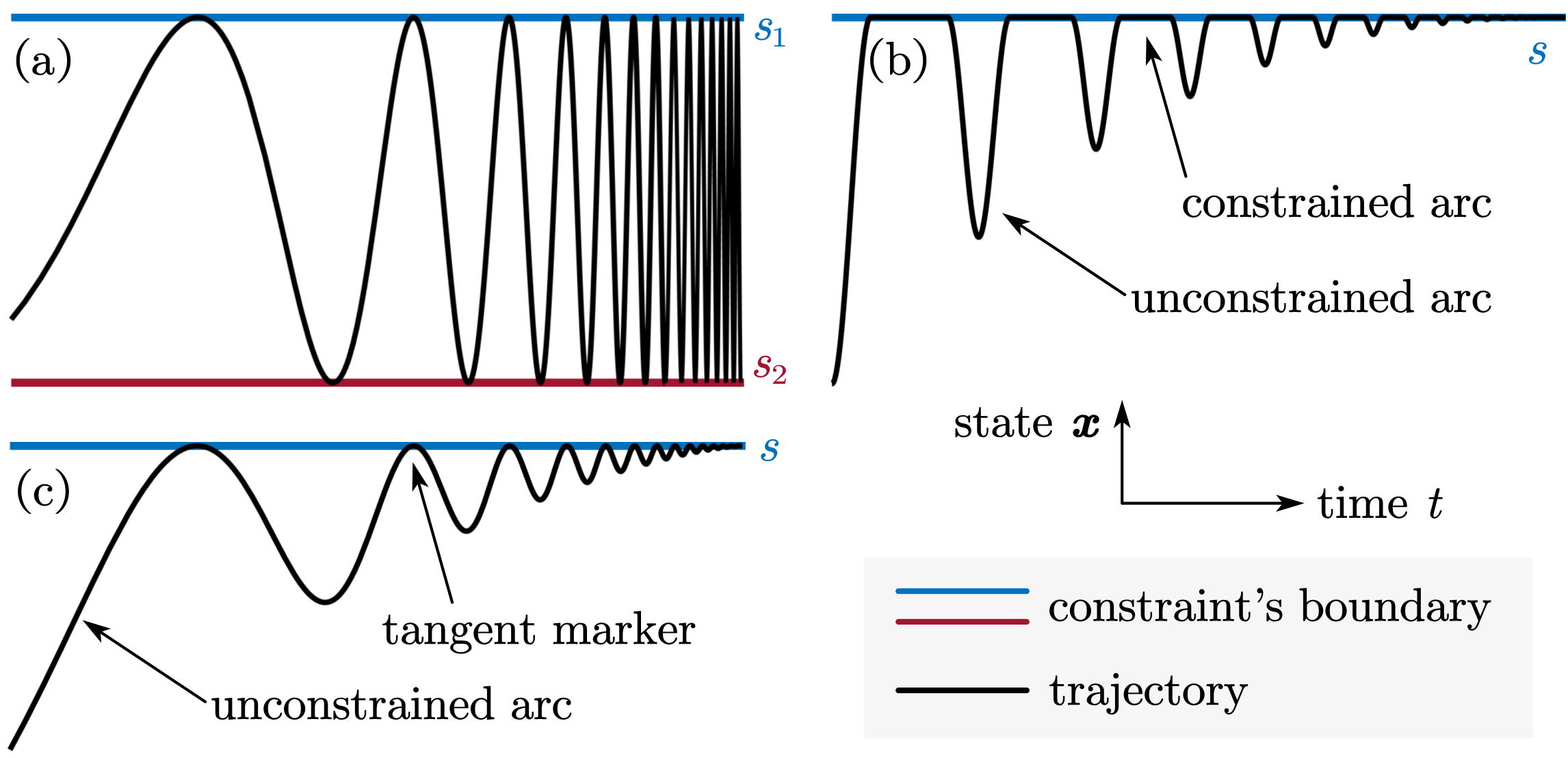}
\caption{Chattering arcs in problem \eqref{eq:optimalproblem}. (a) Multiple active constraints. (b) Constrained arcs. (c) Unconstrained arcs connected at the constraint's boundary. (a) and (b) are not allowed, while (c) exists in chattering periods.}
\label{fig:chattering_case}
\end{figure}

\subsection{Existence of Chattering in Problem \eqref{eq:optimalproblem}}\label{subsec:ExistenceChattering_MainResults}

As reviewed in Section \ref{sec:Introduction}, there exists a longstanding oversight of the chattering phenomenon in problem \eqref{eq:optimalproblem} due to the recursive structure of the chain-of-integrator, i.e., \eqref{eq:optimalproblem_dotxk} and \eqref{eq:optimalproblem_dotx1}. It is widely accepted that $\vx$ should reach the maximum velocity $\left\{\abs{x_{n-1}}\equiv M_{n-1}\right\}$ as fast as possible to achieve minimal time of the whole trajectory, resulting in the well-known S-shaped trajectories without chattering. Theorem \ref{thm:Chattering_n4s3_optimal_allcases} proves the existence of chattering in problem \eqref{eq:optimalproblem} when $n\geq4$, rectifying the above longstanding misconception concerning the optimality of S-shaped trajectories. Based on Theorem \ref{thm:ChatteringPhenomenaBehavior}, chattering constraints can be investigated one by one.

\begin{theorem}\label{thm:Chattering_n4s3_optimal_allcases}
Consider the chattering constraint $s$ in \eqref{eq:optimalproblem}.
\begin{enumerate}
\item\label{thm:Chattering_n4s3_optimal_allcases_loworder} Chattering does not occur if $n\leq3$ or $n=4$ and $\abs{s}\not=3$.
\item\label{thm:Chattering_n4s3_optimal_allcases_n4s3} The case where $n=4$ and $\abs{s}=3$ represents problems of the lowest order that allow chattering. The unique chattering mode is given in Theorems \ref{thm:Chattering_n4s3_equivalent} and \ref{thm:Chattering_n4s3_equivalent_optimal_y11}.
\item\label{thm:Chattering_n4s3_optimal_allcases_highorder} Chattering phenomena can occur when $n\geq5$.
\end{enumerate}
\end{theorem}

\begin{proof}
Theorems \ref{thm:Chattering_n4s3_optimal_allcases}.\ref{thm:Chattering_n4s3_optimal_allcases_loworder} and \ref{thm:Chattering_n4s3_optimal_allcases}.\ref{thm:Chattering_n4s3_optimal_allcases_n4s3} are proved in Sections \ref{sec:ChatteringPhenomena3rdOrder} and \ref{sec:ChatteringPhenomena4thOrder}, respectively. Arbitrarily consider a chattering optimal trajectory of 4th-order, i.e., $\vx=\hat\vx\left(t\right)$ and $u=\hat{u}\left(t\right)$, $t\in\left[0,\hat{t}_\f\right]$. For $n\geq5$, a chattering optimal trajectory of $n$th-order can be constructed through the integration of $\vx=\hat\vx\left(t\right)$. For example, let $\vx_{0,1:4} =\hat{x}\left(0\right)$, $\vx_{\f,1:4} =\hat{x}\left(t_\f\right)$, and $\vx_{0,5:n} =\vzero$. $\vx_{\f,5:n}$ is given by $\vx_0$ and $\hat{u}\left(t\right)$. $\vM_{5:n}$ is large enough. Then, the control $u=\hat{u}\left(t\right)$, $t\in\left[0,\hat{t}_\f\right]$ is optimal in the above problem of order $n\geq5$. Note that $ \hat{u}$ chatters. So Theorem \ref{thm:Chattering_n4s3_optimal_allcases}.\ref{thm:Chattering_n4s3_optimal_allcases_highorder} holds.
\end{proof}

\subsection{Chattering Mode of Problem \eqref{eq:optimalproblem} when $n=4$ and $\abs{s}=3$}\label{subsec:ChatteringMode}

How the optimal control chatters is a key issue to solve problem \eqref{eq:optimalproblem}. This paper provides the trajectories in chattering period for $n=4$ and $\abs{s}=3$. Specifically, a unique \textbf{chattering mode} exists that once $\abs{x_3}$ is tangent to $M_3$, the optimal trajectory chatters following the same way.

Based on Theorem \ref{thm:ChatteringPhenomenaBehavior}, this section only needs to consider $s=\overline{3}$ as the unique state constraint and gets rid of all other constraints in problem \eqref{eq:optimalproblem}. Since $\vx$ is tangent to $\left\{x_3=M_3\right\}$ for infinitely many times and finally reaches the constrained arc $\left\{x_3\equiv M_3\right\}$, by BPO, the following problem is considered:
\begin{IEEEeqnarray}{rl}\label{eq:optimalproblem_n4s3}
\min\quad& J=\int_{0}^{\tf}\mathrm{d}t=\tf,\IEEEyesnumber\IEEEyessubnumber*\\
\st\quad&\dot{x}_4=x_3,\,\dot{x}_3=x_2,\,\dot{x}_2=x_1,\,\dot{x}_1=u,\\
&\vx\left(0\right)=\vx_0=\left(x_{0,1},0,M_3,x_{0,4}\right),\label{eq:optimalproblemn4s3_x0}\\
&\vx\left(\tf\right)=\vx_\f=\left(0,0,M_3,x_{\f4}\right),\label{eq:optimalproblemn4s3_xf}\\
&x_3\left(t\right)\leq M_3,\,\abs{u\left(t\right)}\leq M_0,\,\forall t\in\left[0,t_\f\right]\label{eq:optimalproblemn4s3_x_constraint}
\end{IEEEeqnarray}
where $x_{0,1}<0$ and $t_0=0$. To solve problem \eqref{eq:optimalproblem_n4s3}, the following parameter-free infinite-horizon problem is constructed:
\begin{IEEEeqnarray}{rl}\label{eq:optimalproblem_n4s3_equivalent}
\min\quad& \widehat{J}=\int_{0}^{\infty}y_3\left(\tau\right)\,\mathrm{d}\tau,\IEEEyesnumber\IEEEyessubnumber*\\
\st\quad&\dot{y}_3=y_2,\,\dot{y}_2=y_1,\,\dot{y}_1=v,\\
&\vy\left(0\right)=\vy_0=\left(1,0,0\right),\\
&y_3\left(\tau\right)\geq 0,\,\abs{v\left(\tau\right)}\leq 1,\,\forall\tau\in\left(0,\infty\right).\label{eq:optimalproblem_n4s3_equivalent_y3}
\end{IEEEeqnarray}

Evidently, $\inf \widehat{J}<\infty$ since the time-optimal trajectory between $\vy_0$ and $\vzero$ is a feasible solution with $\widehat{J}<\infty$. Denote $\tau_\infty=\arg\min\left\{\tau\in\left(0,\infty\right):\vy\left(\tau\right)=\vzero\right\}\in\overline{\R}_{++}$. Evidently, if $\tau_\infty<\infty$, then $\vy\equiv\vzero$ on $\left(\tau_\infty,\infty\right)$. An equivalent relationship between problems \eqref{eq:optimalproblem_n4s3} and \eqref{eq:optimalproblem_n4s3_equivalent} is provided in Theorem \ref{thm:Chattering_n4s3_equivalent}.

\begin{theorem}\label{thm:Chattering_n4s3_equivalent}
Assume that problem \eqref{eq:optimalproblem_n4s3_equivalent} has an optimal solution $v=v^*\left(\tau\right)$ with the trajectory $\vy=\vy^*\left(\tau\right)$, satisfying $\tau_\infty^*\triangleq\arg\min\left\{\tau\in\left(0,\infty\right):\vy^*\left(\tau\right)=\vzero\right\}<\infty$. If in problem \eqref{eq:optimalproblem_n4s3},
\begin{equation}\label{eq:optimalsolution_equivalent_n4s3_condition}
x_{\f4}-x_{0,4}\geq -\frac{x_{0,1}}{M_0}\left(\frac{x_{0,1}^3}{M_0^2}\widehat{J}^*+M_3\tau_\infty^*\right),\end{equation}

then the optimal solution of problem \eqref{eq:optimalproblem_n4s3} is as follows:
\begin{IEEEeqnarray}{l}\label{eq:optimalsolution_equivalent_n4s3}
t_\infty^*=-\frac{x_{0,1}}{M_0}\tau_\infty^*\in\left(0,\infty\right),\IEEEyesnumber\IEEEyessubnumber*\label{eq:optimalsolution_equivalent_n4s3_tinfty}\\
t_\f^*=\frac{x_{\f4}-x_{0,4}}{M_3}+\frac{x_{0,1}^4\widehat{J}^*}{M_0^3M_3}\geq t_\infty^*.\label{eq:optimalsolution_equivalent_n4s3_tf}
\end{IEEEeqnarray}
$\forall t\in\left(0,t_\f^*\right)$, the following expressions hold a.e.:
\begin{IEEEeqnarray}{l}
u^*\left(t\right)=-M_0v^*\left(-\frac{M_0}{x_{0,1}}t\right),\IEEEyessubnumber*\label{eq:optimalsolution_equivalent_n4s3_u}\\
x_1^*\left(t\right)=x_{0,1}y_1^*\left(-\frac{M_0}{x_{0,1}}t\right),\\
x_2^*\left(t\right)=-\frac{x_{0,1}^2}{M_0}y_2^*\left(-\frac{M_0}{x_{0,1}}t\right),\\
x_3^*\left(t\right)=\frac{x_{0,1}^3}{M_0^2}y_3^*\left(-\frac{M_0}{x_{0,1}}t\right)+M_3,\label{eq:optimalsolution_equivalent_n4s3_x3}\\
x_4^*\left(t\right)=-\frac{x_{0,1}^4}{M_0^3}\int_{0}^{-\frac{M_0}{x_{0,1}}t}y_3^*\left(\tau\right)\,\mathrm{d}\tau+M_3t+x_{0,4}.\label{eq:optimalsolution_equivalent_n4s3_x4}
\end{IEEEeqnarray}

\end{theorem}

Problem \eqref{eq:optimalproblem_n4s3_equivalent} can be fully solved by Theorem \ref{thm:Chattering_n4s3_equivalent_optimal_y11} where the optimal control chatters. By Theorem \ref{thm:Chattering_n4s3_equivalent}, the optimal solution of problem \eqref{eq:optimalproblem_n4s3} also chatters.

\begin{theorem}\label{thm:Chattering_n4s3_equivalent_optimal_y11}
$\exists0<\beta_1<\beta_2<1<\beta_3$, $\alpha\in\left(0,1\right)$, $\tau_1>0$, s.t.

\begin{IEEEeqnarray}{l}\label{eq:Chattering_n4s3_equivalent_optimal_y11_equationsystem}
\left(1-2\left(1-\beta_1\right)+2\left(1-\beta_2\right)\right)\tau_1=1-\alpha,\IEEEyesnumber\IEEEyessubnumber*\label{eq:Chattering_n4s3_equivalent_optimal_y11_equationsystem_x1}\\
\left(1-2\left(1-\beta_1\right)^2+2\left(1-\beta_2\right)^2\right)\tau_1=2,\label{eq:Chattering_n4s3_equivalent_optimal_y11_equationsystem_x2}\\
\left(1-2\left(1-\beta_1\right)^3+2\left(1-\beta_2\right)^3\right)\tau_1=3,\label{eq:Chattering_n4s3_equivalent_optimal_y11_equationsystem_x3}\\
2\left(\beta_1+\beta_2+\beta_3\right)+\left(\alpha^2-1\right)\sum_{j<k}\beta_j\beta_k=3,\label{eq:Chattering_n4s3_equivalent_optimal_y11_equationsystem_coeff1}\\
\beta_1+\beta_2+\beta_3-\sum_{j<k}\beta_j\beta_k-\beta_1\beta_2\beta_3\left(\alpha^3-1\right)=1,\label{eq:Chattering_n4s3_equivalent_optimal_y11_equationsystem_coeff0}\\
\tau_\infty=\tau_1/\left(1-\alpha\right).
\end{IEEEeqnarray}
Specifically, Eq. \eqref{eq:Chattering_n4s3_equivalent_optimal_y11_equationsystem} has a unique feasible solution, i.e.,
\begin{equation}\label{eq:Chattering_n4s3_equivalent_optimal_solution}
\begin{aligned}
&\alpha^*\approx0.1660687,\tau_{1}^*\approx4.2479105,\tau_\infty^*\approx5.0938372,\\
&\beta_1^*\approx0.4698574,\beta_2^*\approx0.8716996,\beta_3^*\approx1.0283610.
\end{aligned}\end{equation}
The optimal solution of problem \eqref{eq:optimalproblem_n4s3_equivalent} satisfies $\forall i\in\N$, $\vy\left(\tau_i\right)={\alpha^*}^i\ve_1$ where $\tau_i\triangleq\frac{1-{\alpha^*}^i}{1-\alpha^*}\tau_1^*$ is the junction time. $\forall i\in\N^*$, the optimal control in $\left(\tau_{i-1},\tau_i\right)$ is $\forall \beta\in\left(0,1\right)$,
\begin{equation}\label{eq:Chattering_n4s3_equivalent_optimal_u}
v^*\left(\left(1-\beta\right)\tau_{i-1}+\beta\tau_i\right)=\begin{dcases}
-1,&\beta\in\left(0,\beta_1\right),\\
1,&\beta\in\left(\beta_1,\beta_2\right),\\
-1,&\beta\in\left(\beta_2,1\right).
\end{dcases}\end{equation}
The corresponding costate vector is $\forall\tau\in\left(\tau_{i-1},\tau_i\right)$,
\begin{IEEEeqnarray}{l}\label{eq:Chattering_n4s3_equivalent_optimal_p}
p_1\left(\tau\right)=-\frac{p_0}{6}\prod_{k=1}^{3}\left(\tau-\left(1-\beta_k\right)\tau_{i-1}-\beta_k\tau_i\right),\IEEEyesnumber\IEEEyessubnumber*\label{eq:Chattering_n4s3_equivalent_optimal_p1}\\
p_2\left(\tau\right)=-\dot{p}_1,\,p_3\left(\tau\right)=-\dot{p}_2\left(\tau\right),\,\left(p_i\right)_{i=0}^3\not=\vzero.
\end{IEEEeqnarray}
\end{theorem}

\begin{remark}
The unique chattering mode of problem \eqref{eq:optimalproblem_n4s3_equivalent} is given in \eqref{eq:Chattering_n4s3_equivalent_optimal_u}. Applying the homogeneity structure in Proposition \ref{prop:n4s3_equivalent_recursive}, for any feasible $\vy_0$, once $y_3$ is tangent to 0, the optimal control chatters like \eqref{eq:Chattering_n4s3_equivalent_optimal_u}. Therefore, problem \eqref{eq:optimalproblem} with $n=4$ and $\abs{s}=3$ has a unique chattering mode due to Theorem \ref{thm:Chattering_n4s3_equivalent}.
\end{remark}

The proofs of Theorems \ref{thm:Chattering_n4s3_equivalent} and \ref{thm:Chattering_n4s3_equivalent_optimal_y11} are provided in Section \ref{sec:ChatteringPhenomena4thOrder}. As a direct corollary, Theorem \ref{thm:Chattering_n4s3_optimal_solution} gives the optimal solution of problem \eqref{eq:optimalproblem_n4s3} where the chattering mode is in \eqref{eq:Chattering_n4s3_optimal_u}.

\begin{theorem}\label{thm:Chattering_n4s3_optimal_solution}
Apply the values in \eqref{eq:Chattering_n4s3_equivalent_optimal_solution}. Assume that problem \eqref{eq:optimalproblem_n4s3} satisfies the condition \eqref{eq:optimalsolution_equivalent_n4s3_condition}. Then, chattering occurs in the optimal solution of problem \eqref{eq:optimalproblem_n4s3} as follows.
\begin{enumerate}
\item $\forall i\in\N$, $t_i^*=-\frac{x_{0,1}}{M_0}\tau_i^*=-\frac{x_{0,1}\left(1-{\alpha^*}^i\right)}{M_0\left(1-\alpha^*\right)}\tau_1^*$ is the junction time of $\lambda_3$. Then, $x_1^*\left(t_i\right)={\alpha^*}^ix_{0,1}$, $x_2^*\left(t_i\right)=0$, $x_3^*\left(t_i\right)=M_3$, and $x_4^*\left(t_i\right)=x_{\infty,4}-{\alpha^*}^{4i}\left(x_{\infty,4}-x_{0,4}\right)$. The chattering limit time $t_\infty^*$ is given in \eqref{eq:optimalsolution_equivalent_n4s3_tinfty}, and $x_{\infty,4}=x_4^*\left(t_\infty^*\right)=x_{0,4}-\frac{x_{0,1}}{M_0}\left(\frac{x_{0,1}^3}{M_0^2}\widehat{J}^*+M_3\tau_\infty^*\right)$. Specifically, $\vx^*\left(t_\infty^*\right)=\left(0,0,M_3,x_{\infty,4}\right)$.
\item $\forall i\in\N$, the optimal control in $\left(t_{i-1},t_i\right)$ is $\forall \beta\in\left(0,1\right)$,
\begin{equation}\label{eq:Chattering_n4s3_optimal_u}
\hspace{-0.3cm}u^*\left(\left(1-\beta\right)t_{i-1}+\beta t_i\right)=\begin{dcases}
M_0,&\beta\in\left(0,\beta_1^{*}\right),\\
-M_0,&\beta\in\left(\beta_1^{*},\beta_2^{*}\right),\\
M_0,&\beta\in\left(\beta_2^{*},1\right).
\end{dcases}\end{equation}
Furthermore, $\forall t\in\left(t_\infty^*,t_\f^*\right)$, it holds that $x_3^*\left(t\right)\equiv M_3$ and $u^*\left(t\right)\equiv0$, where $t_\f^*$ satisfies \eqref{eq:optimalsolution_equivalent_n4s3_tf}.
\end{enumerate}
\end{theorem}

\section{Chattering Conditions in Problem \eqref{eq:optimalproblem}}\label{sec:NecessaryConditionsChattering}

This section aims to prove Theorem \ref{thm:ChatteringPhenomenaBehavior}, providing necessary conditions for chattering in problem \eqref{eq:optimalproblem}. The existence of chattering is always assumed in this section. Firstly, Section \ref{subsec:UniquenessOfStateConstraintChattering} proves the uniqueness of chattering constraints. Then, Section \ref{subsec:NecessaryConditionsStateConstraintChattering} provides some necessary conditions of chattering constraints. Finally, Section \ref{subsec:BehaviorsOfCostatesDuringChatteringPeriod} analyzes the limiting behavior of $\vx$ and $\vlambda$ in the chattering period.

\subsection{Uniqueness of Chattering Constraints}\label{subsec:UniquenessOfStateConstraintChattering}

In chattering, the control $u$ jumps for infinitely many times in a finite time period. It is evident that state constraints should switch between active and inactive for infinitely many times; otherwise, Lemma \ref{lemma:costate}.\ref{lemma:costate_lambdak_polynomial} implies that $\lambda_1$ has a finite number of roots, which contradicts PMP and chattering.

Denote the limit time as $t_\infty$. Assume that chattering occurs in the left-side neighborhood of $t_\infty$, i.e., $u$ switches for infinitely many times in $\left(t_0,t_\infty\right)$. Assume that junctions at $\left\{s_r\right\}_{r=1}^R\subset\mathcal{N}$ occur for infinitely many times in $\left(t_0,t_\infty\right)$. According to BPO, one can investigate the duration after all state constraints except $\left\{s_r\right\}_{r=1}^R$ allowed active. Hence, state constraints expect $\left\{s_r\right\}_{r=1}^R$ are not taken into consideration. In this section, ``junction time'' refers to the time when a state constraint switches between active and inactive. $\vlambda$ is allowed but unnecessary to jump at junction time.

$\forall 1\leq r\leq R$, denote the junction time set of $s_r$ as $\left\{t_i^{\left(r\right)}\right\}_{i=1}^\infty$ which increases monotonically and converges to $t_\infty$. Then, $x_{\abs{s_r}}\left(t_i^{\left(r\right)}\right)=\sgn\left(s_r\right)M_{\abs{s_r}}$, while $\sgn\left(s_r\right)x_{\abs{s_r}}<M_{\abs{s_r}}$ holds in a one-sided neighborhood of $t_i^{\left(r\right)}$.

The uniqueness of chattering constraints is given as the following proposition, i.e., Fig. \ref{fig:chattering_case}(a) is impossible.

\begin{proposition}[Uniqueness of Chattering Constraints]\label{prop:UniquenessOfStateConstraintChattering}
If the chattering phenomenon occurs in problem \eqref{eq:optimalproblem} at $\left[t_0,t_\infty\right]$, then $\exists\delta>0$, s.t. there exists a unique state constraint switching between active and inactive during $\left(t_\infty-\delta,t_\infty\right)$, i.e., $R =1$.
\end{proposition}

\begin{proof}

Assume that $R\geq2$. $s_1\not=s_2$ implies that (a) $\abs{s_1}\not=\abs{s_2}$ or (b) $s_1=-s_2$. Before the proof, it should be pointed out that $\forall1\leq k\leq n$, $\norm[\infty]{\dot x_k}\triangleq\sup_{t\in\left[0,t_\f\right]}\abs{\dot x_k\left(t\right)}<\infty$.

For Case (a), assume that $\abs{s_1}>\abs{s_2}\geq1$. Note that $\forall i\in\N^*$, $x_{\abs{s_1}}\left(t_i^{\left(1\right)}\right)=\sgn\left(s_1\right)M_{\abs{s_1}}$. Applying Rolle's theorem \cite{stein2009real} recursively, one has $\forall i\in\N^*$, $\exists\hat{t}_i\in\left(t_i^{\left(1\right)},t_{i+\abs{s_1}-\abs{s_2}}^{\left(1\right)}\right)$, s.t. $x_{\abs{s_2}}\left(\hat{t}_i\right)=0$ since $x_{\abs{s_2}}=\frac{\mathrm{d}^{\abs{s_1}-\abs{s_2}}x_{\abs{s_1}}}{\mathrm{d}t^{\abs{s_1}-\abs{s_2}}}$. As $i\to\infty$, it holds that $t_i^{\left(2\right)},\hat{t}_i\to t_\infty$, i.e., $\abs{t_i^{\left(2\right)}-\hat{t}_i}\to0$; hence,
\begin{equation}
\begin{aligned}
0<M_{\abs{s_2}}=&\abs{x_{\abs{s_2}}\left(t_i^{\left(2\right)}\right)-x_{\abs{s_2}}\left(\hat{t}_i\right)}\\
\leq&\norm[\infty]{\dot x_{\abs{s_2}}}\abs{t_i^{\left(2\right)}-\hat{t}_i}\to 0,i\to\infty,
\end{aligned}\end{equation}
which leads to a contradiction since $\norm[\infty]{\dot x_{\abs{s_2}}}<\infty$. Therefore, Case (a) is impossible.

For Case (b), assume that $\sgn\left(s_1\right)=+1$ and $\sgn\left(s_2\right)=-1$. Then, $\forall i\in\N^*$, it holds that $x_{\abs{s_1}}\left(t_i^{\left(1\right)}\right)=M_{\abs{s_1}}$ and $x_{\abs{s_1}}\left(t_i^{\left( 2\right)}\right)=-M_{\abs{s_1}}$. As $i\to\infty$, $t_i^{\left(1\right)},t_i^{\left(2\right)}\to t_\infty$ implies that $\abs{t_i^{\left(1\right)}-t_i^{\left(2\right)}}\to0$. Hence, $\norm[\infty]{\dot x_{\abs{s_1}}}<\infty$ implies that
\begin{equation}
\begin{aligned}
0<2M_{\abs{s_1}}=&\abs{x_{\abs{s_1}}\left(t_i^{\left(1\right)}\right)-x_{\abs{s_1}}\left(t_i^{\left(2\right)}\right)}\\
\leq&\norm[\infty]{\dot x_{\abs{s_1}}}\abs{t_i^{\left(1\right)}-t_i^{\left(2\right)}}\to 0,\,i\to\infty,
\end{aligned}\end{equation}
which leads to a contradiction. So Case (b) is impossible.
\end{proof}

\begin{remark}
Proposition \ref{prop:UniquenessOfStateConstraintChattering} implies that if $t_0$ is close enough to $t_\infty$, then there exists a unique state constraint allowed active. Consider the chattering constraint $s_1$. Assume that $\forall2\leq r\leq R$, $s_r$ switches at $\left\{t_i^{\left(r\right)}\right\}_{i=1}^{N_r}\subset\left(t_0,t_\infty\right)$. Consider $\hat{t}_0  \in \left(\max_{2\leq r\leq R,1\leq i\leq N_r}t_i^{\left(r\right)},t_\infty\right)$. By BPO, the trajectory between $\hat{t}_0$ and $t_\infty$ is optimal, and only $s_1$ is active during $\left(\hat{t}_0,t_\infty\right)$. Therefore, one can investigate a chattering constraint and get rid of all other constraints in a chattering period.
\end{remark}

In the following, denote $s$ as the chattering constraint in the chattering period $\left[t_0,t_\infty\right]$. Let $\sgn\left(s\right)=+1$. The junction time set $\left\{t_i\right\}_{i=0}^\infty$ increases monotonically and converges to $t_\infty$.

\subsection{State Constraints Able to Induce Chattering}\label{subsec:NecessaryConditionsStateConstraintChattering}

Note that $\forall k>\abs{s}$, $t\in\left(t_0,t_\infty\right)$, $\abs{x_k}<M_k$. By Lemma \ref{lemma:costate}.\ref{lemma:costate_lambdak_polynomial}, $\forall k>\abs{s}$, $\lambda_k\left(t\right)$ is zero or a polynomial of degree at most $\left(n-k\right)$. Without loss of generality, assume that $\forall t\in\left(t_0,t_\infty\right)$, $k>\abs{s}$, $\sgn\left(\lambda_k\left(t\right)\right)\equiv\mathrm{const}$; otherwise, by BPO, one can consider the chattering period $\left(\hat{t}_0,t_\infty\right)$, where $\hat{t}_0$ is defined as $\max\left\{t\in\left(t_0,t_\infty\right):\exists k>\abs{s},\text{ s.t. }\lambda_k\text{ crosses }0\text{ at }t\right\}$.

In this section, Proposition \ref{prop:ChatteringNecessary_nonmonotonical} provides the monotonicity of $\vlambda$ in the chattering period. Proposition \ref{prop:ChatteringPhenomenaNecessaryConditions_1sn} gives some necessary conditions of the chattering constraint. Proposition \ref{prop:chattering_nosystembehavior} proves that constrained arcs do not exist in the chattering period.

\begin{proposition}\label{prop:ChatteringNecessary_nonmonotonical}
Consider the chattering constraint $s$ in problem \eqref{eq:optimalproblem}. $\forall 1\leq k\leq\abs{s}$, $\delta>0$, $\lambda_k$ is non-monotonic in $\left(t_\infty-\delta,t_\infty\right)$. Furthermore, $\forall 1\leq k\leq\abs{s}$, $\exists t_\delta,t_\delta'\in\left(t_\infty-\delta,t_\infty\right)$, s.t. $\lambda_k\left(t_\delta\right)>0$, $\lambda_k\left(t_\delta'\right)<0$.
\end{proposition}

\begin{proof}
Assume that $\exists\delta>0$, $1\leq k\leq\abs{s}$, s.t. $\lambda_k$ is monotonic in $\left(t_\infty-\delta,t_\infty\right)$. Then, $\lambda_1$ has at most $k$ roots. By Lemma \ref{lemma:costate}.\ref{lemma:costate_bangbang}, $u$ switches for at most $k$ times during $\left(t_\infty-\delta,t_\infty\right)$, which contradicts chattering. So $\lambda_k$ is non-monotonic.

Assume that $\exists\delta>0$, $1\leq k\leq\abs{s}$, $\forall t\in\left(t_\infty-\delta,t_\infty\right)$, $\lambda_{k}\left(t\right)\leq0$. Then, $k=1$; otherwise, $\lambda_{k-1}$ is monotonic, which leads to a contradiction. However, $k=1$ implies that $u\leq0$ on $\left(t_\infty-\delta,t_\infty\right)$, which contradicts chattering. Therefore, $\forall 1\leq k\leq\abs{s}$, $\delta>0$, $\exists t_\delta\in\left(t_\infty-\delta,t_\infty\right)$, s.t. $\lambda_k\left(t_\delta\right)>0$; similarly, $\exists t_\delta'\in\left(t_\infty-\delta,t_\infty\right)$, s.t. $\lambda_k\left(t_\delta'\right)<0$.
\end{proof}

\begin{remark}
Proposition \ref{prop:ChatteringNecessary_nonmonotonical} implies that during the chattering period, $\forall\abs{s}<k\leq n$, $\sgn\left(\lambda_k\right)\equiv\mathrm{const}$, while $\forall1\leq k\leq\abs{s}$, $\sgn\left(\lambda_k\right)$ switches between $\pm1$ for infinitely many times.
\end{remark}

\begin{proposition}\label{prop:ChatteringPhenomenaNecessaryConditions_1sn}
In problem \eqref{eq:optimalproblem}, the chattering constraint $s$ satisfies $1<\abs{s}<n$ and $\sgn\left(s\right)\lambda_{\abs{s}+1}<0$ during $\left(t_0,t_\infty\right)$.

\end{proposition}

\begin{proof}
Assume that $\sgn\left(s\right)=+1$. By \eqref{eq:junction_condition}, $\lambda_{\abs{s}}$ jumps decreasingly at junction time. Assume that $\dot{\lambda}_{\abs{s}}\leq0$. Then, $\lambda_{\abs{s}}$ decreases monotonically during $\left(t_0,t_\infty\right)$, which contradicts Proposition \ref{prop:ChatteringNecessary_nonmonotonical}. Therefore, $\dot{\lambda}_{\abs{s}}\leq0$ does not always hold. By \eqref{eq:derivative_costate_lambdan_zero}, $\abs{s}\not=n$. As assumed in this section, $\sgn\left(\lambda_{\abs{s}+1}\right)\equiv\const$ during $\left(t_0,t_\infty\right)$. Hence, $\sgn\left(\lambda_{\abs{s}+1}\right)\equiv-\sgn\left(s\right)=-1$.

Assume that $n>1$ and $s=\overline{1}$. According to Lemma \ref{lemma:costate}.\ref{lemma:lambda1continue}, $\lambda_1$ is continuous on $\left[t_0,t_\infty\right]$ despite the junction condition \eqref{eq:junction_condition}. $\exists i^*\in\N^*$, $x_1\left(t_{i^*}\right)=M_1$, and $x_1\left(t\right)<M_1$ on $t\in\left(t_{i^*},t_{i^*+1}\right)$. Then, Lemma \ref{lemma:costate}.\ref{lemma:costate_sign} implies that $\forall t\in\left(t_{i^*},t_{i^*+1}\right)$, $u\left(t\right)\equiv-M_0$; hence, $\lambda_1>0$ on $t\in\left(t_{i^*},t_{i^*+1}\right)$. Either $x_1\left(t\right)\equiv M_1$ or $x_1\left(t\right)<M_1$ holds on $t\in\left(t_{i^*-1},t_{i^*}\right)$; hence, $\lambda_1\left(t\right)\leq0$ on $t\in\left(t_{i^*-1},t_{i^*}\right)$. By $\lambda_2\left(t_{i^*+1}^+\right)=-\dot\lambda_1\left(t_{i^*+1}^+\right)<0$ and the assumption that $\sgn\left(\lambda_2\right)\equiv\mathrm{const}$, $\forall t\in\left(t_0,t_\infty\right)$, $\lambda_2\left(t\right)<0$. $\lambda_1\left(t_{^*}\right)=0$ implies that $\lambda_1\left(t\right)>0$ on $t\in\left(t_{i^*},t_\infty\right)$, which contradicts Proposition \ref{prop:ChatteringNecessary_nonmonotonical}. Hence, $\abs{s}\not=1$.
\end{proof}

\begin{proposition}\label{prop:chattering_nosystembehavior}
Consider the chattering constraint $s$ in problem \eqref{eq:optimalproblem}. Then, $\forall t\in\Setminus{\left(t_{ 0},t_\infty\right)}{\left\{t_i\right\}_{i= 1}^\infty}$, $ \sgn\left(s\right)x_{\abs{s}}\left(t\right)<M_{\abs{s}}$.
\end{proposition}

\begin{proof}
Evidently, the case where $\exists\hat\delta>0$, $\forall t\in\left[t_\infty-\hat\delta,t_\infty\right]$, $x_{\abs{s}}\left(t\right)\equiv M_{\abs{s}}$, contradicts the chattering phenomenon.

Assume that $\sgn\left(s\right)=+1$. By Proposition \ref{prop:ChatteringPhenomenaNecessaryConditions_1sn}, $1<\abs{s}<n$ and $\lambda_{\abs{s}+1}<0$ during $\left(t_0,t_\infty\right)$. Assume that $x_{\abs{s}}\equiv M_{\abs{s}}$ during $\left[t_1,t_2\right]$; hence, $\forall k<\abs{s}$, $x_k=\frac{\mathrm{d}^{\abs{s}-k}x_{\abs{s}}}{\mathrm{d}t^{\abs{s}-k}}\equiv0$ during $\left[t_0,t_\infty\right]$. By Lemma \ref{lemma:costate}.\ref{lemma:costate_sign}, $\lambda_{\abs{s}}\equiv0$ during $\left(t_1,t_2\right)$.

By \eqref{eq:junction_condition}, $\lambda_{\abs{s}}\left(t_2^+\right) \leq0$. Note that $\dot\lambda_{\abs{s}}=-\lambda_{\abs{s}+1}>0$; hence, during $\left(t_2,t_3\right)$, $\lambda_{\abs{s}}$ has at most one root. Since $\forall1\leq k<\abs{s}$, $\dot{\lambda}_k=-\lambda_{k+1}$, it can be proved recursively that $\lambda_1$ has at most one root during $\left(t_2,t_3\right)$, i.e., $u$ switches for at most one time. Denote $\tau_1$ as the root of $\lambda_1$ if it exists; otherwise, denote $\tau_1=t_3$. Then, $u\equiv u_0$ during $\left(t_2,\tau_1\right)$, while $u\equiv -u_0$ during $\left(\tau_1,t_3\right)$, where $u_0\in\left\{M_0,-M_0\right\}$. Note that $x_{\abs{s}}\left(t_2\right)=x_{\abs{s}}\left(t_3\right)=M_3$, $x_{\abs{s}-1}\left(t_3\right)=0$, and $\forall 1\leq k<\abs{s}$, $x_k\left(t_2\right)=0$. Considering $x_{\abs{s}}$ and $x_{\abs{s}-1}$, Proposition \ref{prop:system_dynamics} implies that

\begin{equation}
\left(t_3-t_2\right)^{k}=2\left(t_3-\tau_1\right)^{k},\,\forall k\in\left\{\abs{s},\abs{s}-1\right\},\end{equation}
which leads to a contradiction. Hence, Proposition \ref{prop:chattering_nosystembehavior} holds.

\end{proof}

\begin{remark}
Propositions \ref{prop:UniquenessOfStateConstraintChattering} and \ref{prop:chattering_nosystembehavior} implies that Figs. \ref{fig:chattering_case}(a-b) are impossible, respectively. In other words, infinite numbers of unconstrained arcs are connected at the unique constrained boundary $\left\{x_{\abs{s}}=M_{\abs{s}}\sgn\left(s\right)\right\}$, while constrained arcs do not exist during the chattering period, as shown in Fig. \ref{fig:chattering_case}(c).

\end{remark}

\subsection{Limiting Behaviors in the Chattering Period}\label{subsec:BehaviorsOfCostatesDuringChatteringPeriod}

This section analyzes the limiting behavior of states and costates in the chattering period. Proposition \ref{prop:Lambda_zero_chattering} provides the switching times of control $u$ between two junction time. Proposition \ref{prop:terminal_behaviors_states_costates_chattering} gives the convergence of $\vx$ and $\vlambda$.

\begin{proposition}\label{prop:Lambda_zero_chattering}
Consider the chattering constraint $s$. Then, $\forall i\in\N^*$, $1\leq k\leq\abs{s}$, $\lambda_k$ has at most $\left(\abs{s}-k+1\right)$ roots on $\left(t_i,t_{i+1}\right)$. So $u$ switches for at most $\abs{s}$ times during $\left(t_i,t_{i+1}\right)$.
\end{proposition}

\begin{proof}
Assume that $\sgn\left(s\right)=+1$. Proposition \ref{prop:ChatteringPhenomenaNecessaryConditions_1sn} implies that $\lambda_{\abs{s}+1}<0$ during $\left(t_0,t_\infty\right)$. By \eqref{eq:junction_condition}, $\forall i\in\N^*$, $\lambda_{\abs{s}}\left(t_i^+\right)\leq\lambda_{\abs{s}}\left(t_i^-\right)$. Then, $\forall i\in\N^*$, $\lambda_{\abs{s}}$ increases monotonically during $\left(t_i,t_{i+1}\right)$ and jumps decreasingly at $t_i$. Therefore, $\lambda_{\abs{s}}$ has at most one root during $\left(t_i,t_{i+1}\right)$.

$\forall 1\leq k<\abs{s}$, considering the monotonicity of $\lambda_k$, it can be proved by \eqref{eq:derivative_costate} recursively that $\lambda_k$ has at most $\left(\abs{s}-k+1\right)$ roots during $\left(t_i,t_{i+1}\right)$. Specifically, $\lambda_1$ has at most $\abs{s}$ roots during $\left(t_i,t_{i+1}\right)$. By Lemma \ref{lemma:costate}.\ref{lemma:costate_bangbang}, $u$ switches for at most $\abs{s}$ times during $\left(t_i,t_{i+1}\right)$. Therefore, Proposition \ref{prop:Lambda_zero_chattering} holds.
\end{proof}

\begin{proposition}[Convergence of $\vx$ and $\vlambda$ to $t_\infty$]\label{prop:terminal_behaviors_states_costates_chattering}
Consider the chattering constraint $s$. $ \forall1\leq k\leq\abs{s}$, it holds that:
\begin{enumerate}
\item\label{prop:terminal_behaviors_states_chattering} $\forall\delta>0$, $\sup_{t\in\left(t_\infty-\delta,t_\infty\right)}\abs{x_k\left(t\right)-x_k\left(t_\infty\right)}=\mathcal{O}(\delta^k)$. In particular, $\forall 1\leq k<\abs{s}$, $\lim_{t\to t_\infty}x_k\left(t\right)=x_k\left(t_\infty\right)=0$. For $\abs{s}$, $\lim_{t\to t_\infty}x_{\abs{s}}\left(t\right)=x_{\abs{s}}\left(t_\infty\right)=M_{\abs{s}}\sgn\left(s\right)$.
\item\label{prop:terminal_behaviors_costates_chattering} $\forall\delta>0$, $\sup_{t\in\left(t_\infty-\delta,t_\infty\right)}\abs{\lambda_k}=\mathcal{O}(\delta^{\abs{s}-k+1})$. Furthermore, $\lim_{t\to t_\infty}\lambda_k\left(t\right)=\lambda_k\left(t_\infty\right)=0$.
\end{enumerate}
\end{proposition}

\begin{proof}
Consider the case where $\sgn\left(s\right)=+1$. Note that $\forall i\in\N^*$, $x_{\abs{s}}\left(t_i\right)=M_{\abs{s}}$. As $i\to\infty$, it holds that $t_i\to t_\infty$, $x_{\abs{s}}\left(t_i\right)\to M_{\abs{s}}$. Since $x_{\abs{s}}$ is continuous, $x_{\abs{s}}\left(t_\infty\right)=M_{\abs{s}}$; hence, $\lim_{t\to t_\infty}x_{\abs{s}}\left(t\right)=M_{\abs{s}}\sgn\left(s\right)$.

Applying Rolle's theorem \cite{stein2009real} recursively, $\forall 1\leq k<\abs{s}$, $\exists\left\{t_i^{\left(k\right)}\right\}_{i=1}^\infty$ increasing monotonically and converging to $t_\infty$, s.t. $\forall i\in\N^*$, $x_k\left(t_i^{\left(k\right)}\right)=0$. The continuity of $x_k$ implies that $\lim_{t\to t_\infty}x_k\left(t\right)=x_k\left(t_\infty\right)=\lim_{i\to\infty}x_k\left(t_i^{\left(k\right)}\right)=0$. Note that $\forall 1\leq k\leq \abs{s}$, $\delta>0$, $\sup_{t\in\left(t_\infty-\delta,t_\infty\right)}\abs{x_k\left(t\right)-x_k\left(t_\infty\right)}\leq\frac{M_0}{k!}\delta^k=\mathcal{O}\left(\delta^k\right)$. Therefore, Proposition \ref{prop:terminal_behaviors_states_costates_chattering}.\ref{prop:terminal_behaviors_states_chattering} holds.

For $\abs{s}$, note that $\forall i\in\N^*$, $\lambda_{\abs{s}}$ increases monotonically during $\left(t_i,t_{i+1}\right)$ and jumps decreasingly at $t_i$. By Proposition \ref{prop:ChatteringNecessary_nonmonotonical}, $\lambda_{\abs{s}}$ crosses 0 for infinitely many times during $\left(t_0,t_\infty\right)$. So $\exists\left\{t_i^{\left(\abs{s}\right)}\right\}_{i=1}^\infty$ increasing monotonically, s.t. $\lim_{i\to\infty}t_i^{\left(\abs{s}\right)}=t_\infty$ and $\forall i\in\N^*$, $\lambda_{\abs{s}}\left(t_i^{\left(\abs{s}\right)}\right)=0$. Denote $t_0^{\left(\abs{s}\right)}=t_0$ and
\begin{equation}
\norm[\infty]{\lambda_{\abs{s}+1}}=\sup_{t\in\left(t_0,t_\infty\right)}\abs{\lambda_{\abs{s}+1}\left(t\right)}<\infty.\end{equation}
Then, $\forall i\in\N^*$, $\lambda_{\abs{s}}\left(t_i^{\left(\abs{s}\right)}\right)=0$; hence, $\forall t\in\left[t_{i-1}^{\left(\abs{s}\right)},t_i^{\left(\abs{s}\right)}\right]$,
\begin{equation}
\begin{aligned}
&\abs{\lambda_{\abs{s}}\left(t\right)}=\abs{\lambda_{\abs{s}}\left(t\right)-\lambda_{\abs{s}}\left(t_i^{\left(\abs{s}\right)}\right)}\\
\leq&\norm[\infty]{\lambda_{\abs{s}+1}}\left(t_i^{\left(\abs{s}\right)}-t\right)\leq\norm[\infty]{\lambda_{\abs{s}+1}}\left(t_\infty-t\right).
\end{aligned}\end{equation}
Therefore, $\forall\delta>0$, $\sup_{t\in\left(t_\infty-\delta,t_\infty\right)}\abs{\lambda_{\abs{s}}}\leq\norm[\infty]{\lambda_{\abs{s}+1}}\delta=\mathcal{O}(\delta)$. Define $\lambda_{\abs{s}}\left(t_\infty\right)=0$. Then, $\varlimsup_{t\to t_\infty}\abs{\lambda_{\abs{s}}\left(t\right)}=0=\lambda_{\abs{s}}\left(t_\infty\right)$. Proposition \ref{prop:terminal_behaviors_states_costates_chattering}.\ref{prop:terminal_behaviors_costates_chattering} holds for $\abs{s}$.

Note that $\forall 1\leq k<\abs{s}$, $\lambda_k$ is continuous. By Proposition \ref{prop:ChatteringNecessary_nonmonotonical}, $\exists\left\{t_i^{\left(k\right)}\right\}_{i=1}^\infty$ increasing monotonically and converging to $t_\infty$, s.t. $\forall i\in\N^*$, $\lambda_k\left(t_i^{\left(k\right)}\right)=0$. Similarly to the analysis for $\abs{s}$, $\forall t\in\left(t_0,t_\infty\right)$, $\abs{\lambda_k\left(t\right)}\leq\frac{\norm[\infty]{\lambda_{\abs{s}+1}}}{\left(\abs{s}-k+1\right)!}\left(t_\infty-t\right)^{\abs{s}-k+1}$. For the same reason, Proposition \ref{prop:terminal_behaviors_states_costates_chattering}.\ref{prop:terminal_behaviors_costates_chattering} holds for $1\leq k<\abs{s}$.
\end{proof}

Under the assumption that chattering occurs in problem \eqref{eq:optimalproblem}, the results of Section \ref{sec:NecessaryConditionsChattering} is summarized in Theorem \ref{thm:ChatteringPhenomenaBehavior} which provides insight into the behavior of states, costates, and control near the limit time.

\section{Chattering in 4th-Order Problems with Velocity Constraints}\label{sec:ChatteringPhenomena4thOrder}
This section proves that chattering phenomena can occur when $n=4$ and $\abs{s}=3$, rectifying a longstanding misconception in the industry concerning the optimality of S-shaped trajectories. In other words, problems \eqref{eq:optimalproblem} of 4th-order with velocity constraints represent problems of the lowest order where chattering phenomena can occur.

Assume that $s=\overline{3}$ in this section. Firstly, problem \eqref{eq:optimalproblem_n4s3_equivalent} is analyzed from the Hamiltonian perspective in Section \ref{subsec:n4s3_equivalent_costate}. Then, Section \ref{subsec:n4s3_equivalent_solution} proves Theorem \ref{thm:Chattering_n4s3_equivalent_optimal_y11}, i.e., solving the chattering optimal control of problem \eqref{eq:optimalproblem_n4s3_equivalent}. Thirdly, Section \ref{subsec:n4s3_optimal_xu} proves Theorem \ref{thm:Chattering_n4s3_equivalent}, i.e., transforming problem \eqref{eq:optimalproblem_n4s3} into the parameter-free infinite-horizon problem \eqref{eq:optimalproblem_n4s3_equivalent}. Finally, some discussions are provided in Section \ref{subsec:n4s3_Discussion}.

\subsection{Costate Analysis of Problem \eqref{eq:optimalproblem_n4s3_equivalent}}\label{subsec:n4s3_equivalent_costate}

To solve problem \eqref{eq:optimalproblem_n4s3_equivalent}, the costate analysis of problem \eqref{eq:optimalproblem_n4s3_equivalent} is performed in this section as preliminaries. Denote $\vp\left(\tau\right)=\left(p_k\left(\tau\right)\right)_{k=1}^3$ as the costate vector, which satisfies $p_0\geq0$ and $\left(p_0,\vp\left(\tau\right)\right)\not= \vzero$. The Hamiltonian of problem \eqref{eq:optimalproblem_n4s3_equivalent} is
\begin{equation}\label{eq:hamilton_n4s3_equivalent}
\begin{aligned}
&\widehat{\hamilton}\left(\vy\left(\tau\right),v\left(\tau\right),p_0,\vp\left(\tau\right),\zeta\left(\tau\right),\tau\right)\\
=&p_0y_3+p_1v+p_2y_1+p_3y_2-\zeta y_3
\end{aligned}\end{equation}
where $\zeta\geq0$, $\zeta y_3=0$. The Hamilton's equations \cite{fox1987introduction} imply that $ \dot{\vp}=-\frac{\partial\widehat{\hamilton}}{\partial \vy}$, i.e.,
\begin{equation}\label{eq:dot_costate_n4s3_equivalent}
\dot{p}_1=-p_2,\,\dot{p}_2=-p_3,\,\dot{p}_3=-p_0+\zeta.\end{equation}
Note that $\frac{\partial \widehat{\hamilton}}{\partial \tau}=0$; hence, $\forall \tau>0$,
\begin{equation}\label{eq:hamiltonian_constant_n4s3_equivalent}
\widehat{\hamilton}\left(\vy\left(\tau\right),v\left(\tau\right),p_0,\vp\left(\tau\right),\zeta\left(\tau\right),\tau\right)\equiv0.\end{equation}

PMP implies that
\begin{equation}
v\left(\tau\right)\in\mathop{\arg\min}\limits_{\abs{V}\leq 1}\widehat\hamilton\left(\vy\left(\tau\right),V,p_0,\vp\left(\tau\right),\zeta\left(\tau\right),\tau\right),\end{equation}
i.e.,
\begin{equation}\label{eq:Chattering_n4s3_equivalent_BangBang_star}
v\left(\tau\right)=-\sgn\left(p_1\left(\tau\right)\right),\text{ if }p_1\left(\tau\right)\not=0.\end{equation}

If the constraint $y_3\geq0$ switches between active and inactive at $\tau_1 >0$, then the junction condition \cite{jacobson1971new} occurs that
\begin{equation}\label{eq:Chattering_n4s3_equivalent_junction}
\exists\mu\geq0,\,p_3\left(\tau_1^+\right)-p_3\left(\tau_1^-\right)=\mu.\end{equation}
During the optimal trajectory, $p_1$ and $p_2$ keep continuous, while $p_3$ can jump at junction time.

\begin{proposition}[Optimal Control's Behavior of Problem \eqref{eq:optimalproblem_n4s3_equivalent}]\label{prop:n4s3_equivalent_costate}
The following properties hold for the optimal control of problem  \eqref{eq:optimalproblem_n4s3_equivalent}.
\begin{enumerate}
\item\label{prop:n4s3_equivalent_costate_p1_0} $p_1\equiv0$ holds for a period if and only if $y_3\equiv0$.
\item\label{prop:n4s3_equivalent_BangBang} $v=-\sgn\left(p_1\right)$ holds a.e. during $\tau>0$. In other words, the bang-bang and singular controls hold as follows:
\begin{equation}\label{eq:Chattering_n4s3_equivalent_BangBang}
v\left(\tau\right)=\begin{dcases}
1,&p_1\left(\tau\right)<0,\\
0,&p_1\left(\tau\right)=0,\\
-1,&p_1\left(\tau\right)>0
\end{dcases}\,\text{a.e.}\end{equation}
\item\label{prop:n4s3_equivalent_costate_unique} Problem \eqref{eq:optimalproblem_n4s3_equivalent} has a unique optimal solution.

\item\label{prop:n4s3_equivalent_costate_polynomial} If $y_3>0$ during $\left(\tau_{ 1},\tau_{ 2}\right)$, then $p_k$ is a polynomial of at most order $\left(4-k\right)$ w.r.t. $\tau$ for $k=1,2,3$. Furthermore, $v$ switches for at most 3 times during $\left(\tau_{ 1},\tau_{ 2}\right)$.
\end{enumerate}
\end{proposition}

\begin{proof}
For Proposition \ref{prop:n4s3_equivalent_costate}.\ref{prop:n4s3_equivalent_costate_p1_0}, assume that during $\left(\tau_1,\tau_2\right)$, $p_1\equiv0$ but $y_3>0$. By \eqref{eq:dot_costate_n4s3_equivalent}, $\vp\equiv\vzero$ since $\zeta\equiv0$. Eq. \eqref{eq:hamiltonian_constant_n4s3_equivalent} implies that $p_0y_3\equiv0$; hence, $p_0=0$, which leads to a contradiction against $\left(p_0,\vp\right)\not= \vzero$. Therefore, if $p_1\equiv0$, then $y_3\equiv0$.

Assume that during $\left(\tau_1,\tau_2\right)$, $y_3\equiv0$, then $\vy\equiv0$ and $v\equiv0$. According to \eqref{eq:Chattering_n4s3_equivalent_BangBang_star}, $p_1\equiv0$. Therefore, Proposition \ref{prop:n4s3_equivalent_costate}.\ref{prop:n4s3_equivalent_costate_p1_0} holds.

Proposition \ref{prop:n4s3_equivalent_costate}.\ref{prop:n4s3_equivalent_costate_p1_0} implies that if $p_1\equiv0$, then $v=\dddot{y}_3\equiv0$. Hence, Eq. \eqref{eq:Chattering_n4s3_equivalent_BangBang} holds a.e. due to \eqref{eq:Chattering_n4s3_equivalent_BangBang_star}. Proposition \ref{prop:n4s3_equivalent_costate}.\ref{prop:n4s3_equivalent_BangBang} holds.

For Proposition \ref{prop:n4s3_equivalent_costate}.\ref{prop:n4s3_equivalent_costate_unique}, assume that $v_1^*$ and $v_2^*$ are both the optimal control of problem \eqref{eq:optimalproblem_n4s3_equivalent}. Note that $\widehat{J}\left[v_1^*\right]=\widehat{J}\left[v_2^*\right]=\widehat{J}\left[v_3^*\right]$, where $v_3^*=\frac{3v_1^*+v_2^*}{4}$; hence, $v_3^*$ is also an optimal control. According to Proposition \ref{prop:n4s3_equivalent_costate}.\ref{prop:n4s3_equivalent_BangBang}, $\forall 1\leq k\leq 3$, $ \nu\left(Q_k\right)=0$ holds, where $\nu$ is the Lebesgue measure on $\R$ and $Q_k\triangleq\left\{\tau>0:v_k^*\left(\tau\right)\not\in\left\{0,\pm1\right\}\right\}$. Denote $P\triangleq\left\{\tau>0:v_1^*\left(\tau\right)\not=v_2^*\left(\tau\right)\right\}$. Then, $\forall \tau\in \Setminus{P}{\left(Q_1\cup Q_2\right)}$, $v_3^*\left(\tau\right)\not\in\left\{0,\pm1\right\}$; hence, $\Setminus{P}{\left(Q_1\cup Q_2\right)}\subset Q_3$. Therefore,
\begin{equation}
\begin{aligned}
&0\leq\nu\left(P\right)=\nu\left(P\right)-\nu\left(Q_1\right)-\nu\left(Q_2\right)\\
\leq&\nu\left(\Setminus{P}{\left(Q_1\cup Q_2\right)}\right)\leq\nu\left(Q_3\right)=0.
\end{aligned}\end{equation}
Hence, $ \nu\left(P\right)=0$, i.e., $v_1^*=v_2^*$ a.e. Proposition \ref{prop:n4s3_equivalent_costate}.\ref{prop:n4s3_equivalent_costate_unique} holds.

If $y_3>0$ during $\left(\tau_{1},\tau_{2}\right)$, then $\zeta\equiv0$; hence, Proposition \ref{prop:n4s3_equivalent_costate}.\ref{prop:n4s3_equivalent_costate_polynomial} holds evidently due to \eqref{eq:dot_costate_n4s3_equivalent}.
\end{proof}

\subsection{Optimal Solution of Problem \eqref{eq:optimalproblem_n4s3_equivalent}}\label{subsec:n4s3_equivalent_solution}
Inspired by Fuller's problem \cite{fuller1963study} and Robbins' problem \cite{robbins1980junction}, this section solves problem \eqref{eq:optimalproblem_n4s3_equivalent} through 3 steps. Firstly, Proposition \ref{prop:n4s3_equivalent_recursive} proves the existence of chattering and provides a homogenous relationship w.r.t. controls and states. Then, the chattering trajectory is characterized by Propositions \ref{prop:n4s3_equivalent_recursive_final} and \ref{prop:n4s3_equivalent_beta}. Finally, the optimal solution of problem \eqref{eq:optimalproblem_n4s3_equivalent} is provided in Theorem \ref{thm:Chattering_n4s3_equivalent_optimal_y11}.

To solve problem \eqref{eq:optimalproblem_n4s3_equivalent} recursively, denote $\widehat{J}\left[v; \alpha\right]$ as the objective value of problem \eqref{eq:optimalproblem_n4s3_equivalent} with the initial state $ \alpha\ve_1$ and control $v$. Let $\widehat{J}^*\left( \alpha\right)\triangleq\inf_{v}\widehat{J}\left[v; \alpha\right]$. In other words, the optimal value of the original problem \eqref{eq:optimalproblem_n4s3_equivalent} is $\widehat{J}^*=\widehat{J}^*\left(1\right)$. Denote the optimal control of problem \eqref{eq:optimalproblem_n4s3_equivalent} with initial state vector $ \alpha\ve_1$ as $v^*\left(\tau; \alpha\right)$, where the optimal trajectory is $\vy^*\left(\tau; \alpha\right)$. A homogenous relationship w.r.t. controls and states is provided in Proposition \ref{prop:n4s3_equivalent_recursive}, where the assumption $\tau_\infty<\infty$ in Theorem \ref{thm:Chattering_n4s3_equivalent} is achieved.

\begin{proposition}\label{prop:n4s3_equivalent_recursive}
$\forall \alpha>0$, the following conclusions hold:
\begin{enumerate}
\item\label{prop:n4s3_equivalent_recursive_a} $v\left(\tau\right)$ with $y_k\left(\tau\right)$, $k=1,2,3$, is feasible under the initial state vector $\ve_1$, if and only if $v' \left(\tau\right)=v\left(\frac{\tau}{\alpha}\right)$ with $y_k' \left(\tau\right)=\alpha^k y_k\left(\frac{\tau}{\alpha}\right)$, $k=1,2,3$, is feasible under the initial state vector $\alpha\ve_1$. Furthermore, $\widehat{J}\left[v';\alpha\right]=\alpha^4 \widehat{J}\left[v\right]$.
\item\label{prop:n4s3_equivalent_recursive_a_optimal} For the optimal solution, it holds that $\widehat{J}^*\left(\alpha\right)=\alpha^4\widehat{J}^*$, $v^*\left(\tau;\alpha\right)=v^*\left(\frac{\tau}{\alpha}\right)$, and $y_k^*\left(\tau;\alpha\right)=\alpha^k y_k^*\left(\frac{\tau}{\alpha}\right)$, $k=1,2,3$.
\end{enumerate}
\end{proposition}

\begin{proof}
For Proposition \ref{prop:n4s3_equivalent_recursive}.\ref{prop:n4s3_equivalent_recursive_a}, assume that $v=v\left(\tau\right)$ with $y_k=y_k\left(\tau\right)$ is feasible under the initial state vector $\ve_1$. Let $v'=v\left(\frac{\tau}{\alpha}\right)$ and $y_k'=\alpha^k y_k\left(\frac{\tau}{\alpha}\right)$, $k=1,2,3$. Then, $\forall k=2,3$, $\dot{y}_k'=y_{k-1}'$, and $\dot{y}_1'=v'$. Evidently, $y_3'\geq0$ and $\abs{v'}\leq1$ hold. Therefore, $v'$ with $\vy'$ is feasible under the initial state vector $\alpha\ve_1$. Furthermore,
\begin{equation}
\begin{aligned}
&\widehat{J}\left[v';\alpha\right]=\int_{0}^{\infty}x_3'\left(\frac{\tau}{\alpha}\right)\,\mathrm{d}\tau=\int_{0}^{\infty}\alpha^3x_3\left(\frac{\tau}{\alpha}\right)\,\mathrm{d}\tau\\
=&\alpha^4\int_{0}^{\infty}x_3\left(\tau\right)\,\mathrm{d}\tau=\alpha^4\widehat{J}\left[v\right].
\end{aligned}\end{equation}
The necessity of Proposition \ref{prop:n4s3_equivalent_recursive}.\ref{prop:n4s3_equivalent_recursive_a} holds. Similarly, the sufficiency of Proposition \ref{prop:n4s3_equivalent_recursive}.\ref{prop:n4s3_equivalent_recursive_a} holds. So Proposition \ref{prop:n4s3_equivalent_recursive}.\ref{prop:n4s3_equivalent_recursive_a} holds.

Therefore, $\widehat{J}^*=\widehat{J}\left[v^*\left(\tau\right)\right]=\alpha^{-4}\widehat{J}\left[v^*\left(\frac{\tau}{\alpha}\right);\alpha\right]\leq \alpha^{-4}\widehat{J}^*\left(\alpha\right)$. Similarly, $\widehat{J}^*\left(\alpha\right)\leq \alpha^4\widehat{J}^*$, i.e., $\widehat{J}^*\left(\alpha\right)=\alpha^4\widehat{J}^*$. By Proposition \ref{prop:n4s3_equivalent_costate}.\ref{prop:n4s3_equivalent_costate_unique}, $v^*\left(\tau;\alpha\right)=v^*\left(\frac{\tau}{\alpha}\right)$ is the unique optimal control of problem \eqref{eq:optimalproblem_n4s3_equivalent} with the initial states $\alpha\ve_1$, corresponding to $y_k^*\left(\tau;\alpha\right)=\alpha^ky_k^*\left(\frac{\tau}{\alpha}\right)$. So Proposition \ref{prop:n4s3_equivalent_recursive}.\ref{prop:n4s3_equivalent_recursive_a_optimal} holds.
\end{proof}

For the optimal solution of problem \eqref{eq:optimalproblem_n4s3_equivalent}, denote $\tau_0=0$ and $\forall i\in\N$, $\tau_{i+1}=\arg\min\left\{\tau>\tau_i:y_3\left(\tau\right)=0\right\}$. Then, $\left\{\tau_i\right\}_{i=0}^\infty$ increases monotonically. Denote $\tau_\infty\triangleq\lim_{i\to\infty}\tau_i\in\overline{\R}_{++}$ and $\forall i\in\N$, $\vy_i\triangleq \vy\left(\tau_i\right)=\left(y_{i,1},0,0\right)$ where $y_{i,1}\geq0$. The optimal solution of problem \eqref{eq:optimalproblem_n4s3_equivalent} can be in the following forms. (a) $\exists N\in\N^*$, $y_{N,1}=0$, but $y_{N-1,1}>0$. In this case, $\vy\equiv0$ on $\left(\tau_N,\infty\right)$. In other words, $\forall i\geq N$, $y_{i,1}=0$ and $\tau_{i}=\tau_N$. (b) $\forall i\in\N$, $y_{i,1}>0$. In this case, if $\tau_\infty<\infty$, then a chattering phenomenon occurs. If $\tau_\infty=\infty$, then unconstrained arcs are connected by $y_3=0$ and extend to infinity. Based on the homogenous relationship in Proposition \ref{prop:n4s3_equivalent_recursive}, the existence of chattering in problem \eqref{eq:optimalproblem_n4s3_equivalent} is provided in Proposition \ref{prop:n4s3_equivalent_recursive_final}.

\begin{proposition}\label{prop:n4s3_equivalent_recursive_final}
For the optimal solution of problem \eqref{eq:optimalproblem_n4s3_equivalent}, $ \exists0<\alpha<1$ s.t. $\forall i\in\N$, $y_{i,1}= \alpha^i$. Furthermore, $p_0>0$, $\tau_{\infty}=\frac{\tau_1}{1- \alpha} <\infty$, and $\widehat{J}=\frac{\widehat{J}_1}{1- \alpha^4}$, where $\widehat{J}_1\triangleq\int_{0}^{\tau_1}x_3\left(\tau\right)\,\mathrm{d}\tau$.
\end{proposition}

\begin{proof}
Let $\alpha\triangleq y_{1,1}\geq0$. Assume that $ \alpha=0$. In other words, $y_3^*\left(\tau\right)>0$ on $\left(\tau_0,\tau_1\right)$, and $\vy\left(\tau\right)\equiv\vzero$ on $\tau\geq\tau_1$. According to Proposition \ref{prop:n4s3_equivalent_costate}.\ref{prop:n4s3_equivalent_costate_p1_0}, $\vp\equiv0$ for $\tau\in\left(\tau_1,\infty\right)$. The continuity of $p_1$ and $p_2$ implies that $p_1\left(\tau_1\right)=0$ and $ p_2\left(\tau_1\right)=-\dot{p}_1\left(\tau_1\right)=0$. By Proposition \ref{prop:n4s3_equivalent_costate}.\ref{prop:n4s3_equivalent_costate_polynomial}, $p_1$ has at most one root on $\left(0,\tau_1\right)$; hence, $v$ switches for at most one time on $\left(\tau_0,\tau_1\right)$. Assume that $v\left(\tau\right)=v_0$ for $\tau_0<\tau<\tau'$, and $v\left(\tau\right)=-v_0$ for $\tau'<\tau<\tau_1$, where $v_0\in\left\{\pm1\right\}$ and $\tau_0<\tau'\leq\tau_1$. Then,
\begin{equation}\label{eq:Chattering_n4s3_equivalent_recursive_y11is0}
\begin{dcases}
1+v_0\left(\left(\tau_1-\tau_0\right)-2\left(\tau_1-\tau'\right)\right)=0,\\
\left(\tau_1-\tau_0\right)+\frac{v_0}{2}\left(\left(\tau_1-\tau_0\right)^2-2\left(\tau_1-\tau'\right)^2\right)=0,\\
\frac12\left(\tau_1-\tau_0\right)^2+\frac{v_0}{6}\left(\left(\tau_1-\tau_0\right)^3-2\left(\tau_1-\tau'\right)^3\right)=0.
\end{dcases}\end{equation}
However, Eq. \eqref{eq:Chattering_n4s3_equivalent_recursive_y11is0} has no feasible solution. Therefore, $ \alpha>0$.

By BPO and Proposition \ref{prop:n4s3_equivalent_recursive}, $\forall\tau\geq0$, $\vy\left(\tau_1+\tau\right)=\vy\left(\tau;\alpha\right)$. In other words, $\forall i\in\N$, $y_{i,1}= \alpha y_{i-1,1}$, i.e., $y_{i,1}= \alpha^i$. Since $\widehat{J}=\widehat{J}^*\left( \alpha\right)+\widehat{J}_1$, it holds that $\widehat{J}=\frac{\widehat{J}_1}{1- \alpha^4}$. Therefore, $\widehat{J}$ implies that $0<\alpha<1$. Note that $\tau_i-\tau_{i-1}=y_{i,1}\tau_1$; hence, $\tau_i=\frac{1- \alpha^i}{1- \alpha}\tau_1$ and $\tau_\infty=\lim_{i\to\infty}\tau_i=\frac{\tau_1}{1- \alpha}$ hold.

To achieve optimality, $\vy\equiv\vzero$ during $\left(\tau_\infty,\infty\right)$. By Proposition \ref{prop:n4s3_equivalent_costate}.\ref{prop:n4s3_equivalent_costate_p1_0}, $\vp\equiv\vzero$ during $\left(\tau_\infty,\infty\right)$. Specifically, $p_0>0$ since $\left(p_0,\vp\right)\not=\vzero$. Therefore, Proposition \ref{prop:n4s3_equivalent_recursive_final} holds.

\end{proof}

Proposition \ref{prop:n4s3_equivalent_recursive_final} proves the existence of chattering in the optimal solution of problem \eqref{eq:optimalproblem_n4s3_equivalent}, where $0<\alpha<1$ is defined as the \textbf{chattering attenuation rate}. To solve $\alpha$, Proposition \ref{prop:n4s3_equivalent_beta} investigates the optimal control between $\left(\tau_{i-1},\tau_i\right)$.

\begin{proposition}\label{prop:n4s3_equivalent_beta}
$\forall i\in\N^*$, $v$ switches for 2 times on $\left(\tau_{i-1},\tau_i\right)$.

\end{proposition}

\begin{proof}
By Proposition \ref{prop:n4s3_equivalent_costate}, $v$ switches for at most 3 times on $\left(\tau_{i-1},\tau_i\right)$. Assume that $v$ switches for 3 times on $\left(\tau_0,\tau_1\right)$. Denote the switching time as $\beta_k'\tau_1$, $k=1,2,3$, and $0<\beta_1'<\beta_2'<\beta_3'<1$. By Proposition \ref{prop:n4s3_equivalent_recursive}, $\forall i\in\N^*$, $v$ switches for 3 times on $\left(\tau_{i-1},\tau_i\right)$, where the switching time is $\tau_{i-1}+\beta_k'\left(\tau_i-\tau_{i-1}\right)$, $k=1,2,3$. According to \eqref{eq:Chattering_n4s3_equivalent_BangBang} and \eqref{eq:dot_costate_n4s3_equivalent}, $\forall i\in\N^*$, $\tau\in\left(\tau_{i-1},\tau_i\right)$, $p_1\left(\tau\right)=-\frac{p_0}{6}\prod_{k=1}^{3}\left(\tau-\tau_{i-1}-\beta_k'\left(\tau_i-\tau_{i-1}\right)\right)$. Then, $p_0>0$ implies that $p_1\left(\tau_i^+\right)<0<p_1\left(\tau_i^-\right)$. Since $p_1$ is continuous, $p_1\left(\tau_i\right)=0$. Hence, $p_1\left(\tau_{i-1}\right)=0$. In other words, $p_1$ has at least 5 roots on $\left[\tau_{i-1},\tau_i\right]$, which contradicts Proposition \ref{prop:n4s3_equivalent_costate}.\ref{prop:n4s3_equivalent_costate_polynomial}. Therefore, $v$ switches for at most 2 times on $\left(\tau_{i-1},\tau_i\right)$.

Assume that $v$ switches for at most one time on $\left(\tau_0,\tau_1\right)$. Then, $\exists 0<\tau'\leq\tau_1$, s.t. $v=v_0$ on $\left(0,\tau_1-\tau'\right)$ and $v=-v_0$ on $\left(\tau_1-\tau',\tau_1\right)$, where $v_0\in\left\{\pm1\right\}$. Then, it holds that
\begin{equation}\label{eq:Chattering_n4s3_equivalent_switch1}
\begin{dcases}
1+v_0\left(\tau_1-2\tau'\right)=\alpha,\\
\tau_1+\frac{v_0}{2}\left(\tau_1^2-2\tau'^2\right)=0,\\
\frac12\tau_1^2+\frac{v_0}{6}\left(\tau_1^3-2\tau'^3\right)=0.
\end{dcases}\end{equation}
Eq. \eqref{eq:Chattering_n4s3_equivalent_switch1} implies $\tau_1=\tau'=0$, $\alpha=1$, which contradicts Proposition \ref{prop:n4s3_equivalent_recursive_final}. So $v$ switches for 2 times on $\left(\tau_{i-1},\tau_i\right)$.
\end{proof}

Finally, Theorem \ref{thm:Chattering_n4s3_equivalent_optimal_y11} is proved as follows.

\begin{proof}[Proof of Theorem \ref{thm:Chattering_n4s3_equivalent_optimal_y11}]
According to Proposition \ref{prop:n4s3_equivalent_beta}, $\exists0<\beta_1<\beta_2<1$, s.t. $\forall i\in\N^*$, $v$ switches at $\left(\left(1-\beta_k\right)\tau_{i-1}+\beta_k\tau_i\right)$, $k=1,2$. By Proposition \ref{prop:n4s3_equivalent_costate}.\ref{prop:n4s3_equivalent_costate_polynomial}, $p_1$ is a 3rd-order polynomial. By \eqref{eq:dot_costate_n4s3_equivalent}, $\forall i\in\N^*$, $\exists\beta_3^{\left(i\right)}\not\in\left(0,1\right)$, s.t. $\forall\tau\in\left(\tau_{i-1},\tau_i\right)$, $p_1$ is
\begin{equation}\label{eq:Chattering_n4s3_equivalent_optimal_p1i}
p_{1,i}\left(\tau\right)=-\frac{p_0}{6}\prod_{k=1}^{3}\left(\tau-\left(1-\beta_k^{\left(i\right)}\right)\tau_{i-1}-\beta_k^{\left(i\right)}\tau_i\right),\end{equation}
where $\beta_1^{\left(i\right)}=\beta_1$ and $\beta_2^{\left(i\right)}=\beta_2$. Denote $\mu_i\triangleq p_3\left(\tau_i^+\right)-p_3\left(\tau_i^-\right)\geq0$. By \eqref{eq:dot_costate_n4s3_equivalent}, $\forall\tau\in\left(\tau_i,\tau_{i+1}\right)$,
\begin{equation}\label{eq:Chattering_n4s3_equivalent_optimal_p1i1_p1i}
p_{1,i+1}\left(\tau\right)-p_{1,i}\left(\tau\right)=\frac{\mu_i}{2}\left(\tau-\tau_{i-1}\right)^2.\end{equation}
Compare the coefficients of $1$ and $\tau$ in \eqref{eq:Chattering_n4s3_equivalent_optimal_p1i1_p1i}, it holds that
\begin{equation}\label{eq:Chattering_n4s3_equivalent_optimal_beta3i1beta3i}
\begin{dcases}
2\sum_{k=1}^{3}\beta_k^{\left(i\right)}+\alpha^2\sum_{j<k}\beta_j^{\left(i+1\right)}\beta_k^{\left(i+1\right)}-\sum_{j<k}\beta_j^{\left(i\right)}\beta_k^{\left(i\right)}=3,\\
\sum_{k=1}^{3}\beta_k^{\left(i\right)}-\sum_{j<k}\beta_j^{\left(i\right)}\beta_k^{\left(i\right)}-\beta_1\beta_2\left(\beta_3^{\left(i\right)}-\alpha^3\beta_3^{\left(i+1\right)}\right)=1.
\end{dcases}\end{equation}
Eliminate $\beta_3^{\left(i+1\right)}$ in \eqref{eq:Chattering_n4s3_equivalent_optimal_beta3i1beta3i}. $\forall i\in\N^*$, $\beta_3^{\left(i\right)}=\frac{f_2\left(\beta_1,\beta_2,\alpha\right)}{f_1\left(\beta_1,\beta_2,\alpha\right)}$, where $f_1\left(\beta_1,\beta_2,\alpha\right)=\sum_{k=1}^{2}\beta_k\left(1-\beta_k\right)\left(1-\beta_{3-k}\left(1-\alpha\right)\right)>0$. Therefore, $\beta_3^{\left(i\right)}$ is independent of $i$. Denote $\beta_3^{\left(i\right)}=\beta_3$, $\forall i\in\N^*$. Then, Eq. \eqref{eq:Chattering_n4s3_equivalent_optimal_beta3i1beta3i} implies \eqref{eq:Chattering_n4s3_equivalent_optimal_y11_equationsystem_coeff1} and \eqref{eq:Chattering_n4s3_equivalent_optimal_y11_equationsystem_coeff0}.

According to Proposition \ref{prop:n4s3_equivalent_costate}.\ref{prop:n4s3_equivalent_BangBang}, Eq. \eqref{eq:Chattering_n4s3_equivalent_optimal_p1i} implies that $\exists v_0\in\left\{\pm1\right\}$, s.t. $\forall i\in\N^*$, $\beta\in\left(0,1\right)$, the optimal control is
\begin{equation}
v\left(\left(1-\beta\right)\tau_{i-1}+\beta\tau_i\right)=\begin{dcases}
v_0,&\beta\in\left(0,\beta_1\right),\\
-v_0,&\beta\in\left(\beta_1,\beta_2\right),\\
v_0,&\beta\in\left(\beta_2,1\right).
\end{dcases}\end{equation}

Note that $\vy\left(0\right)=\ve_{1}$ and $\vy\left(\tau_1\right)=\alpha\ve_{1}$; hence, it holds that
\begin{equation}\label{eq:Chattering_n4s3_equivalent_optimal_y0y1}
\begin{dcases}
1+v_0\left(1-2\left(1-\beta_1\right)+2\left(1-\beta_2\right)\right)\tau_1=\alpha,\\
\tau_1+\frac{v_0}{2}\left(1-2\left(1-\beta_1\right)^2+2\left(1-\beta_2\right)^2\right)\tau_1^2=0,\\
\frac{\tau_1^2}{2}+\frac{v_0}{6}\left(1-2\left(1-\beta_1\right)^3+2\left(1-\beta_2\right)^3\right)\tau_1^3=0.
\end{dcases}\end{equation}
Eliminate $v_0$ and $\tau_1$ in \eqref{eq:Chattering_n4s3_equivalent_optimal_y0y1}. Then, it holds that
\begin{equation}\label{eq:Chattering_n4s3_equivalent_optimal_y11_equationsystem_y1y2y3_notau}
\begin{dcases}
4\beta_1^3-6\beta_1^2-4\beta_2^3+6\beta_2^2-1=0,\\
\left(2\beta_1^2-2\beta_2^2+1\right)\left(\alpha-1\right)-\left(4\beta_1-4\beta_2+2\right)\alpha=0.
\end{dcases}\end{equation}
Through solving \eqref{eq:Chattering_n4s3_equivalent_optimal_y11_equationsystem_coeff1}, \eqref{eq:Chattering_n4s3_equivalent_optimal_y11_equationsystem_coeff0}, and \eqref{eq:Chattering_n4s3_equivalent_optimal_y11_equationsystem_y1y2y3_notau}, the unique feasible solution for $\left(\alpha,\beta_1,\beta_2,\beta_3\right)$ is obtained, as shown in \eqref{eq:Chattering_n4s3_equivalent_optimal_solution}. Then, $v_0=-\sgn\left(p_1\left(0^+\right)\right)=-1$. Therefore, Eq. \eqref{eq:Chattering_n4s3_equivalent_optimal_y0y1} implies \eqref{eq:Chattering_n4s3_equivalent_optimal_y11_equationsystem_x1}, \eqref{eq:Chattering_n4s3_equivalent_optimal_y11_equationsystem_x2}, and \eqref{eq:Chattering_n4s3_equivalent_optimal_y11_equationsystem_x3}. The solution for $\tau_1$ can be solved by \eqref{eq:Chattering_n4s3_equivalent_optimal_y11_equationsystem_x1} and the value of $\alpha$, $\beta_1$, $\beta_2$, and $\beta_3$ in \eqref{eq:Chattering_n4s3_equivalent_optimal_solution}. $\tau_\infty$ can be solved by Proposition \ref{prop:n4s3_equivalent_recursive_final}. Furthermore, it can be solved that $\mu_i\approx1.4494594p_0\alpha^{3i-3}>0$. Hence, Theorem \ref{thm:Chattering_n4s3_equivalent_optimal_y11} holds.
\end{proof}

Theorem \ref{thm:Chattering_n4s3_equivalent_optimal_y11} provides a fully analytical expression for the optimal solution of problem \eqref{eq:optimalproblem_n4s3_equivalent}, as shown in Fig. \ref{fig:chattering_n4s3_optimal_equivalent}. In Fig. \ref{fig:chattering_n4s3_optimal_equivalent}(a-b), the state vector $\vy^*$, the costate vector $\vp^*$, and the control $v^*$ chatter with limit time $\tau_\infty^*$. To further examine the trajectory approaching $\tau_\infty^*$, the time axes in Fig. \ref{fig:chattering_n4s3_optimal_equivalent}(c-d) are in logarithmic scales, while the amplitudes of $\vy$ and $\vp$ are multiplied by some certain compensation factors. Then, it can be observed from Fig. \ref{fig:chattering_n4s3_optimal_equivalent}(c-d) that $\vy$, $v$, and $\vp$ all exhibit strict periodicity, which can be reasoned by Proposition \ref{prop:n4s3_equivalent_recursive_final} and \eqref{eq:Chattering_n4s3_equivalent_optimal_p1}, respectively.

\begin{figure}[!t]
\centering
\includegraphics[width=\columnwidth]{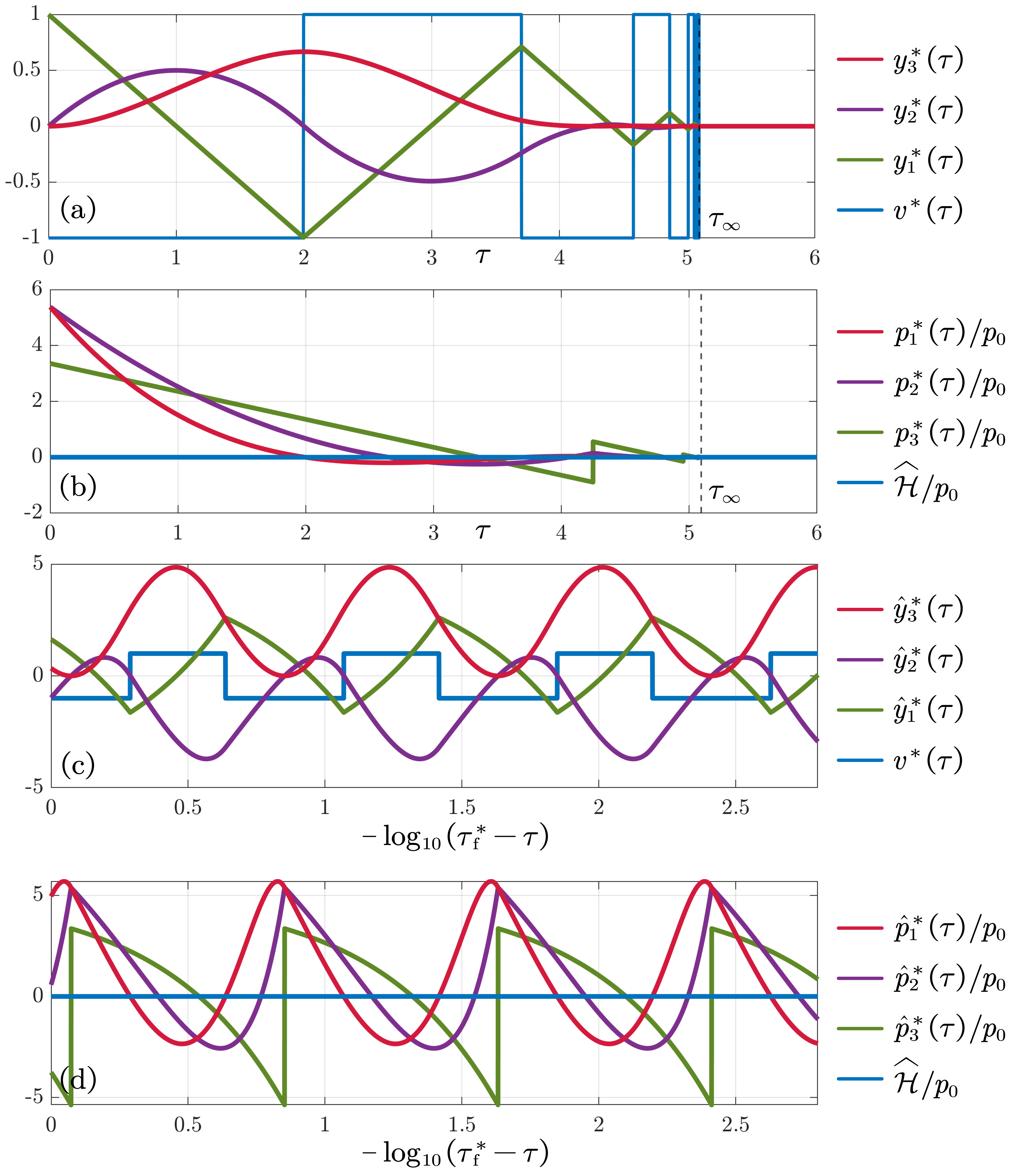}
\caption{Optimal solution of problem \eqref{eq:optimalproblem_n4s3_equivalent}. (a) The optimal trajectory $\vy^*\left(\tau\right)$ and the optimal control $v^*\left(\tau\right)$. (b) The optimal costate vector $\frac{1}{p_0}\vp^*\left(\tau\right)$. (c-d) Enlargement of (a-b) during the chattering period. The abscissa is in logarithmic scale with respect to time, i.e., $-\log_{10}\left(\tau_{\infty}^*-\tau\right)$. $\forall k=1,2,3$, $\hat{y}_k^*\left(\tau\right)=y_k^*\left(\tau\right)\left(1-\frac{\tau}{\tau_{\infty}^*}\right)^{-k}$, and $\hat{p}_k^*\left(\tau\right)=p_k^*\left(\tau\right)\left(1-\frac{\tau}{\tau_{\infty}^*}\right)^{k-4}$.}
\label{fig:chattering_n4s3_optimal_equivalent}
\end{figure}

\begin{remark}
The optimality of the solved $\alpha^*$ in \eqref{eq:Chattering_n4s3_equivalent_optimal_solution} can be verified in another way. $\forall0\leq\alpha<1$, let $y_{i,1}=\alpha^i$, and solve the control $v$ by \eqref{eq:Chattering_n4s3_equivalent_optimal_y0y1}, where $\vy$ reaches $\alpha\ve_1$ at $\tau_1$. Then, the trajectory has a similar homogenous structure to Proposition \ref{prop:n4s3_equivalent_recursive}. Denote $\widehat{J}_1\left(\alpha\right)=\int_{0}^{\tau_1}y_3\left(\tau\right)\,\mathrm{d}\tau$. Then, $\widehat{J}\left(\alpha\right)=\int_{0}^{\infty}y_3\left(\tau\right)\,\mathrm{d}\tau=\frac{\widehat{J}_1\left(\alpha\right)}{1-\alpha^4}$. As shown in Fig. \ref{fig:chattering_Jalpha_equivalent}, $\alpha^*$ in \eqref{eq:Chattering_n4s3_equivalent_optimal_solution} achieves a minimal cost $\widehat{J}\left(\alpha ^*\right)$. The minimal cost supports the optimality of the reasoned $\alpha^*$ once again.

Our previous work \cite{wang2024time} proposes a greedy-and-conservative suboptimal method called MIM. If MIM is applied to problem \eqref{eq:optimalproblem_n4s3}, the corresponding $\vy$ in problem \eqref{eq:optimalproblem_n4s3_equivalent} first moves to $\vzero$ as fast as possible, and then moves along $\vy\equiv\vzero$. In other words, MIM achieves a cost of $\widehat{J}\left(0\right)$ in problem \eqref{eq:optimalproblem_n4s3_equivalent}. Specifically,
\begin{equation}\label{eq:Chattering_n4s3_equivalent_optimal_Jalpha}
\widehat{J}^*=\widehat{J}\left(\alpha^*\right)\approx1.3452202,\,\widehat{J}\left(0\right)\approx1.3467626.\end{equation}
Hence, the relative error between the two trajectoryies is $\frac{\widehat{J}\left(0\right)-\widehat{J}^*}{\widehat{J}^*}\approx0.11\%$. It is the minute discrepancy that leads to the longstanding oversight of the chattering phenomenon in problem \eqref{eq:optimalproblem}, despite its universal applications in the industry.
\end{remark}

\begin{figure}[!t]
\centering
\includegraphics[width=\columnwidth]{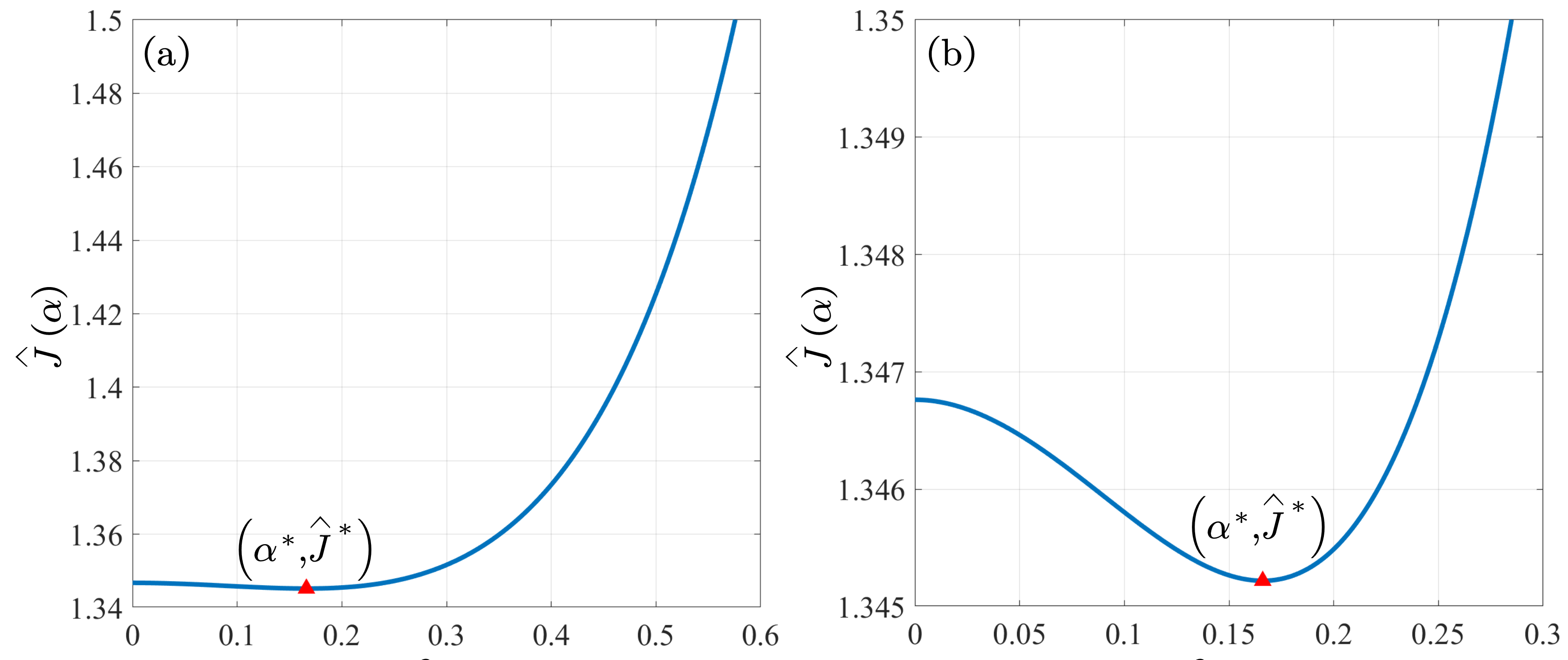}
\caption{Loss function $\widehat{J}\left(\alpha\right)$ when choosing different chattering attenuation rate $\alpha$ in problem \eqref{eq:optimalproblem_n4s3_equivalent}. (a) and (b) are in different scales.}
\label{fig:chattering_Jalpha_equivalent}
\end{figure}

\subsection{Optimal Solution of Problem \eqref{eq:optimalproblem_n4s3}}\label{subsec:n4s3_optimal_xu}

Consider problem \eqref{eq:optimalproblem_n4s3}. By BPO, if $\exists \hat{t}\in\left[0,\tf\right]$, $\vx\left(\hat{t}\right)=\left(0,0,M_3,x_{4}\left(\hat{t}\right)\right)$ and $x_{4}\left(\hat{t}\right)\leq x_{4}\left(\tf\right)$, then $\forall t\in\left[\hat{t},\tf\right]$, $x_3\equiv M_3$. Therefore, the performance of a trajectory depends on the part before $\vx$ enters $\left\{x_3\equiv M_3\right\}$. From this inspiration, this section solves problem \eqref{eq:optimalproblem_n4s3} through proving Theorem \ref{thm:Chattering_n4s3_equivalent}.

\begin{proof}[Proof of Theorem \ref{thm:Chattering_n4s3_equivalent}]
Denote the optimal solution of problem \eqref{eq:optimalproblem_n4s3} as $\vx^*\left(t\right)$ and $u^*\left(t\right)$, $t\in\left[0,t_\f^*\right]$. Denote the solution induced by \eqref{eq:optimalsolution_equivalent_n4s3} as $\hat\vx\left(t\right)$ and $\hat{u}\left(t\right)$, $t\in\left[0,\hat{t}_\f\right]$. By \eqref{eq:optimalsolution_equivalent_n4s3_condition}, \eqref{eq:optimalsolution_equivalent_n4s3_tf} and \eqref{eq:optimalsolution_equivalent_n4s3_x4}, it holds that $\hat\vx\left(\hat{t}_\f\right)=\vxf$. Evidently, $\hat\vx\left(t\right)$ and $\hat{u}\left(t\right)$ are feasible in problem \eqref{eq:optimalproblem_n4s3}. Hence, $t_\f^*\leq\hat{t}_\f$ holds. Let

\begin{equation}\label{eq:optimalsolution_equivalent_n4s3_t2tau}
\hat{y}_3\left(\tau\right)=\begin{dcases}
\frac{M_0^2}{x_{0,1}^3}\left(x_3^*\left(-\frac{x_{0,1}}{M_0}\tau\right)-M_3\right),&\tau\leq-\frac{M_0}{x_{0,1}}t_\f^*,\\
0,&\tau>-\frac{M_0}{x_{0,1}}t_\f^*.
\end{dcases}\end{equation}
Then, the trajectory $\hat\vy\left(\tau\right)$ represented by $\hat{y}_3\left(\tau\right)$ is a feasible solution of problem \eqref{eq:optimalproblem_n4s3_equivalent}. Note that $\int_{0}^{\infty}\hat{y}_3\left(\tau\right)\mathrm{d}\tau=\frac{M_0^3}{x_{0,1}^4}\int_{0}^{t_\f^*}\left(M_3-x_3^*\left(t\right)\right)\mathrm{d}t=M_3t_\f^*-x_{\f4}+x_{0,4}$. Similarly, $\int_{0}^{\infty}y_3^*\left(\tau\right)\mathrm{d}\tau=M_3\hat{t}_\f-x_{\f4}+x_{0,4}$. Then, $t_\f^*\geq\hat{t}_\f$ can be reasoned by $\int_{0}^{\infty}\hat{y}_3\left(\tau\right)\mathrm{d}\tau\geq\int_{0}^{\infty}y_3^*\left(\tau\right)\mathrm{d}\tau$. Therefore, $t_\f^*=\hat{t}_\f$.

By Lemma \ref{lemma:costate}.\ref{lemma:costate_uniqueoptimalcontrol}, $\vx^*\left(t\right)=\hat{\vx}\left(t\right)$ and $u^*\left(t\right)=\hat{u}\left(t\right)$ a.e. Therefore, the solution \eqref{eq:optimalsolution_equivalent_n4s3} is optimal in problem \eqref{eq:optimalproblem_n4s3}.

\end{proof}

\begin{remark}
Theorems \ref{thm:Chattering_n4s3_equivalent} and \ref{thm:Chattering_n4s3_equivalent_optimal_y11} provides the optimal solution of problem \eqref{eq:optimalproblem_n4s3}; hence, Theorems \ref{thm:Chattering_n4s3_optimal_allcases}.\ref{thm:Chattering_n4s3_optimal_allcases_n4s3} and \ref{thm:Chattering_n4s3_optimal_solution} are proved.
\end{remark}

Two feasible solutions for problem \eqref{eq:optimalproblem_n4s3} are compared. The first one is the optimal trajectory with chattering, where
\begin{equation}
t_\infty^*\approx5.0938\frac{\abs{x_{0,1}}}{M_0},\,t_\f^*\approx\frac{x_{\f4}-x_{0,4}}{M_3}+1.3452\frac{x_{0,1}^4}{M_0^3M_3}.\end{equation}
The second one is the MIM-trajectory \cite{wang2024time}. Let $\vx$ moves from $\vx_0$ to $\left\{x_3\equiv M_3\right\}$ as fast as possible. Then,
\begin{equation}
\hat{t}_\infty\approx4.3903\frac{\abs{x_{0,1}}}{M_0},\,\hat{t}_\f\approx\frac{x_{\f4}-x_{0,4}}{M_3}+1.3468\frac{x_{0,1}^4}{M_0^3M_3},\end{equation}
where $\vx_{1:3}$ reaches $M_3\ve_3$ at $\hat{t}_\infty$ and $\vx$ reaches $\vx_\f$ at $\hat{t}_\f$. The MIM-trajectory reaches the maximum speed stage for $t_\infty^*-\hat{t}_\infty\approx0.1424\frac{\abs{x_{0,1}}}{M_0}$ earlier than the optimal trajectory. However, the MIM-trajectory arrives at $\vx_\f$ for $\hat{t}_\f-t_\f^*\approx\left(1.5425\times10^{-3}\right)\frac{x_{0,1}^4}{M_0^3M_3}$ later than the optimal trajectory.

\subsection{Discussions}\label{subsec:n4s3_Discussion}

\subsubsection{Comparison Among Fuller's Problem, Robbins' Problem, and Problem \eqref{eq:optimalproblem_n4s3_equivalent}}

Fuller's problem \cite{fuller1963study} and Robbins' problem \cite{robbins1980junction} represent classical optimal control problems with chattering. In the three problems, chattering occurs because a nonsingular arc cannot directly connect a singular arc with $\vzero$ states. Specifically, singular arcs in both Robbins problem and problem \eqref{eq:optimalproblem_n4s3_equivalent} exist due to high-order state constraints.

The three problems all apply the homogeneity to solve the control, where a unique chattering mode is obtained in every problem to solve the corresponding optimal control. The homogeneity structures of Fuller's problem and problem \eqref{eq:optimalproblem_n4s3_equivalent} are similar due to the similar control constraints. It is the homogeneity that contributes to the computation of chattering trajectories in the three problems.

\begin{figure}[!t]
\centering
\includegraphics[width=\columnwidth]{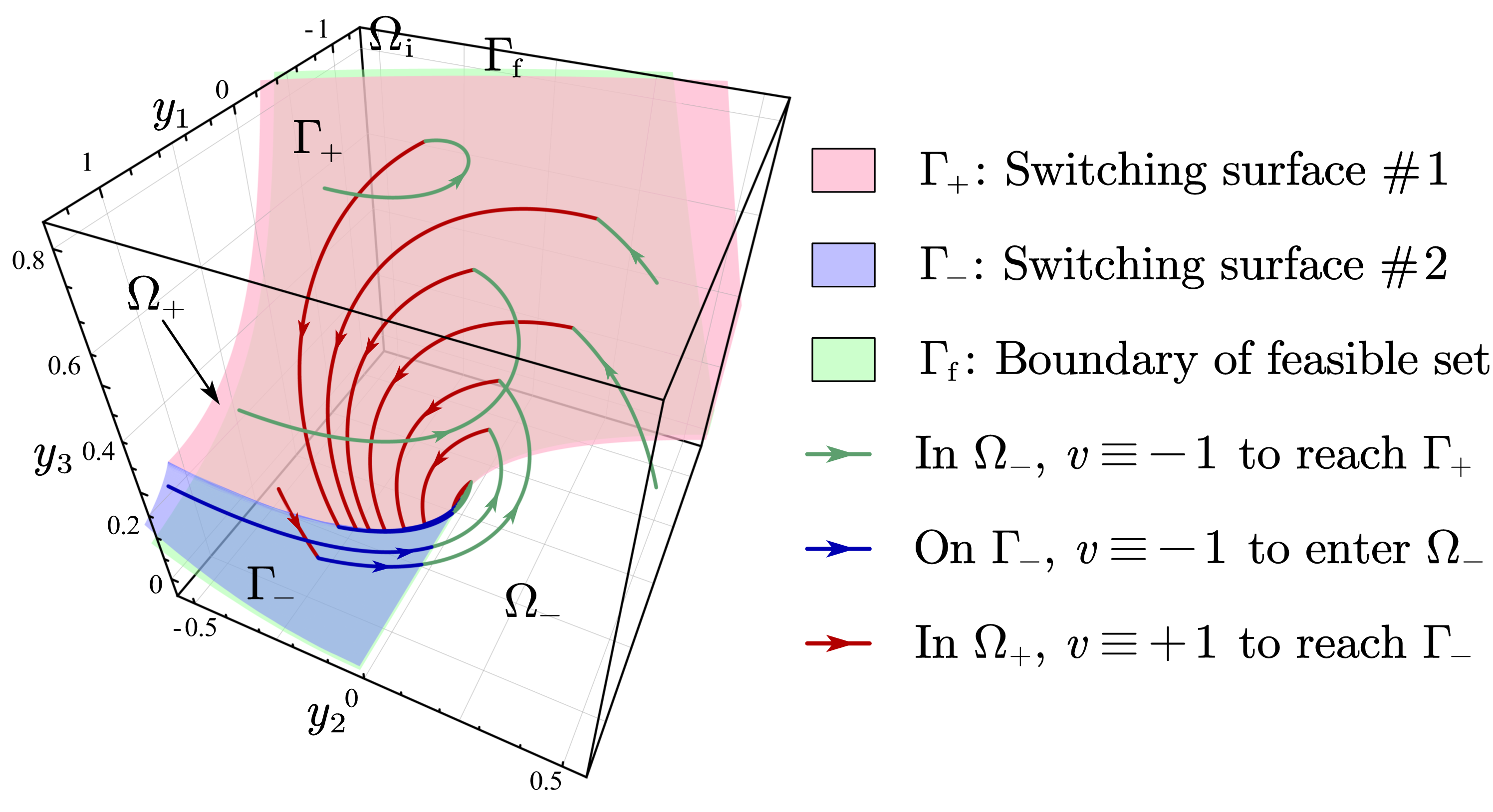}
\caption{Switching surfaces and the state space's structure in problem \eqref{eq:optimalproblem_n4s3_equivalent}.}
\label{fig:switching_surface}
\end{figure}

\subsubsection{Switching Surfaces of Problem \eqref{eq:optimalproblem_n4s3_equivalent}}

Consider problem \eqref{eq:optimalproblem_n4s3_equivalent} with arbitrarily given initial states. Through more refined calculations, the switching surfaces are obtained, as shown in Fig. \ref{fig:switching_surface}. The boundary of feasible set is $\Gamma_\mathrm{f}$, i.e., problem \eqref{eq:optimalproblem_n4s3_equivalent} is infeasible with $\vy\in\Omega_\mathrm{i}$ where $\Omega_\mathrm{i}$ has smaller $y_1$ than $\Gamma_\mathrm{f}$. In $\Omega_-$ which has larger $y_1$ than $\Gamma_+\cup\Gamma_-$, $v\equiv-1$ until $\vy$ reaches the switching surface $\Gamma_+$, and then $\vy$ enters $\Omega_+$ which is between $\Gamma_\mathrm{f}$ and $\Gamma_+\cup\Gamma_-$. In $\Omega_+$, $v\equiv+1$ until $\vy$ reaches $\Gamma_-$, and then $\vy$ moves along $\Gamma_-$ until entering $\Omega_-$. In summary, $\vy$ chatters between $\Gamma_\pm$ and $\Omega_\pm$. The analytical expressions are provided in Appendix \ref{app:switching_surfaces}.

\subsubsection{Optimal Solution of Problem \eqref{eq:optimalproblem} of Order 4}

The conclusions on chattering can be generalized to more general boundary conditions and constraints. In problem \eqref{eq:optimalproblem} of order 4, $\vx$ needs to follow the chattering mode in \eqref{eq:Chattering_n4s3_optimal_u} to enter the constraint arcs $\left\{x_3\equiv\pm M_3\right\}$. Therefore, the switching surface in Fig. \ref{fig:switching_surface} can help to solve problem \eqref{eq:optimalproblem} of order 4.

The optimal trajectory and the MIM-trajectory of a 4th-order position-to-position problem with full state constraints are shown in Figs. \ref{fig:firstshow}(d-e), respectively. In this example, the MIM-trajectory is the same to the S-shaped trajectory. The optimal terminal time is $t_{\f ,\mathrm{opt}}\approx12.6645$, while MIM's terminal time  is $t_{\f ,\mathrm{MIM}}\approx12.6667$, achieving a relative error of $ 1.7\times10^{-4}$. However, the difference between $t_{\f,\mathrm{MIM}}$ and $t_{\f,\mathrm{opt}}$ can be large when the constraints vary, as shown in Fig. \ref{fig:more_optimal}. In fact, $t_{\f,\mathrm{MIM}}-t_{\f,\mathrm{opt}}$ can converge to $\infty$ when $M_0\to0$ with fixed $M_1,M_2,M_3$. Although the error in \eqref{eq:Chattering_n4s3_equivalent_optimal_Jalpha} is small, the loss of S-shaped trajectories can be large and even infinite compared to optimal trajectories with chattering.

\begin{figure}[!t]
\centering
\includegraphics[width=\columnwidth]{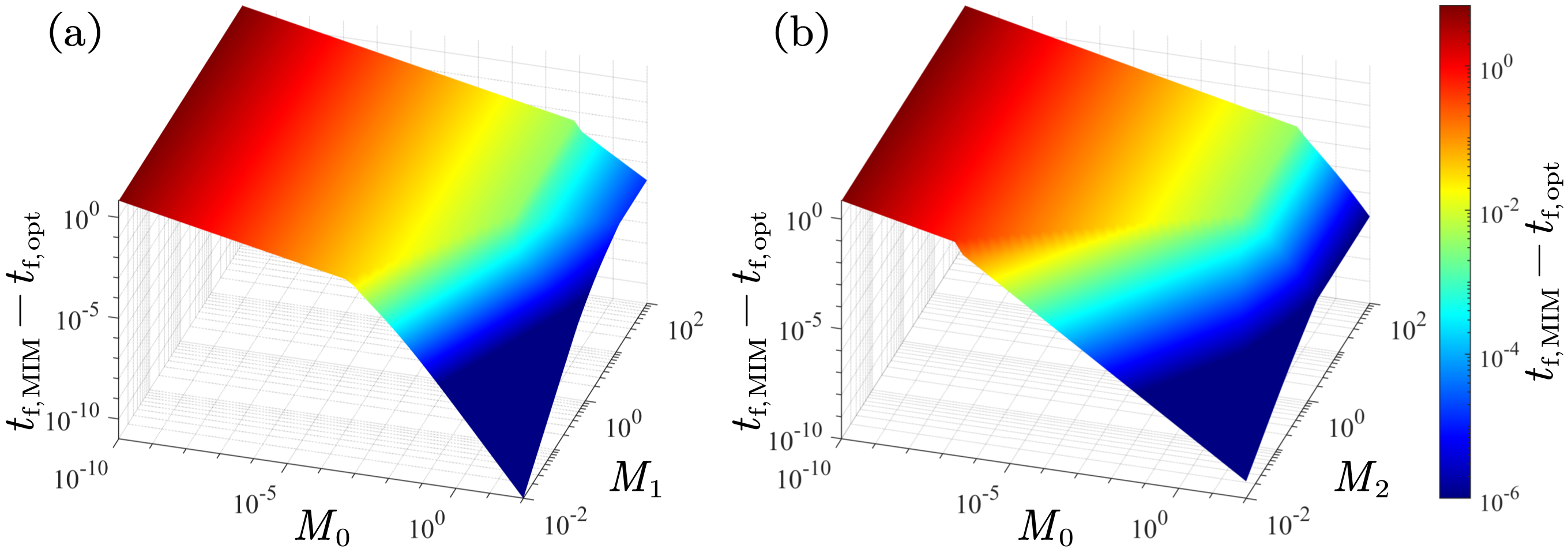}
\caption{$t_{\f,\mathrm{MIM}}-t_{\f,\mathrm{opt}}$ in 4th-order problems. Fix $M_4=\infty$, $M_3=4$, $\vx_0=\vzero$, and $\vx_\f=x_{\f4}\ve_4$ where $x_{\f4}>0$ is large enough. (a) Let $M_0,M_1$ vary, while $M_2=1.5$ is fixed. (b) Let $M_1,M_2$ vary, while $M_1=1$ is fixed.}
\label{fig:more_optimal}
\end{figure}

\subsubsection{Physical Realizability of Chattering Trajectories}

Optimal chattering trajectories can be realized in real-world equipments. For example, if the chattering trajectory in Fig. \ref{fig:firstshow}(d) serves as reference and is interpolated by a finite control frequency, then the interpolated control $u$ has a finite total variation, as shown in Fig. \ref{fig:interpolation_nonchattering}. In the context, the analytical optimal trajectory is interpolated by a control period $0.005$ based on the cubic spline method. The snap $u$ is calculated by the second-order derivative of the acceleration $x_2$ since higher-order derivative is noisy. Therefore, the 4th-order chattering trajectory is physically realizable and can be applied to equipments like ultra-precision wafer stages in practice.

\begin{figure}[!t]
\centering
\includegraphics[width=\columnwidth]{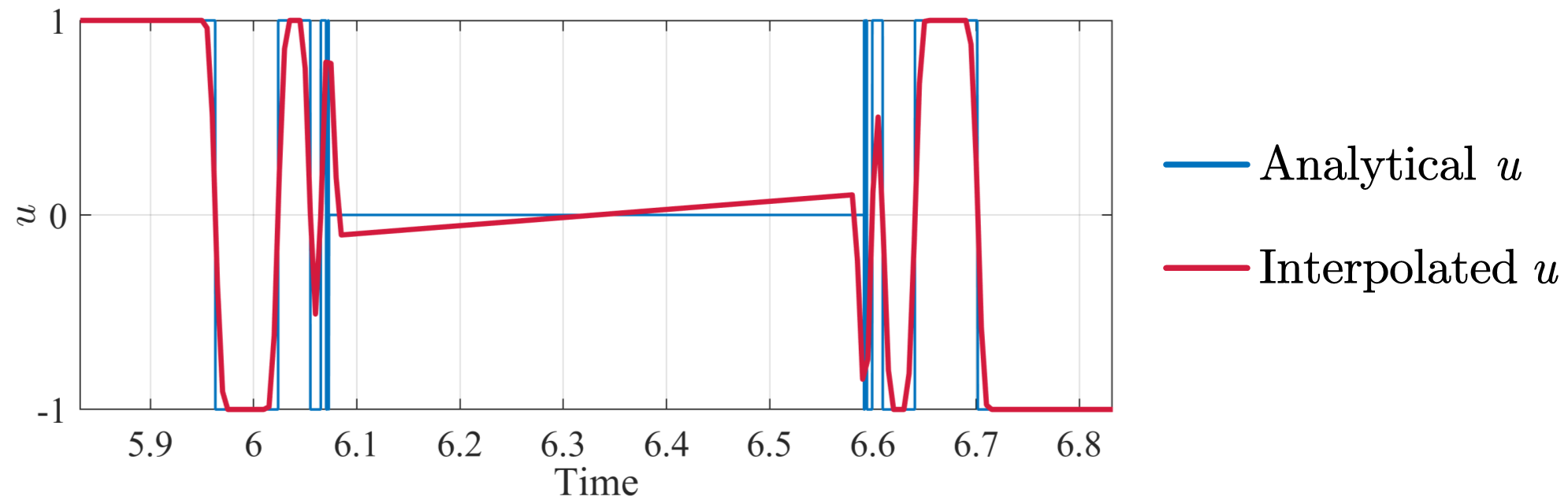}
\caption{The analytical and interpolated $u$ of the trajectory in Fig. \ref{fig:firstshow}(d).}
\label{fig:interpolation_nonchattering}
\end{figure}

\section{Non-Existence of Chattering in Low-Order Problems}\label{sec:ChatteringPhenomena3rdOrder}

As pointed out in Section \ref{sec:Introduction}, no existing works have pointed out whether the chattering phenomenon exists in time-optimal control problem for chain-of-integrator in the form of \eqref{eq:optimalproblem} so far. With a large amount of work on trajectory planning, it is universally accepted that S-shaped trajectories without chattering are optimal in 3rd-order or lower-order problems, i.e., jerk-limited trajectories. This section proves Theorem \ref{thm:Chattering_n4s3_optimal_allcases}.\ref{thm:Chattering_n4s3_optimal_allcases_loworder}, i.e., chattering does not occur in problem \eqref{eq:optimalproblem} when $n\leq4$ except the case where $n=4$ and $\abs{s}=3$; hence, existing literature on 3rd-order or lower-order problems are correct. Without loss of generality, assume that the chattering occurs in a left-side neighborhood of $t_\infty$, i.e., $\left[t_0,t_\infty\right]$ in Theorem \ref{thm:ChatteringPhenomenaBehavior}. Based on Theorem \ref{thm:ChatteringPhenomenaBehavior}, only one constraint $s$ is considered at one time. The set of junction time is $\left\{t_i\right\}_{i=1}^\infty$ which monotonically increases and converges to $t_\infty$.

\subsection{Cases where $n\leq3$}\label{subsec:ChatteringPhenomena3rdOrder_n3}
Theorem \ref{thm:ChatteringPhenomenaBehavior}.\ref{thm:ChatteringPhenomenaBehavior_1sn} implies that the chattering phenomenon does not occur when $n\leq2$ since $1<\abs{s}<n$.

For $n=3$, it holds that $\abs{s}=2$ due to Theorem \ref{thm:ChatteringPhenomenaBehavior}.\ref{thm:ChatteringPhenomenaBehavior_1sn}. Consider the chattering constraint $s=\overline{2}$. Then, $\forall i\in\N^*$, $x_1\left(t_i\right)=0$ and $x_2\left(t_i\right)=M_2$. By Theorem \ref{thm:ChatteringPhenomenaBehavior}.\ref{thm:ChatteringPhenomenaBehavior_sgnxs}, $x_2>0$ during $\left(t_0,t_\infty\right)$; hence, $x_3\left(t_i\right)<x_3\left(t_{i+1}\right)$. Theorem \ref{thm:ChatteringPhenomenaBehavior}.\ref{thm:ChatteringPhenomenaBehavior_sh} implies that during $\left(t_i,t_{i+1}\right)$, $x_2<M_2$. So $\forall t\in\left[t_i,t_i+\frac{x_3\left(t_{i+1}\right)-x_3\left(t_{i}\right)}{M_2}\right]$,
\begin{equation}
\begin{aligned}
&x_3\left(t\right)=x_3\left(t_i\right)+\int_{t_i}^{t}x_2\left(\tau\right)\mathrm{d}\tau\\
<&x_3\left(t_i\right)+M_2\left(t-t_i\right)\leq x_3\left(t_{i+1}\right).
\end{aligned}\end{equation}
Therefore, $t_{i+1}-t_i>\frac{x_3\left(t_{i+1}\right)-x_3\left(t_{i}\right)}{M_2}$. However, a feasible control $\hat{u}\equiv0$ in $\left(t_i,t_i+\frac{x_3\left(t_{i+1}\right)-x_3\left(t_{i}\right)}{M_2}\right)$ successfully drives $\vx$ from $\vx\left(t_i\right)$ to $\vx\left(t_{i+1}\right)$ along $\left\{x_2\equiv M_2\right\}$ with less time, which contradicts BPO. Therefore, chattering phenomena do not occur when $n=3$.

\begin{remark}
Existing classification-based works on jerk-limited trajectory planning \cite{haschke2008line,kroger2011opening,he2020time} have proved that 3rd-order optimal controls switch for finitely many times when $M_3=\infty$, which are consistent with the conclusion in this section. This paper further proves that chattering phenomena do not occur in 3rd-order problems when $M_3<\infty$. As pointed out in Section \ref{sec:ChatteringPhenomena4thOrder}, chattering can occur when $n=4$ and $\abs{s}=3$; hence, classification-based S-shaped trajectories cannot be extended to time-optimal snap-limited trajectories.

\end{remark}

\subsection{Cases where $n=4$ and $\abs{s}\not=3$}\label{subsec:ChatteringPhenomena3rdOrder_n4s2}
Consider problem \eqref{eq:optimalproblem} of order 4. By Theorem \ref{thm:ChatteringPhenomenaBehavior}.\ref{thm:ChatteringPhenomenaBehavior_1sn}, $\abs{s}\in\left\{2,3\right\}$. This section considers the chattering constraint $s=\overline{2}$. Firstly, the recursive expression for junction time $\left\{t_i\right\}_{i=1}^\infty$ is provided in Proposition \ref{prop:Chattering_n4s2_recursive_tau}. Then, Proposition \ref{prop:Chattering_n4s2_recursive_tau_sum_infty} proves that $t_i$ converges to $\infty$, which contradicts the chattering phenomenon.

According to Theorem \ref{thm:ChatteringPhenomenaBehavior}.\ref{thm:ChatteringPhenomenaBehavior_switchingtime}, $\forall i\in\N$, $u$ switches for at most 2 times during $\left(t_i,t_{i+1}\right)$. Assume that
\begin{equation}\label{eq:chattering_n4s2_u}
u\left(t\right)=\begin{dcases}
u_i,&t\in\left(t_{i+1}-\tau_i'',t_{i+1}-\tau_i'\right),\\
-u_i,&t\in\left(t_{i+1}-\tau_i',t_{i+1}-\tau_i\right),\\
u_i,&t\in\left(t_{i+1}-\tau_i,t_{i+1}\right),
\end{dcases}\end{equation}
where $u_i\in\left\{M_0,-M_0\right\}$. Specifically, $0\leq\tau_{i}\leq\tau_{i}'\leq\tau_{i}'' \triangleq t_{i+1}-t_i$.  Note that $x_1\left(t_i\right)=x_1\left(t_{i+1}\right)=0$ and $x_2\left(t_i\right)=x_2\left(t_{i+1}\right)=M_2$.

Since $x_2\leq M_2$, it can be solved that $\forall i\in\N$,
\begin{equation}\label{eq:chattering_n4s2_tau_u}
u_i=-M_0,\,\tau_{i}=\frac{\tau_{i}'}{3}=\frac{\tau_{i}''}{4}=\frac{t_{i+1}-t_i}{4}.\end{equation}
Hence, the control $u$ on $\left(t_i,t_{i+1}\right)$ is determined by $t_{i+1}-t_i=4\tau_i$. Based on the uniqueness of optimal control, i.e., Lemma \ref{lemma:costate}.\ref{lemma:costate_uniqueoptimalcontrol}, the recursive expression for $\left\{\tau_i\right\}_{i=1}^\infty$ is given as follows.

\begin{proposition}\label{prop:Chattering_n4s2_recursive_tau}
If chattering occurs when $n=4$ and $s=\overline{2}$, then $\forall i\in\N^*$, it holds that $f_{\mathrm{c}}\left(\tau_{i+2};\tau_i,\tau_{i+1}\right)=0$, where
\begin{equation}\label{eq:chattering_n4s2_recursive_tau_fc}
\begin{aligned}
f_{\mathrm{c}}\left(\xi;\xi_1,\xi_2\right)\triangleq&\left(\xi_1^2-\xi_2^2\right)\xi^2+\left(\xi_1^3+2\xi_1^2\xi_2-\xi_2^3\right)\xi\\
-&\xi_1^3\xi_2-2\xi_1^2\xi_2^2+\xi_2^4.
\end{aligned}\end{equation}
Furthermore, if $0<\tau_{i+1}<\tau_i$, then $f_{\mathrm{c}}\left(\tau_{i+2};\tau_i,\tau_{i+1}\right)=0$ has a unique positive real root $\tau_{i+2}$, and $0<\tau_{i+2}<\tau_{i+1}<\tau_i$.
\end{proposition}

\begin{proof}
$\forall i\in\N$, consider the trajectory between $t_i$ and $t_{i+4}$. Then, Eq. \eqref{eq:chattering_n4s2_tau_u} and Proposition \ref{prop:system_dynamics} imply that
\begin{equation}\label{eq:chattering_n4s2_recursive_tau}
\begin{dcases}
x_1\left(t_{i+4}\right)=x_1\left(t_{i}\right)=0,\\
x_2\left(t_{i+4}\right)=x_2\left(t_{i}\right)=M_2,\\
x_3\left(t_{i+4}\right)=x_3\left(t_i\right)+4M_2F_1\left(\vtau\right)-2M_0F_2\left(\vtau\right),\\\
x_4\left(t_{i+4}\right)=x_4\left(t_i\right)+8M_2F_1\left(\vtau\right)^2-4M_0F_3\left(\vtau\right).
\end{dcases}\end{equation}

In \eqref{eq:chattering_n4s2_recursive_tau}, denote $\vtau=\left(\tau_j\right)_{j=i}^{i+3}$ and $\vF\left(\vtau\right)=\left(F_{ k}\left(\vtau\right)\right)_{ k=1}^3$, where
\begin{equation}\label{eq:chattering_n4s2_recursive_tau_F}
\begin{dcases}
F_1\left(\vtau\right)=\sum_{j=0}^{3}\tau_{i+j},\,F_2\left(\vtau\right)=\sum_{j=0}^{3}\tau_{i+j}^3,\\
F_3\left(\vtau\right)=\sum_{j=0}^{3}\tau_{i+j}^3\left(\tau_{i+j}+2\sum_{k=j+1}^{3}\tau_{i+k}\right).
\end{dcases}\end{equation}

Assume that $\det\frac{\partial\vF}{\partial\left(\tau_j\right)_{j=i}^{i+2}}\not=0$. According to the implicit function theorem \cite{stein2009real}, $\exists \delta\in\left(0,\min_{i\leq j\leq i+3}\tau_j\right)$, $\hat{\vtau}=\left(\hat{\tau}_{j}\right)_{j=i}^{i+3}\in \Setminus{B_\delta\left(\vtau\right)}{\left\{\vtau\right\}}$, s.t. $\vF\left(\hat{\vtau}\right)=\vF\left(\vtau\right)$. Following \eqref{eq:chattering_n4s2_u} and \eqref{eq:chattering_n4s2_tau_u}, denote $u$ and $\hat{u}$ as the controls induced by $\vtau$ and $\hat{\vtau}$, respectively. According to \eqref{eq:chattering_n4s2_recursive_tau}, both $u$ and $\hat{u}$ can drive $\vx$ from $\vx\left(t_i\right)$ to $\vx\left(t_{i+4}\right)$ during $\left(t_i,t_{i+4}\right)$, with the same motion time $t_{i+4}-t_i=4F_1\left(\vtau\right)=4F_1\left(\hat\vtau\right)$. The above conclusion contradicts Lemma \ref{lemma:costate}.\ref{lemma:costate_uniqueoptimalcontrol}. Therefore, it holds that

\begin{equation}
\det\frac{\partial\vF}{\partial\left(\tau_j\right)_{j=i}^{i+2}}=6\left(\tau_{i+1}+\tau_{i+2}\right)f_{\mathrm{c}}\left(\tau_{i+2};\tau_i,\tau_{i+1}\right)=0,\end{equation}
where $f_{ \mathrm{c}}$ is defined in \eqref{eq:chattering_n4s2_recursive_tau_fc}. Note that $\tau_{i+1},\tau_{i+2}>0$; hence, $f_{ \mathrm{c}}\left(\tau_{i+2};\tau_i,\tau_{i+1}\right)=0$ holds.

If $0<\tau_{i+1}<\tau_i$, then
\begin{equation}
\begin{dcases}
f_{\mathrm{c}}\left(0;\tau_i,\tau_{i+1}\right)=-\tau_i^3\tau_{i+1}-2\tau_i^2\tau_{i+1}^2+\tau_{i+1}^4<0,\\
f_{\mathrm{c}}'\left(0;\tau_i,\tau_{i+1}\right)=\tau_i^3+2\tau_i^2\tau_{i+1}-\tau_{i+1}^3>0,\\
f_{\mathrm{c}}''\left(\tau_{i+2};\tau_i,\tau_{i+1}\right)=2\left(\tau_i^2-\tau_{i+1}^2\right)>0,
\end{dcases}\end{equation}

So $f_{\mathrm{c}}\left(\tau;\tau_i,\tau_{i+1}\right)=0$ has a unique positive real root $\tau_{i+2}$. Note that $0<\tau_{i+1}<\tau_i$ and
\begin{equation}
\begin{dcases}
f_{\mathrm{c}}\left(\tau_i;\tau_i,\tau_{i+1}\right)=\left(\tau_i-\tau_{i+1}\right)\left(2\tau_i-\tau_{i+1}\right)\left(\tau_i+\tau_{i+1}\right)^2,\\
f_{\mathrm{c}}\left(\tau_{i+1};\tau_i,\tau_{i+1}\right)=\tau_{i+1}^2\left(\tau_i-\tau_{i+1}\right)\left(\tau_i+\tau_{i+1}\right).\\
\end{dcases}\end{equation}
Hence,
\begin{equation}
0=f_{\mathrm{c}}\left(\tau_{i+2};\tau_i,\tau_{i+1}\right)<f_{\mathrm{c}}\left(\tau_{i+1};\tau_i,\tau_{i+1}\right)<f_{\mathrm{c}}\left(\tau_i;\tau_i,\tau_{i+1}\right).\end{equation}
Note that $f_{ \mathrm{c}}\left(\tau;\tau_i,\tau_{i+1}\right)$ increases monotonically w.r.t. $\tau$ when $\tau>0$. Therefore, $0<\tau_{i+2}<\tau_{i+1}<\tau_i$ holds.
\end{proof}

Since $\lim_{i\to\infty}\tau_i=\lim_{i\to\infty}\frac{t_{i+1}-t_i}{4}=0$, $\exists i^*\in\N^*$, s.t. $\tau_{i^*}>\tau_{i^*+1}$. Without loss of generality, assume that $\tau_1>\tau_2$; otherwise, by BPO, one can consider the trajectory during $\left[t_{i^*-1},t_{\infty}\right]$. By Proposition \ref{prop:Chattering_n4s2_recursive_tau}, $\left\{\tau_i\right\}_{i=1}^\infty$ decreases strictly monotonically. A chattering phenomenon requires that $\sum_{i=1}^{\infty}\tau_i=\frac{t_\infty-t_1}{4}<\infty$. Hence, $\left\{\tau_i\right\}_{i=1}^\infty$ should exhibit a sufficiently rapid decay rate. However, Proposition \ref{prop:Chattering_n4s2_recursive_tau_sum_infty} points out that $\sum_{i=1}^{\infty}\tau_i=\infty$, leading to a contradiction.

\begin{proposition}\label{prop:Chattering_n4s2_recursive_tau_sum_infty}
Consider $\left\{\tau_i\right\}_{i=1}^\infty\subset\R$ where $0<\tau_2<\tau_1$. Assume that $\forall i\in\N$, $f_{\mathrm{c}}\left(\tau_{i+2};\tau_i,\tau_{i+1}\right)=0$ holds, where $f_{\mathrm{c}}$ is defined in \eqref{eq:chattering_n4s2_recursive_tau_fc}. Then, $\sum_{i=1}^{\infty}\tau_i=\infty$.
\end{proposition}

\begin{proof}
Since $0<\tau_2<\tau_1$, Proposition \ref{prop:Chattering_n4s2_recursive_tau} implies that $\left\{\tau_i\right\}_{i=1}^\infty\subset\R_{++}$ decreases strictly monotonically.

Denote $r_i\triangleq1-\frac{\tau_{i+1}}{\tau_i}\in\left(0,1\right)$. By $\frac{1}{\tau_i^4}f_{\mathrm{c}}\left(\tau_{i+2};\tau_i,\tau_{i+1}\right)=f_{\mathrm{c}}\left(\left(1-r_i\right)\left(1-r_{i+1}\right);1,1-r_i\right)=0$, it holds that
\begin{equation}\label{eq:chattering_n4s2_recursive_tau_sum_infty_fr}
f_\mathrm{r}\left(r_{i+1};r_i\right)\triangleq r_{i+1}^2-a_ir_{i+1}+1=0,\end{equation}
where $a_i=3+\frac{1}{r_i\left(1-r_i\right)}>3$. Note that
\begin{equation}
\begin{dcases}
f_\mathrm{r}\left(0;r_i\right)=1>0,\,f_\mathrm{r}\left(r_{i+1};r_i\right)=0,\\
f_\mathrm{r}\left(r_i;r_i\right)=-\frac{r_i\left(2-r_i\right)^2}{1-r_i}<0,\\
f_\mathrm{r}\left(1;r_i\right)=-1-\frac{1}{r_i\left(1-r_i\right)}<0.
\end{dcases}\end{equation}
Therefore, $0<r_{i+1}<r_i<1$ holds. In other words, $\left\{r_i\right\}_{i=1}^\infty$ is bounded and strictly monotonically decreasing and. Hence, $\lim_{i\to\infty}r_i=r^*\in\left[0,1\right]$ exists. Note that $\forall i\in\N^*$,
\begin{equation}
r_i\left(1-r_i\right)\left(r_{i+1}^2-3r_{i+1}+1\right)-r_{i+1}=0.\end{equation}
Let $i\to\infty$, and $-{r^*}^2\left(2-r^*\right)^2=0$ holds. Then, $r^*=0$ since $r^*\in\left[0,1\right]$. So $a_i\to\infty$ and $\frac{\tau_{i+1}}{\tau_i}\to1^-$ as $i\to\infty$.

Note that $\frac{1}{a_i} \to0$ and $\frac{1}{a_i}=r_i-4r_i^2+\mathcal{O}\left(r_i^3\right)$ as $i\to\infty$. Since $r_{i+1}=\frac{a_i-\sqrt{a_i^2-4}}{2}$, let $i\to\infty$, resulting in
\begin{equation}
r_{i+1}=\frac{1}{a_i}+\mathcal{O}\left(\frac{1}{a_i^3}\right)=r_i-4r_i^2+\mathcal{O}\left(r_i^3\right),\,i\to\infty.\end{equation}
Therefore, it holds that
\begin{equation}
\lim_{i\to\infty}\frac{r_{i+1}}{r_i}=\lim_{i\to\infty}-4r_i+\mathcal{O}\left(r_i^2\right)=1\end{equation}
and
\begin{equation}
\lim_{i\to\infty}\frac{1}{r_{i+1}}-\frac{1}{r_i}=\lim_{i\to\infty}4\frac{r_i}{r_{i+1}}+\mathcal{O}\left(\frac{r_i^2}{r_{i+1}}\right)=4.\end{equation}

Applying Stolz-Ces\`aro theorem \cite{stein2009real}, the limitation
\begin{equation}
\lim_{i\to\infty}ir_i=\lim_{i\to\infty}\frac{i}{\frac{1}{r_i}}=\lim_{i\to\infty}\frac{i+1-i}{\frac{1}{r_{i+1}}-\frac{1}{r_i}}=\frac14\end{equation}
exists. Therefore,
\begin{equation}
\lim_{i\to\infty}i\left(\frac{\tau_i}{\tau_{i+1}}-1\right)=\lim_{i\to\infty}i\left(\frac{1}{1-r_i}-1\right)=\frac{1}{4}<1.\end{equation}
According to Raabe-Duhamel's test \cite{stein2009real}, $\sum_{i=1}^{\infty}\tau_i=\infty$.
\end{proof}

According to Proposition \ref{prop:Chattering_n4s2_recursive_tau_sum_infty}, $\sum_{i=1}^{\infty}\tau_i=\infty$. However, $\sum_{i=1}^{\infty}\tau_i=\sum_{i=1}^{\infty}\frac{t_{i+1}-t_i}{4}=\frac{t_\infty-t_1}{4}<\infty$, which leads to a contradiction. Hence, chattering phenomena do not occur when $n=4$ and $s=\overline{2}$. A similar analysis can be applied to the case where $n=4$ and $s=\underline{2}$. Therefore, chattering phenomena do not occur when $n=4$ and $\abs{s} \not=3$.

\section{Conclusion}
This paper has set out to investigate chattering phenomena in a classical and open problem \eqref{eq:optimalproblem}, i.e., time-optimal control for high-order chain-of-integrator systems with full state constraints. However, there have existed neither proofs on non-existence nor counterexamples to the chattering phenomenon in the classical problem \eqref{eq:optimalproblem} so far, where chattering means that the control switches for infinitely many times over a finite time period in the investigated problem. This paper established a theoretical framework for the chattering phenomenon in problem \eqref{eq:optimalproblem}, pointing out that there exists at most one active state constraint during a chattering period. An upper bound on control's switching times in an unconstrained arc during chattering is determined. The convergence of states and costates at the chattering limit point is analyzed. This paper proved the existence of the chattering phenomenon in 4th-order problems with velocity constraints in the presence of sufficient separation between the initial and terminal positions, where the decay rate in the time domain was precisely calculated as $\alpha^*\approx0.1660687$. The conclusion can be applied to construct 4th-order trajectories with full state constraints in strict time-optimality. To the best of our knowledge, this paper provides the first strictly time-optimal 4th trajectory with full state constraints. Note that position-to-position snap-limited trajectories with full state constraints are universally applied in ultra-precision control in the industry. Furthermore, this paper proves that chattering phenomena do not exist in other cases of order $n\leq4$. In other words, 4th-order problems with velocity constraints represent problems allowing chattering of the lowest order. The above conclusions rectify a longstanding misconception in the industry concerning the time-optimality of S-shaped trajectories with minimal switching times.

\section*{Acknowledgment}
The authors would like to thank Grace Hii for her experitse in academic English. This work was supported by the National Key Research and Development Program of China under Grant 2023YFB4302003.

\ifCLASSOPTIONcaptionsoff
\newpage
\fi

\bibliographystyle{IEEEtran}
\bibliography{IEEEabrv,refs/ref}

\vskip -2\baselineskip plus -1fil

\begin{IEEEbiography}[{\includegraphics[width=1in,keepaspectratio]{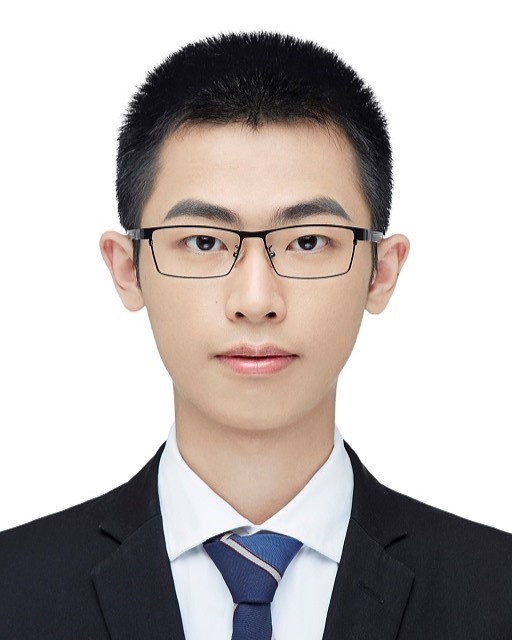}}]{Yunan Wang}
    (S'22) received the B.E. degree in mechanical engineering, in 2022, from the Department of Mechanical Engineering, Tsinghua University, Beijing, China. He is currently working toward the Ph.D. degree in mechanical engineering with the Department of Mechanical Engineering, Tsinghua University, Beijing, China.

    His research interests include optimal control, trajectory planning, toolpath planning, and precision motion control. He was the recipient of the Best Conference Paper Finalist at the 2022 International Conference on Advanced Robotics and Mechatronics, and 2021 Top Grade Scholarship for Undergraduate Students of Tsinghua University.
\end{IEEEbiography}

\vskip -2\baselineskip plus -1fil

\begin{IEEEbiography}[{\includegraphics[width=1in,keepaspectratio]{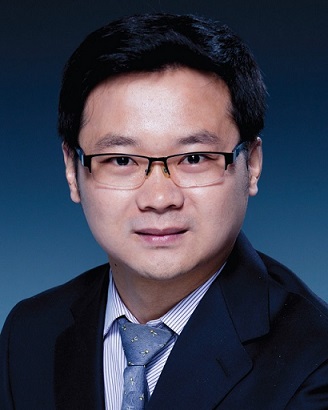}}]{Chuxiong Hu}
    (S'09-M'11-SM'17) received his B.E. and Ph.D. degrees in Mechatronic Control Engineering from Zhejiang University, Hangzhou, China, in 2005 and 2010, respectively. He is currently a professor (tenured) at Department of Mechanical Engineering, Tsinghua University, Beijing, China. From 2007 to 2008, he was a visiting scholar in mechanical engineering with Purdue University, West Lafayette, USA. In 2018, he was a visiting scholar in mechanical engineering with University of California, Berkeley, CA, USA. His research interests include precision motion control, trajectory planning, precision mechatronic systems, intelligent learning and prediction, and robot.

    Prof. Hu is the author of more than 180 journal/conference papers. He was the recipient of the Best Student Paper Finalist at the 2011 American Control Conference, the 2012 Best Mechatronics Paper Award from the ASME Dynamic Systems and Control Division, the 2013 National 100 Excellent Doctoral Dissertations Nomination Award of China, the 2016 Best Paper in Automation Award, the 2018 Best Paper in AI Award from the IEEE International Conference on Information and Automation, the 2022 Best Paper in Theory from the IEEE/ASME International Conference on Mechatronic, Embedded Systems and Applications, the 2023 Best Student Paper from the  International Conference on Control, Mechatronics and Automation, the 2024 Best Paper Award in Advanced Robotics from the International Conference on Advanced Robotics and Mechatronics, and the 2024 Best Paper Award from the IEEE Conference on Industrial Electronics and Applications. He is now an Associate Editor for the IEEE Transactions on Industrial Informatics and a Technical Editor for the IEEE/ASME Transactions on Mechatronics.
\end{IEEEbiography}

\vskip -2\baselineskip plus -1fil

\begin{IEEEbiography}[{\includegraphics[width=1in,keepaspectratio]{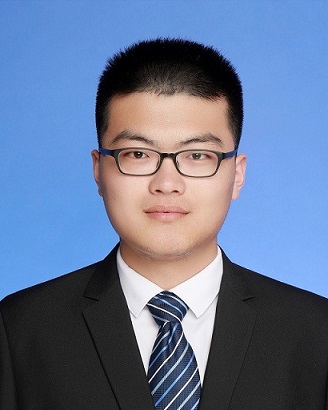}}]{Zeyang Li}
    received the B.E. degree in mechanical engineering from Shanghai Jiao Tong University in 2021, and the M.S. degree in mechanical engineering from Tsinghua University in 2024. He is currently working toward the Ph.D. degree with the Laboratory for Information and Decision Systems, Massachusetts Institute of Technology. His research interests include reinforcement learning and optimal control.
\end{IEEEbiography}

\vskip -2\baselineskip plus -1fil

\begin{IEEEbiography}[{\includegraphics[width=1in,keepaspectratio]{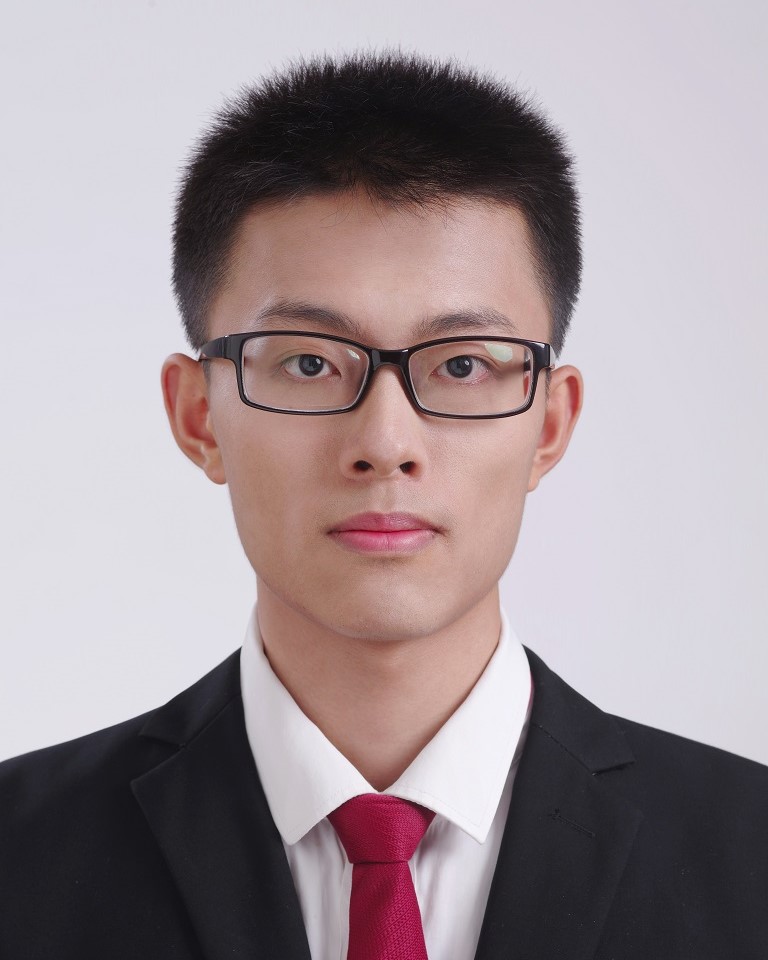}}]{Yujie Lin}
    received the B.S. degree in mathematics and applied mathematics, in 2023, from the Department of Mathematical Sciences, Tsinghua University, Beijing, China. He is currently working toward the Ph.D. degree in mathematics, at Qiuzhen College, Tsinghua University. His research interests include 3\&4-dimensional topology and knot theory.
\end{IEEEbiography}

\vskip -2\baselineskip plus -1fil

\begin{IEEEbiography}[{\includegraphics[width=1in,keepaspectratio]{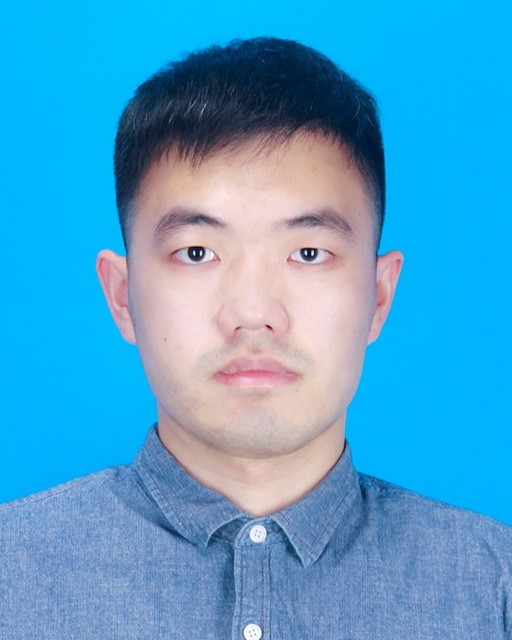}}]{Shize Lin}
    received the B.E. degree in mechanical engineering, in 2020, from the Department of Mechanical Engineering, Tsinghua University, Beijing, China, where he is currently working toward the Ph.D. degree in mechanical engineering. His research interests include robotics, motion planning and precision motion control.
\end{IEEEbiography}

\vskip -2\baselineskip plus -1fil

\begin{IEEEbiography}[{\includegraphics[width=1in,keepaspectratio]{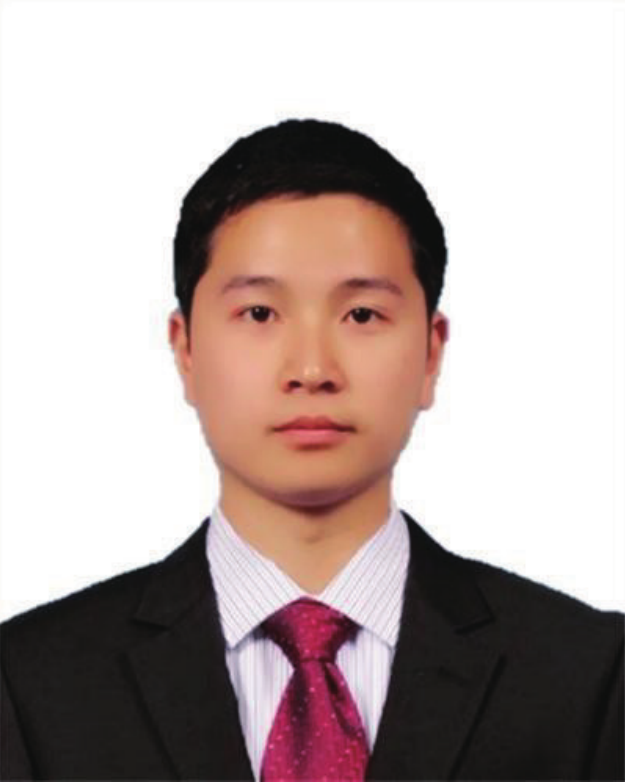}}]{Suqin He}
    received the B.E. degree in mechanical engineering from Department of Mechanical Engineering, Tsinghua University, Beijing, China, in 2016, and the Ph.D. degree in mechanical engineering from the Department of Mechanical Engineering, Tsinghua University, Beijing, China, in 2023.

    His research interests include multi-axis trajectory planning and precision motion control on robotics and CNC machine tools. He was the recipient of the Best Automation Paper from IEEE Internal Conference on Information and Automation in 2016.
\end{IEEEbiography}

\newpage

\onecolumn

\appendices

\section{Proof of Lemma \ref{lemma:costate}}\label{app:proof_lemma_costate}

Lemma \ref{lemma:costate} is a conclusion of our previous work \cite{wang2024time}. The proofs are provided in this appendix for completeness.

\begin{proof}[Proof of Lemma \ref{lemma:costate}.\ref{lemma:lambda1continue}]
Without loss of generality, consider juntion time $t_1$ when $x_1\left(t_1\right)=M_1$. According to \eqref{eq:junction_condition}, $\lambda_1\left(t_1^+\right)\leq\lambda_1\left(t_1^-\right)$ holds.

Assume that $\lambda_1\left(t_1^-\right)>0$. Then, $u\left(t_1^-\right)=-M_0$. In other words, $x_1>M_1$ holds in a left-sided neighborhood of $t_1$, which constraints $\abs{x_1}<M_1$; hence, $\lambda_1\left(t_1^-\right)\leq0$. For the same reason, $\lambda_1\left(t_1^+\right)\leq0$ holds. Therefore, $\lambda_1\left(t_1^+\right)=0=\lambda_1\left(t_1^+\right)$, i.e., $\lambda_1=0$ is continuous at junction time.
\end{proof}

\begin{proof}[Proof of Lemma \ref{lemma:costate}.\ref{lemma:costate_lambdak_polynomial}]
In an unconstrained arc, $\veta\equiv\vzero$ due to \eqref{eq:eta_constraint_zero}. By \eqref{eq:derivative_costate}, $\lambda_k$ is an $\left(n-k\right)$th degree polynomial.

In a constrained arc $\left\{\abs{x_{k}}\equiv M_{k}\right\}$, one has $u\equiv0$. By \eqref{eq:bang_singular_bang_law}, $\lambda_1\equiv0$. Note that $\forall j\not=k$, $\eta_j=0$. Then, $\forall j\leq k$, $\lambda_j=\left(-1\right)^{j-1}\frac{\mathrm{d}^{j-1}}{\mathrm{d}t^{j-1}}\lambda_1\equiv0$ holds. By \eqref{eq:derivative_costate_lambdan_zero}, $k<n$; hence, $\eta_k\equiv\lambda_{k+1}$ holds and $\forall j>k$, $\lambda_{j}$ is an $\left(n-j\right)$th degree polynomial.
\end{proof}

The behavior of $\vlambda$ is given in the above two proofs. If $\vx$ is in a constrained arc $\left\{\abs{x_j}\equiv M_j\right\}$, $j\leq k$, then $\lambda_k$ is zero; otherwise, $\lambda_k$ is polynomial. At a junction time when $\abs{x_k}=M_k$, $\lambda_k$ is allowed discontinous only if $k\not=1$. Based on the above analysis, the bang-bang and singular control laws are proved as follows.

\begin{proof}[Proof of Lemma \ref{lemma:costate}.\ref{lemma:costate_bangbang}]
By \eqref{eq:bang_singular_bang_law}, it only needs to prove the case where $\lambda_1=0$. Assume that $\lambda_1\left(t\right)\equiv0$ for a period. If an unconstrained arc occur in this period, then $\vlambda\equiv\vzero$ and $\veta\equiv\vzero$ hold. By \eqref{eq:hamilton_equiv_0}, it holds that $\lambda_0=0$, which contradicts $\left(\lambda_0,\vlambda\left(t\right)\right)\not=\vzero$. Therefore, $\lambda_1\left(t\right)\equiv0$ for a period only if a constrained arc occurs, i.e., $u\equiv0$. So $u\equiv0$ if $\lambda_1\equiv0$.
\end{proof}

The bang-bang and singular controls imply that $u\left(t\right)\in\left\{0,\pm M_0\right\}$ a.e., which contributes to the following proof on the uniqueness of the optimal control.

\begin{proof}[Proof of Lemma \ref{lemma:costate}.\ref{lemma:costate_uniqueoptimalcontrol}]
Assume that $u_1^*$ and $u_2^*$ are both the optimal control of problem \eqref{eq:optimalproblem} with terminal time $t_\f^*$. Define $u_3^*\left(t\right)=\frac{3}{4}u_1^*\left(t\right)+\frac{1}{4}u_2^*\left(t\right)$, $t\in\left[0,t_\f^*\right]$. Since problem \eqref{eq:optimalproblem} is linear, the control $u_3^*$ is feasible. Furthermore, $u_3^*$ is also optimal since $u_3^*$ also achieves terminal time $t_\f^*$.

According to Lemma \ref{lemma:costate}.\ref{lemma:costate_bangbang}, $\forall 1\leq k\leq 3$, it holds that $ \nu\left(Q_k\right)=0$, where $Q_k\triangleq\left\{t\in\left[0,t_\f^*\right]:u_k^*\left(t\right)\not\in\left\{0,\pm M_0\right\}\right\}$. Denote $P\triangleq\left\{t\in\left[0,t_\f^*\right]:u_1^*\left(t\right)\not=u_2^*\left(t\right)\right\}$. Then, $\forall t\in \Setminus{P}{\left(Q_1\cup Q_2\right)}$, it holds that $u_3^*\left(t\right)\not\in\left\{0,\pm M_0\right\}$; hence, $\Setminus{P}{\left(Q_1\cup Q_2\right)}\subset Q_3$. Therefore,
\begin{equation}
\begin{aligned}
&0\leq\nu\left(P\right)=\nu\left(P\right)-\nu\left(Q_1\right)-\nu\left(Q_2\right)\\
\leq&\nu\left(\Setminus{P}{\left(Q_1\cup Q_2\right)}\right)\leq\nu\left(Q_3\right)=0.
\end{aligned}\end{equation}
Hence, $ \nu\left(P\right)=0$, i.e., $u_1^*\left(t\right)=u_2^*\left(t\right)$ a.e.
\end{proof}

According to the above analysis, the optimal control satisfies $u\equiv M_0$ or $u\equiv-M_0$ in an unconstrained arc, while $u$ is zero in a constrained arc. The following two proofs provide the behavior of $\vx$ at the boundaries of state constraints.

\begin{proof}[Proof of Lemma \ref{lemma:costate}.\ref{lemma:costate_sign}]
Assume that $\vx$ leaves $\left\{x_k\equiv M_k\right\}$ at $t_1$ and enters an unconstrained arc. Then, $\forall j<k$, $x_j\left(t_1\right)=\where{\frac{\mathrm{d}^{k-j}x_k}{\mathrm{d}t^{k-j}}}{t=t_1^-}=0$ holds. Denote $u\left(t_1^+\right)=u_0\in\left\{\pm M_0\right\}$. In other words, $\exists\varepsilon>0$, s.t. $\forall \delta\in\left(0,\varepsilon\right)$, $u\left(t_1+\delta\right)\equiv u_0$. By Proposition \ref{prop:system_dynamics}, $\forall \delta\in\left(0,\varepsilon\right)$, $x_k\left(t_1+\delta\right)=M_k+\frac{u_0}{k!}\delta^k\leq M_k$. Therefore, $u_0=-M_0$ in this case. Similarly, $u\left(t_1^+\right)=M_0$ if $\vx$ leaves $\left\{x_k\equiv-M_k\right\}$ at $t_1$ and enters an unconstrained arc. For the same reason, if $\vx$ enters a constrained arc $\left\{\abs{x_k}\equiv M_k\right\}$ at $t_1$, then $u\left(t_1^-\right)=\left(-1\right)^{k-1}\sgn\left(x_k\left(t_1\right)\right)$.
\end{proof}

\begin{proof}[Proof of Lemma \ref{lemma:costate}.\ref{lemma:costate_junction_secondbracket}]
Assume that $x_k$ is tangent to $M_k$ at $t_1\in\left(0,\tf\right)$. In other words, $x_k\left(t_1\right)=M_k$ and $\exists\varepsilon>0$, s.t. $\forall\delta\in\left(-\varepsilon,0\right)\cup\left(0,\varepsilon\right)$, $x_k\left(t_1+\delta\right)<M_k$. If $k=1$, then $u\left(t_1^\pm\right)=\mp M_0$ holds since $x_1\leq M_1$. Case b in Lemma \ref{lemma:costate}.\ref{lemma:costate_junction_secondbracket} holds.

Assume that $k\geq2$. Since $x_k$ reaches the maximum at $t_1$, it holds that $\dot{x}_k=x_{k-1}=0$. Assume that $x_{k-1}=x_{k-2}=\dots=x_1=0$ holds at $t_1$. Assume that $u\left(t_1^+\right)=u_0\in\left\{\pm M_0\right\}$. According to Proposition \ref{prop:system_dynamics}, $x_k\left(t_1+\delta\right)=M_k+\frac{u_0}{k!}\delta^k$ holds when $\delta>0$ is small enough. Therefore, $u_0=-M_0$. Similarly, $u\left(t_1^-\right)=\left(-1\right)^{k-1}M_0$ holds. Case b occurs.

Otherwise, Case b does not occur. Then, $\exists1\leq h<k$, s.t. $x_{k-1}=x_{k-2}=\dots=x_{k-h+1}=0$ and $x_{k-h}\not=0$ at $t_1$. Note that $x_k\left(t_1+\delta\right)=M_k+\frac{1}{\left(k-h\right)!}x_{k-h}\left(t_1\right)\delta^k+\mathcal{O}\left(\delta^{k+1}\right)$ when $\abs{\delta}$ is small enough. Since $x_k\leq M_k$, $h$ should be even and $x_{k-h}\left(t_1\right)<0$. Case a occurs.
\end{proof}

\section{Switching Surfaces in Problem \eqref{eq:optimalproblem_n4s3_equivalent}}\label{app:switching_surfaces}

This section reasons the switching surfaces in problem \eqref{eq:optimalproblem_n4s3_equivalent}, i.e., Fig. \ref{fig:switching_surface}, which is significant for completely solving problem \eqref{eq:optimalproblem} of order 4 with arbitrarily assigned initial and terminal states.

Consider the following problem:

\begin{IEEEeqnarray}{rl}\label{eq:optimalproblem_n4s3_equivalent_y0}
\min\quad& \widehat{J}=\int_{0}^{\infty}y_3\left(\tau\right)\,\mathrm{d}\tau,\IEEEyesnumber\IEEEyessubnumber*\\
\st\quad&\dot\vy\left(\tau\right)=\vA\vy\left(\tau\right)+\vB v\left(\tau\right),\,\forall\tau\in\left(0,\infty\right),\\
&\vy\left(0\right)=\vy_0=\left(y_{0,k}\right)_{k=1}^3,\\
&y_3\left(\tau\right)\geq 0,\,\forall\tau\in\left(0,\infty\right),\\
&\abs{v\left(\tau\right)}\leq 1,\,\forall\tau\in\left(0,\infty\right).
\end{IEEEeqnarray}

The analytical expressions of surfaces in Fig. \ref{fig:switching_surface} is provided as follows.
\begin{IEEEeqnarray}{rl}\label{eq:switching_surface}
&\Gamma_+=\left\{\begin{aligned}
&\left(a\left(1-t_1+2t_2\right),a^2\left(-t_1-2t_1t_2+\frac12t_1^2+t_2^2\right),a^3\left(t_1^2t_2-t_1t_2^2+\frac12t_1^2-\frac16t_1^3+\frac13t_2^3\right)\right):\\
&a\geq0,\,t_2^*\leq t_2\leq r^*\leq t_1\leq t_1^*,\,t_1\prod_{k=1}^{3}\left(1+\frac{\beta_k^*\tau_1^*}{t_1}\right)=t_2\prod_{k=1}^{3}\left(1+\frac{\beta_k^*\tau_1^*}{t_2}\right)
\end{aligned}\right\},\IEEEyesnumber\IEEEyessubnumber*\\
&\Gamma_-=\left\{\left(a\left(1+t\right),a^2\left(-t-\frac12t^2\right),a^3\left(\frac12t^2+\frac16t^3\right)\right):a\geq0,\,0\leq t\leq r^*\right\},\\
&\Gamma_\mathrm{f}=\left\{\left(a\left(1-t\right),a^2\left(-t+\frac12t^2\right),a^3\left(\frac12t^2-\frac16t^3\right)\right):a\geq0,\,0\leq t\leq3\right\}.
\end{IEEEeqnarray}
Specifically, $\beta_k^*$, $\tau_1^*$ and $\alpha^*$ are provided in \eqref{eq:Chattering_n4s3_equivalent_optimal_solution}. $r^*\approx6.4979$ is the local maximum of $f\left(r\right)=r\prod_{k=1}^{3}\left(1+\frac{\beta_k^*\tau_1^*}{r}\right)$. Note that $f\left(r\right)$ monotonically increases in $\left(0,r^*\right)$ and decreases in $\left(r^*,\infty\right)$. $\left(t_1^*,t_2^*\right)\approx\left(16.8674,2.7289\right)$ is the solution of the following system of equations:
\begin{equation}
\begin{dcases}
t_1^*\prod_{k=1}^{3}\left(1+\frac{\beta_k^*\tau_1^*}{t_1^*}\right)=t_2^*\prod_{k=1}^{3}\left(1+\frac{\beta_k^*\tau_1^*}{t_2^*}\right),\\
{t_1^*}^2t_2^*-t_1^*{t_2^*}^2+\frac12{t_1^*}^2-\frac16{t_1^*}^3+\frac13{t_2^*}^3=0.
\end{dcases}\end{equation}

The homogeneity shown in Proposition \ref{prop:n4s3_equivalent_recursive} can be observed from \eqref{eq:Chattering_n4s3_equivalent_optimal_solution}.

\subsection{Conditions for Chattering in Problem \eqref{eq:optimalproblem_n4s3_equivalent_y0}}

Theorem \ref{thm:Chattering_n4s3_equivalent_optimal_y11} provides the optimal trajectory with $\vy_0\in\left\{a\ve_1:a\geq0\right\}$. This section reasons a sufficient and necessary condition for chattering w.r.t. $\vy_0$.

\begin{proposition}[Conditions for Chattering]
Assume that problem \eqref{eq:optimalproblem_n4s3_equivalent_y0} is feasible. Let
\begin{equation}
\tau_0=\begin{dcases}
0,&\text{if }y_3\left(0\right)=y_2\left(0\right)=0,\,y_1\left(0\right)\geq0,\\
\arg\min\left\{\tau>0:y_3\left(\tau\right)=0\right\}\in\R_{++},&\text{otherwise}.\\
\end{dcases}\end{equation}
Then, $\tau_0<\infty$. Furthermore, chattering does not occur if and only if $\vy_0\in\left\{\left(t,-\frac12t^2,\frac16t^3\right):t\geq0\right\}$.
\end{proposition}

\begin{proof}
Assume that $\forall \tau\geq0$, $y_3\left(\tau\right)>0$, i.e., the whole trajectory is unconstrained. Specifically, $p_2\left(t\right)$ is a cubic polynomial. Evidently, $p_2\equiv0$ implies $\vp\equiv\vzero$ which contradicts $\left(p_0,\vp\left(t\right)\right)\not=\vzero$ and $\widehat\hamilton\equiv0$. Hence, $\exists t_0\geq0$, s.t. $\forall \tau\geq \tau_0$, $p_2\left(\tau\right)$ preserves a constant sign. (a) If $\forall \tau\geq t_0$, $p_2\left(t\right)>0$, then $\forall \tau\geq \tau_0$, $v\left(\tau\right)\equiv-1$. In this case, $y_3\to-\infty$ as $\tau\to\infty$, which contradicts $y_3>0$. (b) If $\forall \tau\geq t_0$, $p_2\left(t\right)<0$, then $\forall \tau\geq \tau_0$, $v\left(\tau\right)\equiv+1$. In this case, $y_3\to+\infty$ as $\tau\to\infty$, which contradicts the optimality. In summary, the constraint $y_3=0$ is active at some time.

Therefore, $\tau_0<\infty$. Evidently, $y_3\left(\tau_0\right)=y_2\left(\tau_0\right)=0$ and $y_1\left(\tau_0\right)\geq0$. According to BPO and Proposition \ref{prop:n4s3_equivalent_recursive}, the trajectory for $\tau\geq\tau_0$ should chatter if and only if $y_1\left(\tau_0\right)>0$.

Assume that $y_1\left(\tau_0\right)=0$, i.e., chattering does not occur. Then, $\exists\delta>0$, s.t.
\begin{equation}
\begin{dcases}
v\left(\tau\right)=-1,\,\vy\left(\tau\right)=\left(\tau_0-\tau,-\frac12\left(\tau_0-\tau\right)^2,\frac16\left(\tau_0-\tau\right)^3\right),\,p_1\left(\tau\right)>0,&\tau\in\left(\tau_0-\delta,\tau_0\right),\\
v\left(\tau\right)=+1,\,\vy\left(\tau\right)=\vzero,\,\vp\left(\tau\right)=\vzero,&\tau\in\left(\tau_0,\tau_0+\delta\right).
\end{dcases}\end{equation}
In particular, $p_1\left(\tau\right)=\frac{p_0}{6}\left(\tau-\tau_0\right)^2\left(\tau-\tau_0'\right)$ where $\tau_0'\geq\tau_0$. Hence, $\forall\tau\in\left(0,\tau_0\right)$, $p_1\left(\tau\right)\geq0$ and $v\left(\tau\right)\equiv+1$. Therefore, chatter does not occur only if $\vy_0\in\left\{\left(t,-\frac12t^2,\frac16t^3\right):t\geq0\right\}$.

For the case where $\vy_0\in\left\{\left(t,-\frac12t^2,\frac16t^3\right):t>0\right\}$, assume that $y_1\left(\tau_0\right)>0$. Without loss of generality, consider $\vy_0=\left(1,-\frac12,\frac16\right)$ and $y_1\left(\tau_0\right)=a>0$. If $v$ switches for two times during $\left(0,\tau_0\right)$, one can reason that the following system of equations holds.
\begin{equation}
\begin{dcases}
a=\left(1-\tau_0\right)+2\tau_0'-2\tau_0'',\\
0=\frac12\left(1-\tau_0\right)^2+\tau_0'^2-\tau_0''^2,\\
0=\frac16\left(1-\tau_0\right)^3+\frac13\tau_0'^3-\frac13\tau_0''^3,\\
0=\prod_{k=1}^{3}\left(\tau_0'+\beta_k^*\tau_1^*a\right)-3\mu\tau_0'^2,\\
0=\prod_{k=1}^{3}\left(\tau_0''+\beta_k^*\tau_1^*a\right)-3\mu\tau_0''^2.
\end{dcases}\end{equation}
The above system of equations has a unique solution $a=0$ which contradicts $a>0$. Therefore, $v$ switches for at most one time during $\left(0,\tau_0\right)$. In this case, the following system of equations holds.
\begin{equation}
\begin{dcases}
a=\left(1-\tau_0\right)+2\tau_0-2\tau_0',\\
0=-\frac12\left(1-\tau_0\right)^2+\tau_0^2-\tau_0'^2,\\
0=\frac16\left(1-\tau_0\right)^3+\frac13\tau_0^3-\frac13\tau_0'^3.
\end{dcases}\end{equation}
The above system of equations has a unique solution $a=0$ which contradicts $a=0$. Therefore, $a=y_1\left(\tau_0\right)=0$; hence, chattering does not occur if $\vy_0\in\left\{\left(t,-\frac12t^2,\frac16t^3\right):t\geq0\right\}$.
\end{proof}

\begin{remark}
In fact, one can prove that $\left\{\left(t,-\frac12t^2,\frac16t^3\right):t\geq0\right\}\subset\Omega_-$.
\end{remark}

\subsection{Cases where $v$ Switches for Two Times before $\tau_0$}
Consider the case where $v$ switches for two times before $\tau_0$. Assume that $\vy_0\left(\tau_0\right)=\left(a,0,0\right)$ where $a>0$. The two switching time points are $\tau_0'$ and $\tau_0''$ where $0<\tau_0'<\tau_0''<\tau_0$. Let $\vy'=\left(y_k'\right)_{k=1}^3=\vy\left(\tau_0'\right)$ and $\vy''=\left(y_k''\right)_{k=1}^3=\vy\left(\tau_0''\right)$ are the switching points. Then, one has
\begin{equation}
\begin{dcases}
a=y_1'+\left(\tau_0-\tau_0'\right)-2\left(\tau_0-\tau_0''\right),\\
0=y_2'+y_1'\left(\tau_0-\tau_0'\right)+\frac12\left(\tau_0-\tau_0'\right)^2-\left(\tau_0-\tau_0''\right)^2,\\
0=y_3'+y_2'\left(\tau_0-\tau_0'\right)+\frac12y_1'\left(\tau_0-\tau_0'\right)^2+\frac16\left(\tau_0-\tau_0'\right)^3-\frac13\left(\tau_0-\tau_0''\right)^3,\\
a=y_1''-\left(\tau_0-\tau_0''\right),\\
0=y_2''+y_1''\left(\tau_0-\tau_0''\right)-\frac12\left(\tau_0-\tau_0''\right)^2,\\
0=y_3''+y_2''\left(\tau_0-\tau_0''\right)-\frac12y_1''\left(\tau_0-\tau_0''\right)^2+\frac16\left(\tau_0-\tau_0''\right)^3.
\end{dcases}\end{equation}
Consider the costate vector. According to Theorem \ref{thm:Chattering_n4s3_equivalent_optimal_y11}, one has $\forall \tau\geq\left[\tau_0,\tau_1\right]$,
\begin{equation}
p_1\left(\tau\right)=p_1^+\left(\tau\right)\triangleq-\frac{p_0}{6}\prod_{k=1}^{3}\left(\tau-\left(1-\beta_k^*\right)\tau_0-\beta_k^*\tau_1\right),\end{equation}
where $\tau_1$ is the next juntion time. By the junction condition \eqref{eq:Chattering_n4s3_equivalent_junction}, $\exists\mu\geq0$, s.t. $\forall \tau\geq\left[0,\tau_0\right]$, $p_1\left(\tau\right)=p_1^-\left(\tau\right)\triangleq p_1^+\left(\tau\right)-\frac\mu2\left(\tau-\tau_0\right)^2$. Since $\tau_0'$ and $\tau_0''$ are the switching time points, one has
\begin{equation}\label{eq:switching_surface_2times_p_condition}
p_1^-\left(\tau_0'\right)=p_1^-\left(\tau_0''\right)=0.\end{equation}

Let $t_1=\frac{\tau_0-\tau_0''}{a}$ and $t_2=\frac{\tau_0-\tau_0'}{a}$. Evidently, $0\leq t_2<t_1$. \eqref{eq:switching_surface_2times_p_condition} implies that
\begin{equation}
0< t_2\leq r^*\leq t_1,\,t_1\prod_{k=1}^{3}\left(1+\frac{\beta_k^*\tau_1^*}{t_1}\right)=t_2\prod_{k=1}^{3}\left(1+\frac{\beta_k^*\tau_1^*}{t_2}\right).\end{equation}
Furthermore, $y_3'\geq0$ implies that $t_2^*\leq t_2\leq r^*\leq t_1\leq t_1^*$. Therefore, the switching point $\vy'\in\Gamma_+$. In this case, $\vy_0\in\Omega_-$; hence, one can verify that $\Gamma_+=\left\{\vy\left(\tau_0'\right):\vy_0\in\Omega_-\right\}$.

For the same reason, one can prove that
\begin{equation}
\widehat{\Gamma}_-=\left\{\vy\left(\tau_0''\right):\vy_0\in\Omega_-\right\}=\left\{\left(a\left(1+t\right),a^2\left(-t-\frac12t^2\right),a^3\left(\frac12t^2+\frac16t^3\right)\right):a\geq0,\,t_2^*\leq t\leq r^*\right\}\subset\Gamma_-\end{equation}
is the switching surface for $\vy''$ in this case.

\subsection{Cases where $v$ Switches for One Time before $\tau_0$}
Consider the case where $v$ switches for one time before $\tau_0$. Assume that $\vy_0\left(\tau_0\right)=\left(a,0,0\right)$ where $a>0$. The switching time point is $\tau_0''$ where $0<\tau_0''<\tau_0$. Let $\vy''=\left(y_k''\right)_{k=1}^3=\vy\left(\tau_0''\right)$ is the switching point. Then, one has
\begin{equation}
\begin{dcases}
a=y_1''-\left(\tau_0-\tau_0''\right),\\
0=y_2''+y_1''\left(\tau_0-\tau_0''\right)-\frac12\left(\tau_0-\tau_0''\right)^2,\\
0=y_3''+y_2''\left(\tau_0-\tau_0''\right)-\frac12y_1''\left(\tau_0-\tau_0''\right)^2+\frac16\left(\tau_0-\tau_0''\right)^3,\\
p_1^-\left(\tau_0''\right)=0.
\end{dcases}\end{equation}
In particular, $\tau_0''$ is the unique root of $p_1^-$ on $\left[0,\tau_0\right]$. Let $t=\frac{\tau_0-\tau_0''}{a}$. One can prove that $0\leq t\leq r^*$. Furthermore, $\Gamma_-=\left\{\vy\left(\tau_0''\right):\vy_0\in\Omega_+\right\}$ is the switching surface in this surface.

\end{document}